\documentclass[11pt]{amsart}
\usepackage[T1]{fontenc}
\usepackage[ansinew]{inputenc}
\usepackage{amssymb,amsopn,amsmath,amsthm,authblk,eucal,hyperref,lipsum, mathrsfs,array,graphicx,xcolor,bbm,enumerate,wasysym}
\usepackage{verbatim}
\usepackage[all]{xy}
\usepackage[displaymath, mathlines]{lineno}
\usepackage{fullpage}

\newtheorem{thm}{Theorem}[section]
\newtheorem{lemma}[thm]{Lemma}
\newtheorem{rem}[thm]{Remark}
\newtheorem{defn}[thm]{Definition}
\newtheorem{claim}[thm]{Claim}
\newtheorem{prop}[thm]{Proposition}
\newtheorem{cor}[thm]{Corollary}

\newcommand{\Q}{\mathbb Q}

\newcommand{\R}{\mathbb R}
\newcommand{\C}{\mathbb C}
\renewcommand{\H}{\mathbb H}
\newcommand{\Hh}{\mathcal H}
\newcommand{\N}{\mathbb N}

\newcommand{\D}{\mathbb D}
\newcommand{\E}{\mathbb E}

\newcommand{\Lo}{\text{Loop}}
\newcommand{\Aa}{\mathbb A}
\newcommand{\B}{\mathbb B}

\newcommand{\A}[1]{\Aa_{-#1,-#1+2\lambda}}

\newcommand{\F}{\mathcal F}

\newcommand{\eps}{\epsilon}

\newcommand{\CR}{\operatorname{CR}}

\DeclareMathOperator{\SLE}{SLE}

\makeatletter
\g@addto@macro{\endabstract}{\@setabstract}
\newcommand{\authorfootnotes}{\renewcommand\thefootnote{\@fnsymbol\c@footnote}}%
\makeatother

\let \epsilon \varepsilon

\numberwithin{equation}{section}

\begin{document}
	\title{Two-valued local sets of the 2D continuum Gaussian free field: connectivity, labels, and induced metrics}
\maketitle
\begin{center}
	\vspace{-0.025\textheight}
	\normalsize
	\authorfootnotes
	Juhan Aru \footnote{Department of Mathematics, ETH Zürich, Rämistr. 101, 8092 Zürich, Switzerland}, Avelio Sepúlveda\footnote{Univ Lyon, Université Claude Bernard Lyon 1, CNRS UMR 5208, Institut Camille Jordan, 69622 Villeurbanne, France }
\end{center}
\setcounter{footnote}{0}
		 	\vspace{-0.015\textheight}
	\begin{abstract}
	 	We study two-valued local sets, $\Aa_{-a,b}$, of the two-dimensional continuum Gaussian free field (GFF) with zero boundary condition in simply connected domains. Intuitively, $\Aa_{-a,b}$ is the (random) set of points connected to the boundary by a path on which the values of the GFF remain in $[-a,b]$. For specific choices of the parameters $a, b$ the two-valued sets have the law of the CLE$_4$ carpet, the law of the union of level lines between all pairs of boundary points, or, conjecturally, the law of the interfaces of the scaling limit of XOR-Ising model.
	 	
	 	Two-valued sets are the closure of the union of countably many SLE$_4$ type of loops, where each loop comes with a label equal to either $-a$ or $b$. One of the main results of this paper describes the connectivity properties of these loops. Roughly, we show that all the loops are disjoint if $a+b \geq
	 	4\lambda$, and that their intersection graph is connected if $a + b <
	 	4\lambda$. This also allows us to study the labels (the heights) of the loops. We prove that the labels of the loops are a function of the set $\Aa_{-a,b}$ if and only if $a\neq b$ and $2\lambda \leq a+b < 4\lambda$ and that the labels are independent given the set if and only if $a = b = 2\lambda$. We also show that the threshold for the level-set percolation in the 2D continuum GFF is $-2\lambda$.
	 	
	 	Finally, we discuss the coupling of the labelled CLE$_4$ with the GFF. We characterise this coupling as a specific local set coupling, and show how to approximate these local sets. We further see how in these approximations the labels naturally encode distances to the boundary.
	\end{abstract}

	\section{Introduction}
	Two-valued local sets (TVS) of the two-dimensional Gaussian free field (GFF), denoted $\Aa_{-a,b}$, were introduced in \cite{ASW}. They are the two-dimensional analogue of the exit times from an interval $[-a,b]$ by a standard Brownian motion. Intuitively, they correspond to the set of points of the domain that can be connected to the boundary via a path on which the GFF takes values only in $[-a,b]$. TVS are tightly linked to the study of the 2D GFF: for $\lambda=\sqrt{\pi/8}$ the set $\Aa_{-2\lambda, 2\lambda}$ describes the outer boundaries of the outermost sign clusters of the 2D GFF \cite{LupuCLE, QW}, and Qian and Werner used $\Aa_{-\lambda,\lambda}$ to couple the free boundary and zero boundary GFFs \cite{QW17}. TVS also appear naturally in other statistical physics models: for example it is known that CLE$_4$ has the law of $\Aa_{-2\lambda, 2\lambda}$ \cite{MS,ASW} and moreover it is conjectured that $\Aa_{-a,b}$ with $a+b = 2(1+\sqrt{2})\lambda$ should be the scaling limit of interfaces corresponding to the XOR-Ising model \cite{Wilson}.

	As the 2D GFF is not defined pointwise, one has to give meaning to TVS. This was done in \cite{ASW,ALS1},  using the concept of local sets.  This concept appeared first in the study of Markov random field in the 70s and 80s (see in particular \cite{Roz}) and it was reintroduced in \cite{SchSh2} in the context of the coupling between the GFF and SLE$_4$. Local sets are the natural generalisation of stopping times for multi-dimensional time. More precisely, take $(\Gamma,A)$ a coupling between $\Gamma$ a (zero-boundary) GFF in a domain $D$ and $A\subseteq \overline D$ a closed set. We say that $A$ is a local set of $\Gamma$ if, conditionally on $A$, the law of $\Gamma$ restricted to $D \backslash A$ is equal the sum of $\Gamma^A$, a GFF in $D \backslash A$, and a (conditionally) independent a random harmonic function $h_A$ defined in $D \backslash A$. This harmonic function can be interpreted as the harmonic extension to $D \backslash A$ of the values of the GFF on $\partial A$.
						
	In \cite{ASW} two-valued sets were defined via their expected properties and their construction was provided using SLE$_4$ type of level lines. More precisely, it was shown that if $a,b> 0 $, $a+b\geq 2\lambda$ and $\Gamma$ is a GFF in a simply connected domain $D$, then there exists a unique local set $\Aa_{-a,b}$ that satisfies
	\begin{enumerate}
		\item $h_{\Aa_{-a,b}}$ is constant in every connected component of $D \backslash \Aa_{-a,b}$ with values in $\{-a,b\}$.
		\item The set $\Aa_{-a,b}$ is a thin local set. In the current setting it means that for any smooth test function $f$, the random variable $(\Gamma,f)$ is almost surely equal to  $ (\Gamma^{\Aa_{-a,b}}, f)+\int_{D \backslash {\Aa_{-a,b}} }  h_{\Aa_{-a,b}}(x) f(x) dx$.
		\item $\Aa_{-a,b}\cup \partial D$ has a finite number of connected components.
	\end{enumerate}
	
	In the same article several basic properties were proved, all of which are intuitive from the heuristic description as a set of points connected to the boundary via a path on which GFF takes values in $[-a,b]$. For example, $(\Aa_{-a,b}, \Gamma^{\Aa_{-a,b}}, h_{\Aa_{-a,b}})$ is measurable function of the GFF $\Gamma$. Also, the sets $\Aa_{-a,b}$ are monotone with respect to  $a$ and $b$, in other words, if $[-a,b]\subseteq [-a',b']$, then $\Aa_{-a,b}\subseteq \Aa_{-a',b'}$. Furthermore, $\Aa_{-a,b}\cup \partial D$ is connected. Thus, all connected components of $D \backslash \Aa_{-a,b}$ are simply-connected. A notable difference to the continuous setting also appeared in \cite{ASW}:when $a+b < 2\lambda$, there are no local sets $A$ satisfying (1), (2) and (3): this roughly just says that the GFF is so rough that you cannot move away from the boundary without making a fluctuation of at least the size of $2\lambda$. 
	
	In the current article, we are interested in the study of the geometric properties of the set $\Aa_{-a,b}$ and its complement. We describe in more detail its size, its connectivity, but also answer the question whether given the set, one can recover the heights. 
	
	Before describing our results in more detail, let us also mention the recent work of \cite{GP}. Whereas some of the connectivity properties we prove are directly related to versions of SLE$_4$ processes, in \cite{GP} the authors study connectivity properties of the loops generated by a SLE$_\kappa$ process with $\kappa \in (4,8)$, using very different techniques from ours.
	
	\subsection{An overview of results} Let us now state some of the results that we shall derive, several other smaller results can be found in the main text. Throughout the present section, $D$ can be thought of as the unit disk.
	
	The bulk of the paper deals describes the ``connectivity properties'' of the loops of $\Aa_{-a,b}$. Here, by \textit{loops}, we mean the boundaries of the connected  components of $D\backslash \Aa_{-a,b}$. Indeed, it follows from the construction in \cite{ASW} that these boundaries are Jordan curves. Loops of $\Aa_{-a,b}$ are always locally SLE$_4$ curves and thus of Hausdorff dimension $3/2$. 
	
	We now loosely state our main result, see Theorem \ref{mainthm} for a rigorous statement and Figure \ref{SMT} for an illustration:
	\begin{itemize}
		\item If $a + b = 2\lambda$, one can pass from each loop to any other via a finite number of loops such that every two consecutive loops share a boundary segment.
		\item If $2\lambda < a+b < 4\lambda$, one can pass from each loop to any other one via a finite number of loops such that every two consecutive loops intersect.
		\item If $a+b \geq 4\lambda$, all loops are pairwise disjoint.
	\end{itemize}
	
	\begin{figure}[h!]\label{SMT}
		\includegraphics[scale=0.5]{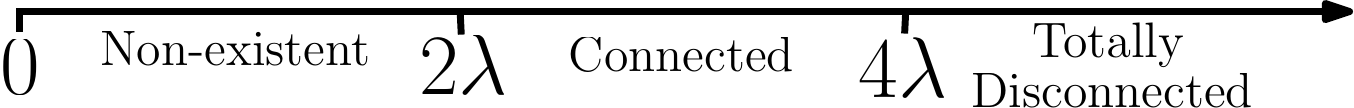}
		\caption{Connectivity properties of the loops of $\Aa_{-a,b}$ given $a+b$. We also study the behaviour at the critical points which correspond to that at its right.}
	\end{figure}
	
	The central idea in the proof of the first two cases is to provide a particular construction of $\Aa_{-a,b}$ where it is easy to see that the required property holds. This is a recurrent technique in our proofs. The demonstration of the third statement uses the fact that the loops of $\Aa_{-2\lambda,2\lambda}$ have the law of a CLE$_4$ and that CLE$_4$ loops are non-intersecting.
	
	\begin{figure}[h!]
		\centering
	\includegraphics[scale=0.06]{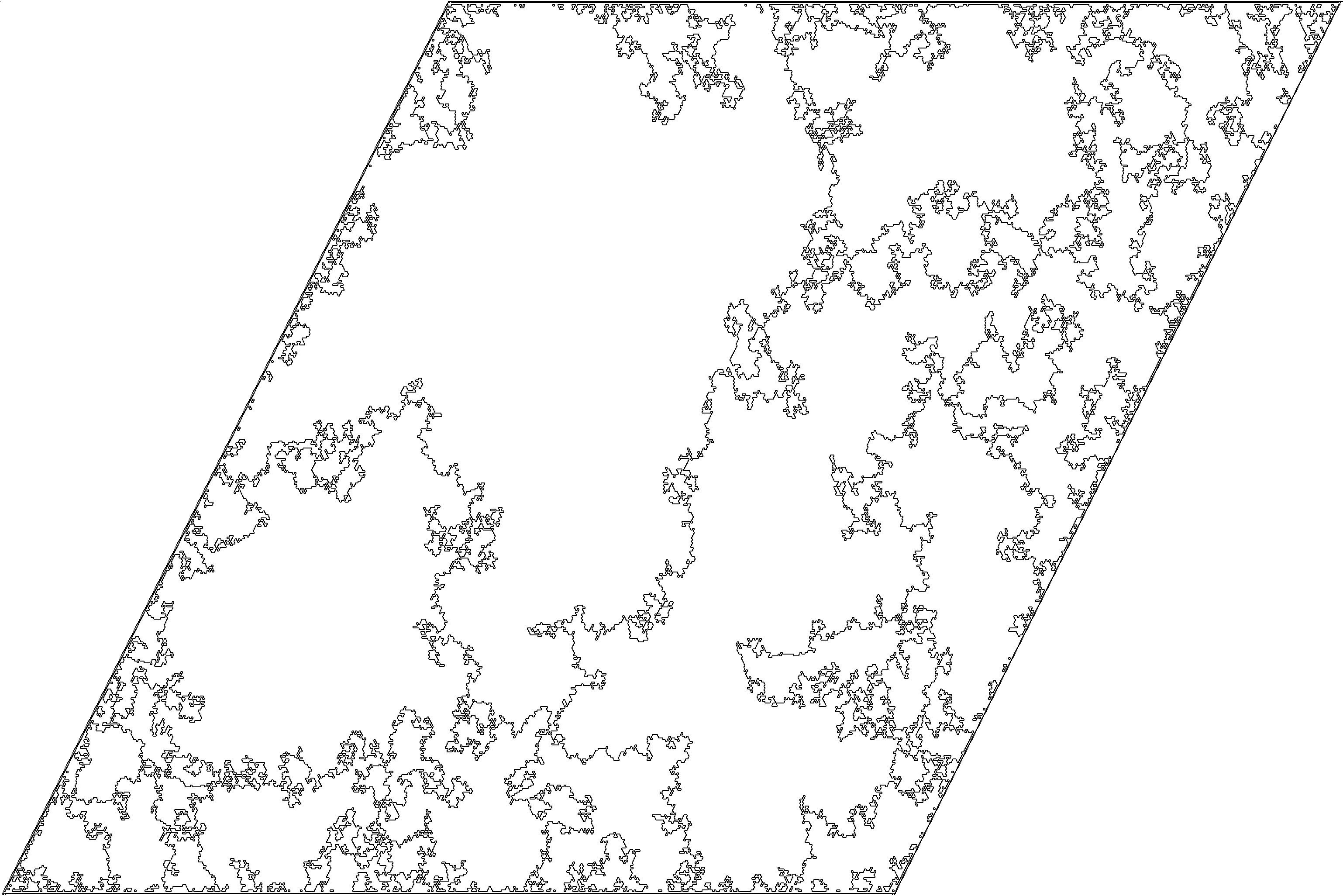}
		\includegraphics[scale=0.13]{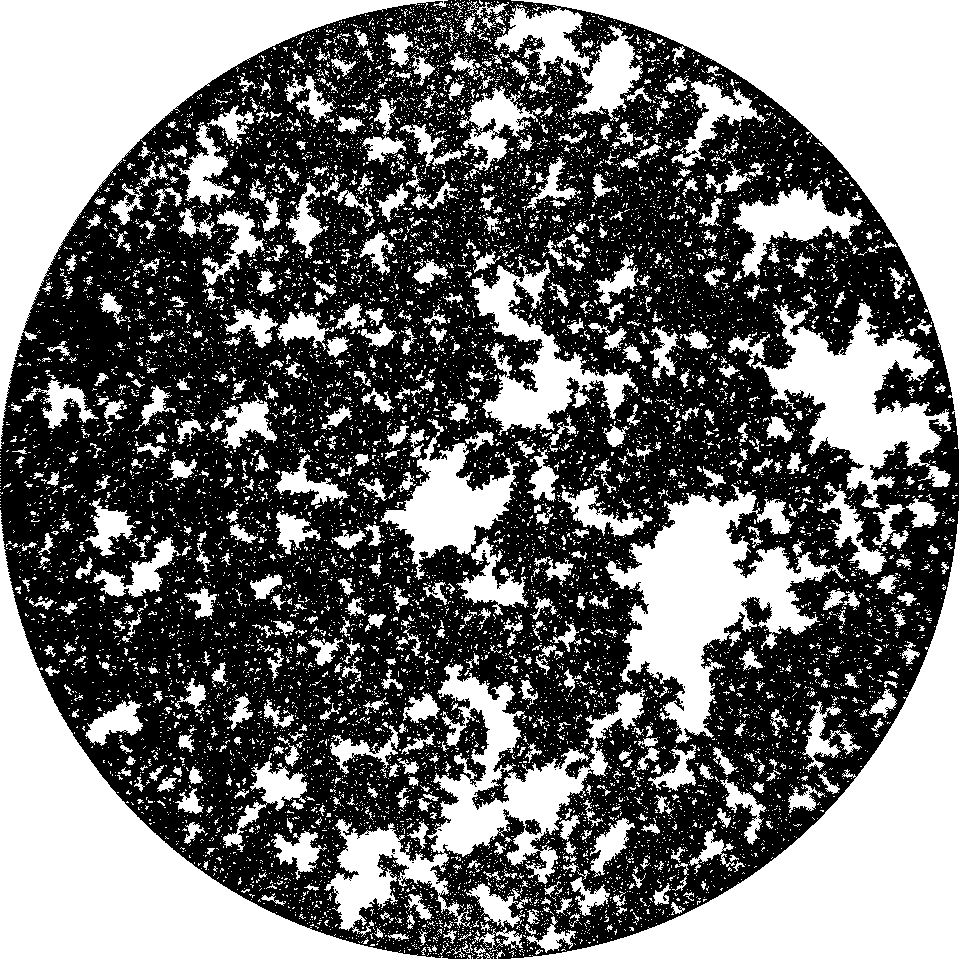}
		\caption{On the left a simulation of $\Aa_{-\lambda,\lambda}$ done by Brent Werness. On the right a simulation of $\Aa_{-2\lambda,2\lambda}$ done by David Wilson. Note the difference in the connectivity properties of loops.}
	\end{figure}
	
	A direct consequence of the main theorem concerns the SLE$_4$ fan. The SLE$_4$ fan is the equivalent of the SLE$_\kappa$ fan defined in \cite{MS1} for other values of $\kappa$ and it corresponds to a closed union of all level lines of $\Gamma$ between two boundary points. Due to the roughness of $\Gamma$, this is a random fractal set of $0$ Lebesgue measure. Again, the connected components of its complement are surrounded by loops and we show that the intersection graph of these loops is connected. See Corollary \ref{Cor fan} for a precise proof and statement.
	
	The main theorem also gives us tools to study the height profile of the two-valued sets. More precisely, for any loop $\ell$ of $\Aa_{-a,b}$, the label of $\ell$ is defined to be the value of $h_{\Aa_{-a,b}}$ inside this loop. One can then ask the following question \cite{ASW}: \textit{when is $h_{\Aa_{-a,b}}$ a measurable function of the set $\Aa_{-a,b}$?}.  It was known \cite{MS, ASW} that in the case of $\Aa_{-2\lambda, 2\lambda}$, the law of the labels conditionally on $\Aa_{-2\lambda, 2\lambda}$ is given by independent fair coin tosses. In Proposition \ref{non-independence} we show that this independence property is true if and only if $a = b = 2\lambda$. Moreover, in Proposition \ref{mes label} we show that not only independence, but also measurability of the labels fails for $a+b \geq 4\lambda$:
	
	\begin{itemize}
		\item If $2\lambda \leq a+b <4 \lambda$ and $a\neq b$, the labels of $\Aa_{-a,b}$ are a measurable function of the set $\Aa_{-a,b}$.
		\item If $\lambda \leq a <2\lambda$, the labels of $\Aa_{-a,a}$ are a measurable function of the set $\Aa_{-a,a}$ and the label of the loop surrounding $0$.
		\item If $a+b\geq 4\lambda$, the labels of $\Aa_{-a,b}$ cannot be recovered only knowing $\Aa_{-a,b}$ and any finite number of labels.
	\end{itemize}
	
	Another result we would like to mention corresponds to the ``level-set percolation'' of the 2D continuum GFF. More precisely, we say that there is ``level set percolation'' at height $(-a,b)$ if one can almost surely join any two boundary points via a continuous path inside the domain on which the GFF has values in $[-a,b]$. We show in Proposition \ref{Perco} that ``level set percolation'' occurs if and only if $\min(a,b) \geq 2\lambda$.  
	
	Finally, we construct and characterise, $(\Gamma,\B_r)$, a local set coupling where $\B_r$ has the same law as $\Aa_{-2\lambda, 2\lambda-r}$, for $0<r<2\lambda$, and such that the label of any loop $\ell$ encodes its distance to the boundary (see Proposition \ref{prop:: existence and uniqueness Br}). More precisely, the label of each loop $\ell$ is given by $2\lambda-r d(\ell, \partial D)$, where $d(\ell, \partial D)$ is the minimal length of a path of intersecting loops that connects $\ell$ to the boundary (see Proposition \ref{prop:lint}). 

	When we let $r\to 0$ in the latter coupling, we recover the coupling with the labelled CLE$_4$. Labelled CLE$_4$ was introduced in \cite{WernerWu} via a  conformally invariant growth-process, and a coupling to the GFF was proved in \cite{WaWu}. In fact, the labels of the CLE$_4$ are given by $2\lambda - t$, where $t$ is exactly the time-parameter in the conformally invariant growth-process mentioned above. In some sense, this coupling describes to how the outer-boundary of the outer-most sign cluster of the GFF changes after the Lévy transformation of the GFF, and we will explain how the sets $\B_r$ are in some sense natural approximations to a Lévy transform. 
	
	Lévy transform is well-defined in the case of the metric graphs \cite{LuW} and our  axiomatic characterisation of the coupling with the labelled CLE$_4$ in Section \ref{B0} allows us to study convergence of this Lévy transform. Our two aims in studying the labelled CLE$_4$ were as follows: to show that the labels in this coupling are measurable w.r.t. the underlying set, and to prove the existence of a conformally invariant metric, that has been mentioned in \cite{WaWu, WernerWu}. The sets $\B_r$ described above satisfy both of these properties, however we are unable to deduce the same statements in the limit.
	
	The rest of the paper is roughly structured as follows: we start with preliminaries on the GFF and the local sets in Section 2. In Section 3, we introduce the sets $\Aa_{-a,b}$, recall their construction and study some of their basic properties. In this section, we also prove a new construction of $\Aa_{-a,2\lambda - a}$ and study the level set percolation. In Section 4, we study the connectivity properties of the loops of $\Aa_{-a,b}$ and in Section 5 we address the question of the measurability of the labels $\Aa_{-a,b}$. Finally, in Section 6 we study the local set coupling with the labelled CLE$_4$.

	\section{Preliminaries on the Gaussian free field and local sets}
	
	Let $D\subseteq \R^2$ denote a bounded, open and simply connected planar domain. By conformal invariance, we can always assume that $D$ is equal to $\D$, the unit disk. Recall that the (zero boundary) Gaussian free field (GFF) in $D$  can be viewed as a centred Gaussian process $\Gamma$, 
	indexed by the set of continuous functions in $D$, such that if $f, g$ are continuous functions
	\[\E [(\Gamma,f) (\Gamma,g)]  =  \iint_{D\times D} f(x) G_D(x,y) g(y) d x d y. \] 
	Here $G_D$ is the Green's function (with Dirichlet boundary conditions) in $D$, that is by convention normalized such that $G_D(x,y)\sim \frac{1}{2\pi}\log(1/|x-y|)$ as $x \to y$. For this choice of normalization of $G$ (and therefore of the GFF), we set  \[\lambda=\sqrt{\pi/8}.\]Sometimes, other normalizations are used in the literature: If $G_D (x,y) \sim c \log(1/|x-y|)$ as $x \to y$, then $\lambda$ should 
	be taken to be $(\pi/2)\times \sqrt {c}$. 
	
	The Gaussian free field satisfies a spatial Markov property, and in fact, it also satisfies a strong spatial Markov property. To formalise this concept Schramm and Sheffield introduced local sets in \cite{SchSh2}. They can be thought of as a generalisation of stopping times to a higher dimension. 
	
	\begin{defn}[Local sets]
		Consider a random triple $(\Gamma, A,\Gamma_A)$, where $\Gamma$ is a  GFF in $D$, $A$ is a random closed subset of $\overline D$  and $\Gamma_A$ a random distribution that can be viewed as a harmonic function, $h_A$, when restricted to $D \backslash A$.
		We say that $A$ is a local set for $\Gamma$ if conditionally on $A$ and $\Gamma_A$, $\Gamma^A:=\Gamma - \Gamma_A$ is a  GFF in $D \backslash A$. 
	\end{defn}
	
	Here, by a random closed set, we mean a probability measure on the space of closed subsets of $\overline D$, endowed with the Hausdorff metric and its corresponding Borel $\sigma-$algebra. For random distributions we use the topology of the Sobolev space $\Hh^{-1}$. 
	
	All the local sets we consider in this paper are going to satisfy the following two assumptions, thus for the rest of the paper we may as well take them to be part of the definition of local sets, simplifying some non-important, but technical aspects:
	\begin{itemize}
		\item we work with local sets $A$ that are measurable functions of $\Gamma$;
		\item and such that $A\cup \partial D$ is connected. 
	\end{itemize}

The first part allows us to talk about ``exploring'' a specific local set of the free field. The second claim implies that all connected components of $D\backslash A$ are simply-connected and that the only polar local set is the empty set.

	Let us list a few properties of local sets : 
	\begin{lemma}[Basic properties of local sets]\label{BPLS}    $\ $
		\begin {enumerate}
		\item Any local set can be coupled in a unique way with a given GFF: Let be $(\Gamma,A,\Gamma_A,\widehat \Gamma_A)$, where $(\Gamma,A,\Gamma_A)$ and $(\Gamma,A,\widehat \Gamma_A)$ satisfy the conditions of this definition. Then, a.s. $\Gamma_A=\widehat \Gamma_A$. Thus, being a local set is a property of the coupling $(\Gamma,A)$, as  $\Gamma_A$ is a measurable function of $(\Gamma,A)$. 
		\item When $A$ and $B$ are local sets coupled with the same GFF $\Gamma$, and that $(A, \Gamma_A)$ and $(B, \Gamma_B)$ are conditionally independent given $\Gamma$, then $A \cup B$ is also a local set coupled with $\Gamma$. Additionally, $B\backslash A$ is a local set of $\Gamma^A$ with $(\Gamma^A)_{B\backslash A} = \Gamma_{B\cup A}-\Gamma_{A}$.
		\item Let $(\Gamma,A_n)$ be such that for all $n\in \N$ $(\Gamma,A_n)$ is a local set coupling, and for some $k\in \N$  almost surely $A_n\cup \partial D$ has less than $k$ connected components. Then $(\Gamma,A_n,\Gamma_{A_n})$ is tight and any subsequential limit is a local set coupling.
				
		If moreover the sets $A_n$ are increasing in $n$, then, $\overline{\bigcup  A_n}$ is also a local set and $\Gamma_{A_n}\to\Gamma_{\overline{\bigcup  A_n}} $ in probability in $\Hh^{-1}(D)$ as $n\to \infty$. 
		\item Let $(\Gamma,A_n)$  be such that for all $n\in \N$ $(\Gamma,A_n)$ is a local set coupling and the sets $A_n$ are decreasing in $n$. Then, $\bigcap A_n$ is also a local set and $\Gamma_{A_n}\to\Gamma_{\bigcap_n  A_n} $ a.s. as $n\to \infty$. 
	\end{enumerate}
\end{lemma}
\begin{proof} These properties can be found in \cite {SchSh2,WWln2,ALS1}. More precisely, the first claim comes from Lemma 3.9 of \cite{SchSh2}. 
The second result follows from Lemma 3.10 of \cite{SchSh2} and the proof of Lemma 3.11 of \cite{SchSh2}. The last two results follow because under the conditions on 
	$A_n$, Beurling estimate ensures that $G_{D\backslash A_n}\to G_{D\backslash A}$ as $n\to \infty$. A slightly more complicated version of this result is proved in Lemma 5.7 of \cite{ALS1}, where the local sets $A_n$ are coupled with metric graph GFFs on finer and finer meshes. Similarly, the first part of (3) can also be obtained via a slight adaptation of Lemma 4.6 of \cite{SchSh2}, which considers a sequence of discrete GFFs on finer and finer meshes. 
\end{proof}

Often we deal with a sequence of local sets that result form an exploration process, motivating the following definition:.
\begin{defn}[Local set process]
	We say that a coupling $(\Gamma,(\eta(t))_{t\geq 0})$ is a local set process if $\Gamma$ is a GFF in $D$, $\eta(0)\subseteq \partial D$, and $\eta_t$ is an increasing continuous family of local sets such that for all stopping time $\tau$ of the filtration $\F_t:=\sigma(\eta(s):s\leq t)$, 
	$(\Gamma,\eta(\tau))$ is a local set.
\end{defn}

Local set processes can be naturally parametrised from the viewpoint of any interior point $z$: the expected height $h_{\eta_t}(z)$ then becomes a Brownian motion. More precisely, if we define $\CR(z;D)$ as the conformal radius of $D$ from $z$ we have that:

\begin{prop}[Proposition 6.5 of \cite{MS1}]\label{BMinterior}
	For any $z \in D$, if $(\eta(t))_{t\geq 0}$ is parametrised such that 
	$(\log(\CR(z;D))-\log(\CR(z;D\backslash \eta([0,t])))(z,z)=t$, then $(h_{\eta([0,t])}(z))_{t\geq 0}$ has the law of a Brownian motion. 
\end{prop}

We mainly work with local sets that do not charge the GFF, called thin local set (see \cite{WWln2,Se}). More precisely, they are local sets $A$ such that  for any smooth test function $f \in \mathcal{C}^\infty_0$, the random variable $(\Gamma,f)$ is almost surely equal to  $(\int_{D \backslash A }  h_A(x) f(x) dx) + (\Gamma^A,f)$. This definition assumes that $h_A$ belongs to $\mathcal L^1(D\backslash A)$, which is always the case in our paper. For the general definition see \cite{Se}. In this setting it is not hard to derive a sufficient condition for a local set to be thin: if the Minkowski dimension of $A$ is almost surely strictly smaller than $d < 2$, then $A$ is thin (e.g. see proof of Proposition 4.3 in \cite{Se}).

Thin local sets are not that easy to work with: for example, we still cannot prove the intuitively clear statement that any thin sets has a.s. zero Lebesgue measure. Yet, we can still say that they are small in a certain way:

\begin{lemma}\label{thin empty}
	Let $\Gamma$ be a GFF in $D$. If $A$ is a thin local set of a GFF $\Gamma$, then, a.s. $A$ has empty interior.
\end{lemma}

\begin{proof}
	Assume for contradiction that $A$ has non-empty interior. Then there exists an $x\in D$ and $r>0$ such that with positive probability $B(x,D)\subseteq A$. Define $f$ a non-zero function with compact support in $B(x,D)$. Then, on the event where $B(x,D)\subseteq A$ \[(\Gamma_A,f)=(\Gamma,f)\neq 0=\int h_A(x)f(x)dx,\]
	giving a contradiction.
\end{proof}

Another natural class of local sets is that of bounded type-local set (BTLS), introduced in \cite{ASW}. A $K-$BTLS is a thin local set such that almost surely $|h_A|\leq K$ for some fixed $K > 0$. Intuitively $K-$BTLS correspond to stopping times $\tau$ of the Brownian motion $B_t$, which are small in the sense that, say, $\E [\tau] < \infty$, and satisfy $|B_\tau| \leq K$ for some $K > 0$. It is easy to see that in the case of the Brownian motion, any $K-$BTLS satisfies $\tau \leq \sigma_{-K,K}$ where $\sigma_{-K,K}$ is the first exit time from the interval $[-K,K]$. An analogue result was proved in \cite{ASW}: any $K-$BTLS of the 2D GFF is contained in a two valued set $\Aa_{-K',K'}$ with $K' \geq K + 2\lambda$ (we believe it should be true with $K' = K$).

\subsection{Level lines of the continuum GFF with piecewise boundary conditions}\label{LLs}
One of the simplest families of BTLS are the generalised level lines, first described in \cite{SchSh2}. 
We recall here some of their properties, see \cite{WaWu,ASW} for a more thorough treatment of the subject. To simplify our statements take $D:=\H$. Furthermore, let $u$ be a harmonic function in $D$. We say that $(\eta(t))_{t\geq 0}$, a curve parametrised by half-plane capacity, is the generalised level line for the GFF $\Gamma + u$ in $D$ up to a stopping time $\tau$ if for all $t \geq 0$:

\begin{description}
	\item[$(*)$]The set $\eta([0, t \wedge \tau])$ is a BTLS of the GFF $\Gamma$, with harmonic function $h_t:=h_{\eta([0,t\wedge \tau])}$ satisfying the following properties: $h_t + u$ is a harmonic function in $D \backslash \eta([0,\min (t, \tau)])$ with boundary values $-\lambda$  on the left-hand side of $\eta$, $+ \lambda$ on the right side of $\eta$, and with the same boundary values as $u$ on $\partial D$.
\end{description}

The first example of level lines comes from \cite{SchSh2}: Let $u_0$ be the unique bounded harmonic function in $\H$ with boundary condition $-\lambda$ in $\R-$ and $\lambda$ in $\R^+$. Then it is shown in \cite{SchSh2} that there exists a unique $\eta$ satisfying $(*)$ for $\tau= \infty$, and its law is that of an SLE$_4$. Several subsequent papers \cite{SchSh2,MS1,WaWu,PW} have studied more general boundary data in simply-connected case and also level lines in a non-simply connected setting \cite{ASW}. 

In this paper, we are just going to work with piecewise constant boundary conditions\footnote{Here, and elsewhere this means that the boundary conditions are given by a piecewise constant function that changes only finitely many times}. A careful treatment of level lines in this regime is done in \cite{WaWu}. We are now going to state Theorem 1.1.1 of \cite{WaWu}. Let $u$ be a bounded harmonic function with piecewise constant boundary data  such that $u(0^-) < \lambda$ and $u(0^+)>-\lambda$. 

\begin{lemma}[Existence of generalised level line targeted to $\infty$]\label{lemext}
	There exists a unique law on random simple curves $(\eta(t), t\geq 0)$ coupled with the GFF such that $(*)$ holds for the function $u$ and possibly infinite stopping time $\tau$ that is defined as the first time when $\eta$ hits a point $x\in \R$ such that $x \geq 0$ and $u(x^+) \leq -\lambda$ or $x \leq 0$ and $u(x^-) \geq \lambda$. We call $\eta$ the generalised level line for the GFF $\Gamma + u$. 
\end{lemma}

For convenience we also use the notion of a $(-a, -a+2\lambda)$-level line of $\Gamma+u$: it is a generalised level line of $\Gamma +a - \lambda+u$ and has boundary conditions $-a, -a +2\lambda$ with respect to the field $\Gamma+u$. Moreover, it is known that when $u=0$ this level line has the law of a SLE$_4(-a/\lambda, a/\lambda-2)$ process, see Theorem 1.1.1 of \cite{WaWu}.

Notice that as the level line is parametrised using half-plane capacity, it will accumulate at $\infty$ if not stopped earlier. 

Finally, let us recall Lemma in Section 3 of \cite{ALS1}, that sums up one of the key arguments of \cite{ASW}:
	\begin{lemma}\label{donotenter}
		Let $\eta$ be a generalized level line of a GFF $\Gamma + u$ in $D$ as above and $A$ a BTLS of $\Gamma$ conditionally independent of $\eta$. Take $z\in D$ and define $O(z)$ the connected component of $D\backslash A$ containing $z$. On the event where on any connected component of $\partial O(z)$ the boundary values of $(h_A+u)\mid _{O(z)}$ are either everywhere $\geq \lambda$ or everywhere $\leq -\lambda$ , we have that a.s. $\eta([0,\infty])\cap O(z)=\emptyset$.
	\end{lemma}
	
\section{Basic properties of two-valued local sets}\label{BPTVE}

The primary objective of this section is to introduce and state elemental properties of two-valued local sets (TVS). TVS were defined in \cite{ASW} as the equivalent of the exit times of an interval by a one-dimensional Brownian motion. After that, they have been used as a tool to construct the Liouville Quantum gravity \cite{APS}, to couple the Dirichlet and the free boundary GFF \cite{QW17}, and to study first passage sets \cite{ALS1}.

This section is organised as follows: first, we define TVS and recall some of the properties that were proved in \cite{ASW, ALS1}. Then, we recall their construction. Finally, we show some new properties about the loops of TVS.

\subsection{Definition and basic properties of TVS}
Fix $a,b>0$, and $\Gamma$ a GFF in a simple connected domain $D$. We say that $\Aa_{-a,b}$  is a TVS of levels $-a$ and $b$ if it is a thin local set of $\Gamma$ such that \hypertarget{tvs}: 
\begin{itemize}
	\item[(\twonotes)] For all $z\in D\backslash \Aa_{-a,b}$, a.s. $h_{\Aa_{-a,b}} (z) \in \{-a,b\}$.
\end{itemize} 

Let us recall the main properties of TVS.
\begin {prop}[Proposition 2 of \cite{ASW}]
\label {cledesc2}
Let us consider $-a < 0 < b$. 
\begin {enumerate}
\item
When $a+b  < 2 \lambda$, there are no thin local sets of $\Gamma$ satisfying \hyperlink{tvs}{(\twonotes)}.
\item
When $a+b \ge 2 \lambda$, it is possible to construct  $\Aa_{-a,b}$ coupled with a GFF $\Gamma$. Moreover, the sets $\Aa_{-a,b}$ are
\begin{itemize} 
	\item Unique in the sense that if $A'$ is another thin local set of $\Gamma$ satisfying \hyperlink{tvs}{(\twonotes)}, then $A' = \Aa_{-a,b}$ almost surely.  
	\item Measurable functions of the GFF $\Gamma$ that they are coupled with.
	\item Monotonic in the following sense: if $[a,b] \subset [a', b']$ and $-a < 0 < b$ with $b+a \ge 2\lambda$,  then almost surely, $\Aa_{-a,b} \subset \Aa_{-a', b'}$. 
	\item For any compact set $K\subseteq D$, $\Aa_{-a,b}\cap K$ has Minkowski dimension smaller or equal $2-2\lambda^2/(a+b)^2$. In particular they are thin local sets.
\end{itemize}
\end {enumerate}
\end {prop} 

The prime example of such a set is CLE$_4$ coupled with the Gaussian free field as $\Aa_{-2\lambda, 2\lambda}$, see \cite{MS, ASW}. 

As $\Aa_{-a,b}$ is a measurable function of the GFF $\Gamma$. When there are several GFFs at hand, we sometimes write $\Aa_{-a,b}(\Gamma)$ to be clear which GFF the set is coupled to. Sometimes, $\Gamma=\sum_{O}\Gamma^O$, where each $O$ is a simply connected domain and $\Gamma^{O}$ is an independent GFF in $O$. In those cases we write $\Aa_{-a,b}(\Gamma,O)$ as the TVS of level $-a$ and $b$ of the GFF $\Gamma^O$. Additionally, note that from the uniqueness statement we can conclude that almost surely $\Aa_{-a,b}(\Gamma)=\Aa_{-b,a}(-\Gamma)$.

The building stone of TVSs are the smallest ones, i.e., those such that $a+b=2\lambda$. We call $\A{a}$ (sometimes) the arc loop ensemble (ALE)\footnote{Especially when listening to the Beatles in a British pub close to the Newton Institute.} associated to $-a$, as they are a union of SLE$_4$ type of arcs. ALEs are used as a building block to construct more general TVS and are central in their study. Also, they were used in \cite{QW17} to couple the Dirichlet and Neumann free fields. 


\subsection{Construction of $\Aa_{-a,b}$}
Before discussing some further properties of two value sets, let us recall their construction as given in Section 6 of \cite{ASW} in some detail:

\subsubsection{Construction of $\A{a}$} \label{Cons basic A} Consider a GFF $\Gamma$ in $\D$ and fix two boundary points, say $-i$ and $i$ and explore the $(-a, -a+2\lambda)$-level line from $-i$ to $i$. Then for each connected component of $D\backslash \eta([0,\infty])$, $h_{\eta([0,\infty])}$ is the only bounded harmonic function with boundary condition $0$ in $\partial D$ and $-a$ or $2\lambda-a$ in $\eta([0,\infty))$ depending on whether the connected components lies at the left or at the right of $\eta(\infty)$ respectively. We define $A^1=\eta([0,\infty])$.

In each connected component $O$ of $D\backslash A^1$ to the left of $\eta$, we take each $x$ and $y$ to be one of the two different intersection points between $\eta$ and $\partial D$. Suppose that $(-i,x,y)$ are in counter-clockwise order. We then explore $\eta^O(\cdot)$, the $(-a,-a+2\lambda)$-level line  from $x$ to $y$ of $\Gamma^{A^1}+h_{A^1}$ restricted to $O$. There are two types of connected components of $O\backslash \eta^O([0,\infty])$: the ones whose boundary is a subset of $\eta^O([0,\infty]) \cup A^1$ and the others whose boundary is a subset of $\eta^O([0,\infty])\cup \partial D$. Note that if $\tilde O$ is a connected component of the first type, $h_{\eta^O([0,\infty])\cap A^1}$ restricted to $\tilde O$ is equal to $-a$ (see Figure \ref{fig:firstiteration}).
\begin{figure}[h!]    
	\centering
	\includegraphics[scale=0.5]{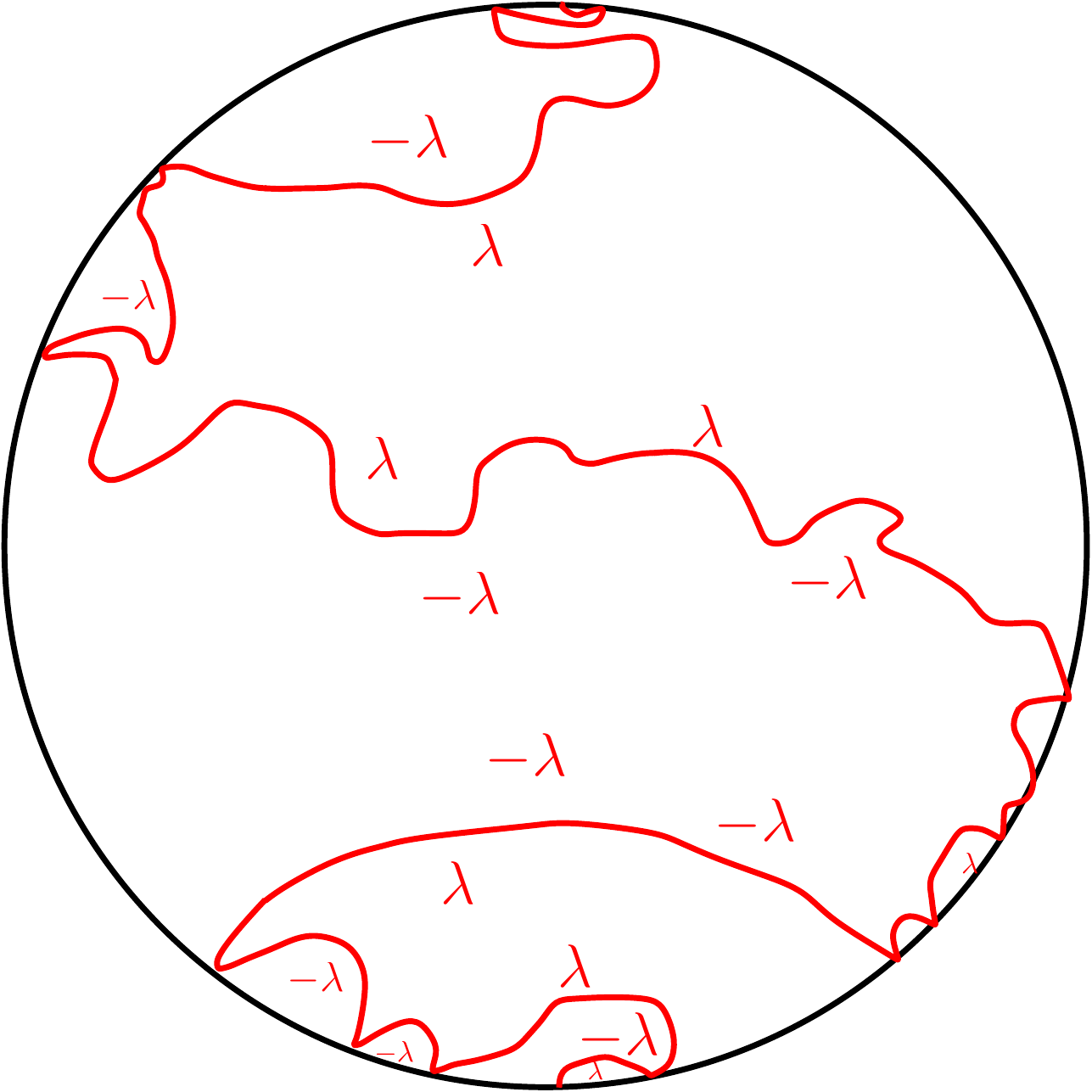}
	\includegraphics[scale=0.5]{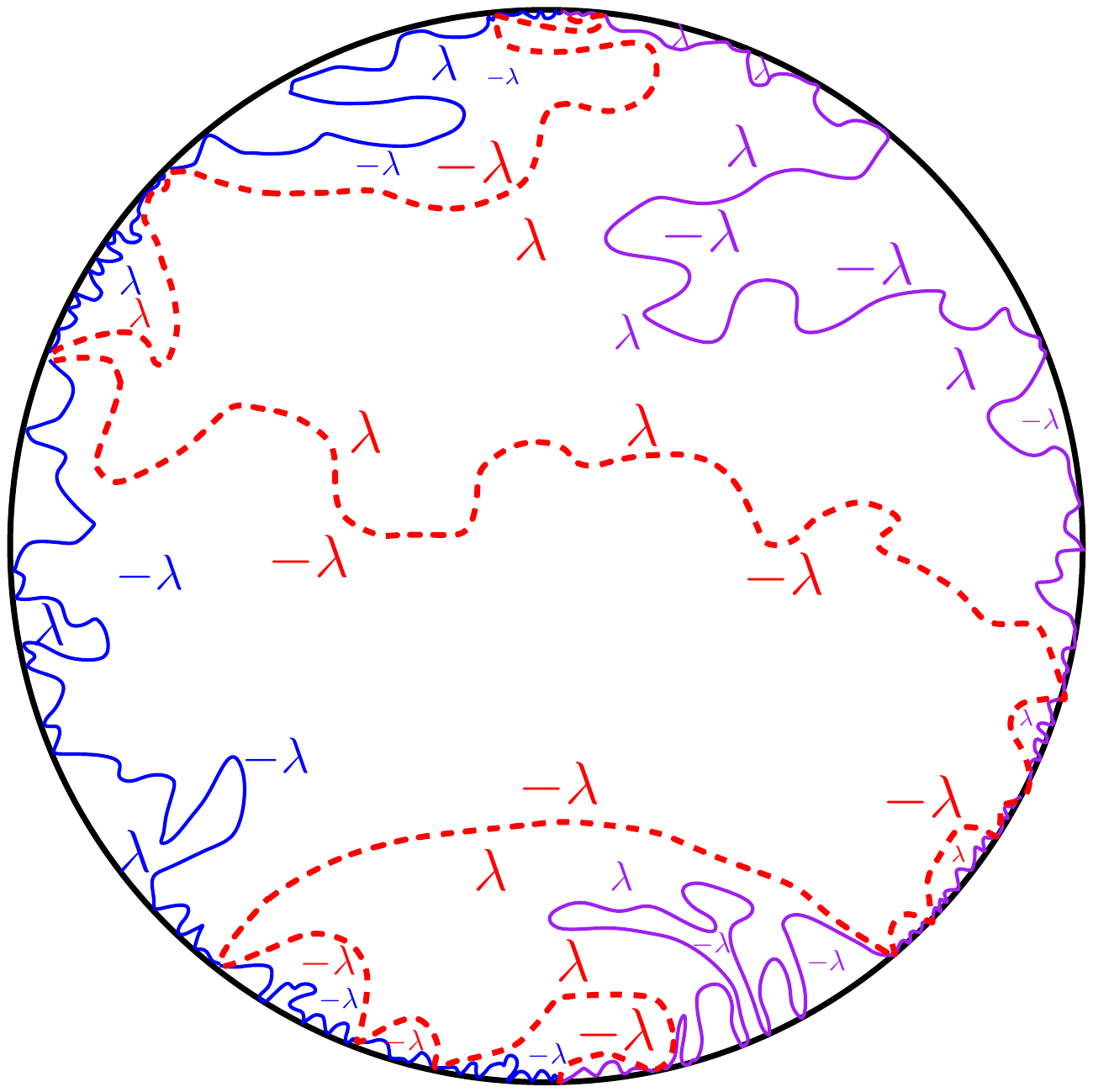}
	\caption {Construction of $\Aa_{-\lambda,\lambda}$. On the left we see a $(-\lambda,\lambda)$-level line from $-i$ to $i$, this is $A^1$. On the right we also see additional level lines going from one intersection point of $A^1$ to the other, thus the whole of $A^2$.}
	\label {fig:firstiteration}
\end{figure}

For each connected component $O$ of $D\backslash A^1$ to the right of $\eta$, we do similarly. The only difference is that we explore, $\eta^O(\cdot)$, the $(-a,-a+2\lambda)$ level line  from $y$ to $x$ of $\Gamma^{A^1}+h_{A^1}$ restricted to $O$. 

Now, define $A^2$ to be the closed union of $A^1$ with $\eta_{[0,\infty]}^O$ for each $O$ connected component of $\D\backslash A^1$.  Then in the connected components of $\D \backslash A^2$ whose boundary is a subset of $A^2$, the harmonic function $h_{A^2}$ is either constant equal to $-a$ or $2\lambda-a$ and we stop the iteration in these components. In the  other components, $h_{A^2}$ is equal to the only bounded harmonic function with boundary condition $0$ in $\partial D$ and $-a$ or $2\lambda-a$ in $A^2\cap \partial O'$ depending on whether the connected components lies to the right or the left of $A^1$ respectively. In these components, we iterate exactly as before to construct $A^n$. $\A{a}$ is the closed union of $A^n$.

From the uniqueness of $\A{a}$ (Proposition \ref{cledesc2}), we know that the arbitrary chosen starting and target points for the level lines and the order in which we sampled the level lines do not matter.

\begin{rem} \label{ALE}
	Note that the ALE, $\A{a}$, is constructed as a union of SLE$_4$-type paths. Moreover, each excursion of $\eta$ away from $\partial \D$ is on the boundary between two connected components of $\A{a}$ (one  loop labelled $-a$ to its right and one $2\lambda-a$ loop to its left). In particular, the Hausdorff dimension of an ALE is almost surely equal to $3/2$. Additionally, each connected component $O$ of $\A{a}$ is such that $\overline O\cap \partial D\neq \emptyset$. 
\end{rem}

\begin{figure}[h!]
	\includegraphics[scale=0.14]{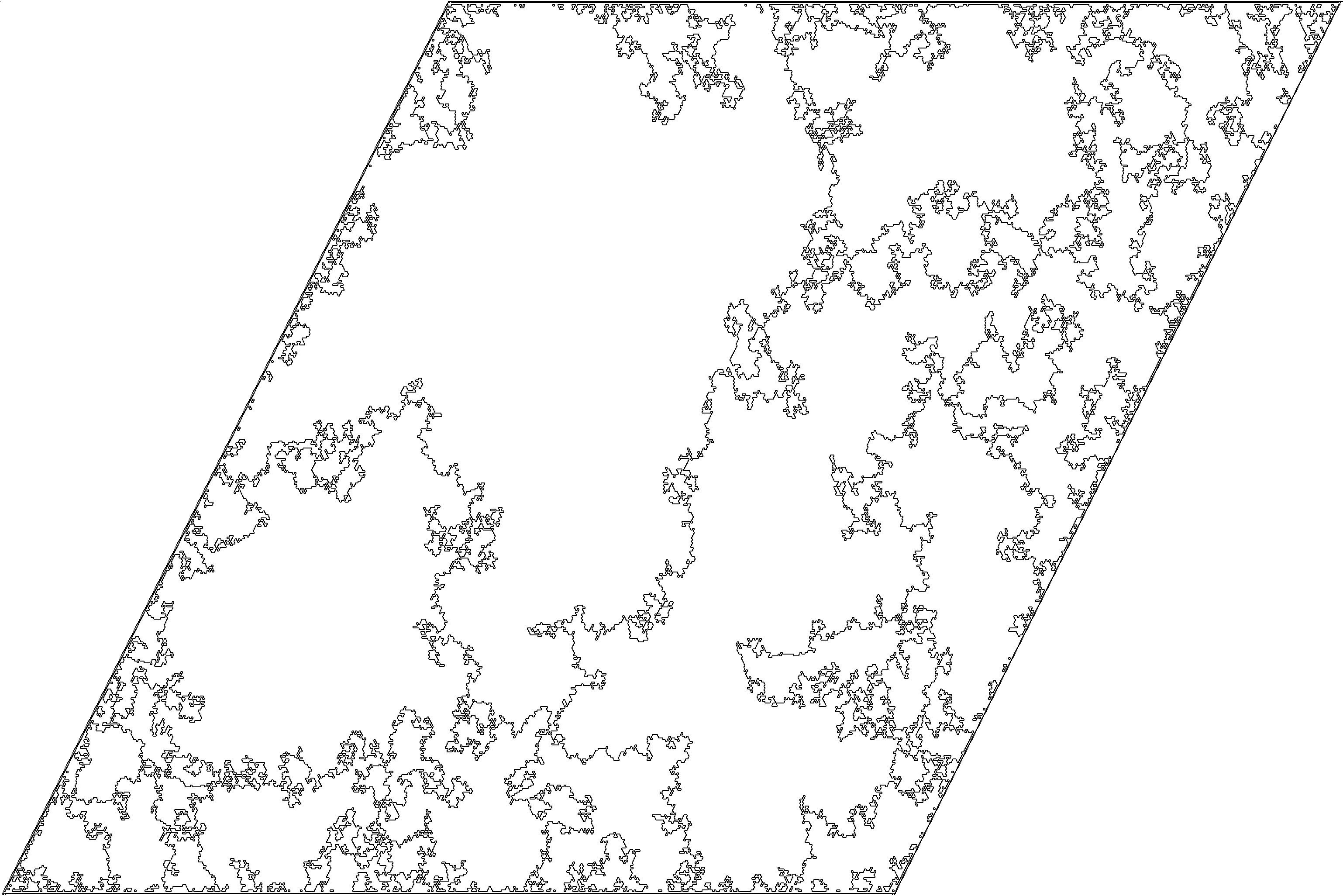}
	\caption{ALE of level $-\lambda$. Does it look similar to our drawings?. Simulation done by Brent Werness.}
\end{figure}

\subsubsection{Construction of $\Aa_{-a,b}$}\label{Construction A} We first the construct  $\Aa_{-a,b}$ for some ranges of values of $a$ and $b$, and then describe the general case. In order to simplify the notation, we use the following convention: if $A$ is a BTLS and $O$ is a connected component of $D \backslash A$, we say that $O$ is labelled $c\in \R$ if $h_A$ restricted to $O$ is equal to the constant $c$. 

\begin {itemize}
\item \underline{\textit{$a=0$ or $b=0$:}} We set $\Aa_{-a,b}=\emptyset$ and the corresponding harmonic function takes the value $0$ everywhere.

\item\underline{\textit{$a=n_1\lambda$ and $b=n_2\lambda$, where $n_1$ and $n_2$ are positive integers:}}   Define $A^1=\Aa_{-\lambda,\lambda}$, and define $A^{n+1}$ iteratively in the following way: inside each connected component of $O$ of $D\backslash A^n$ not labelled $-n_1\lambda$ or $n_2\lambda$ explore $\Aa_{-\lambda,\lambda}(\Gamma^{A^n},O)$. Define $A^{n+1}$ as the closed union between $A^n$ and the explored sets. Then, $\Aa_{-a,b}=\overline{ \bigcup A^n}$.

\item \underline{\textit{$a+b = n\lambda$ where $n \geq 3$ is an integer}}: Define $u\in [0, 2\lambda)$ such that there exists two integers $n_1\geq 0, n_2\geq 2$ with $a=u+n_1\lambda$ and $b =-u+n_2\lambda$. Let us start with $A:=\A{u}$. Inside each connected component $O$ of $D\backslash A$ labelled $-u$, resp. $-u+2\lambda$, explore $A_{-n_1 \lambda,  n_2 \lambda}(\Gamma^{A},O)$, resp. $\Aa_{-(n_1+2)\lambda,  (n_2 -2) \lambda}(\Gamma^{A},O)$. We have that $\Aa_{-a,b}$ is the closed union of $A$ with the explored sets.
\item
\underline{\textit{General case with $b+a > 2\lambda$}}: As $\Aa_{-a,b}(\Gamma)=\Aa_{-b,a}(-\Gamma)$, we may assume that $b>\lambda$.  Let $m\in \N$ such that $m \lambda-a \in (b-\lambda,b]$ and note that $m\geq 2$. Define $A^1:= A_{-a,m\lambda-a}$, and iteratively construct $A^n$ in the following way:
\begin{itemize}
	\item If $n$ is odd, then $D\backslash A^n$ is made of the closed union of loops with labels equal to either $-a$, $b$ or $m\lambda-a$. In every connected component, $O$, of $D\backslash A^n$ labelled $m\lambda-a$ we explore $A_{b+a-2m\lambda,b+a-m\lambda}(\Gamma^{A^n},O)$. Define $A^{n+1}$ the closed union of $A^{n}$ with the explored sets. Then all loops of $A^{n+1}$ have labels $-a$, $b$ or $b-m\lambda\in [-a,-a+\lambda)$.
	\item If $n$ is even, then $D\backslash A^n$ is made of the closed union of loops labelled  either $-a$, $b$ or $b-m\lambda$. In every connected component $O$ of $D\backslash A^n$ labelled $b-m\lambda$ explore $A_{-a-b+m\lambda,-a-b+2m\lambda}(\Gamma^{A^n},O)$. Define $A^{n+1}$ the closed union of $A^{n}$ with the newly explored sets. It is clear that all loops of $A^{n+1}$ have label $-a$, $b$ or $m\lambda-a$.
\end{itemize}
Then $\Aa_{-a,b}:=\overline{ \bigcup A^n}$.
\end {itemize}

Note that the uniqueness of $\Aa_{-a,b}$ implies that any arbitrary choice we made during the construction of $\Aa_{-a,b}$ does not matter. 

\begin{rem}\label{rem:thin}
Let us make an important point here: all explored sets throughout the construction are thin. One way to prove this fact, is to use Proposition 3 of \cite{ASW}: at any point in the construction we have only used level lines whose boundary data is bounded by some absolute constant $K=\max\{-a,b\}$. Each such level line is a $K-$BTLS and thus by Prop 3 of \cite{ASW} contained in a certain (only $K$-dependent) iteration of CLE$_4$. In particular, in any compact of $\D$ the Minkowski dimension of the constructed sets is bounded by $C(K)<2$. 
\end{rem}

\begin{rem}\label{rem: 2 boundary values}
Let $u$ be bounded harmonic function whose boundary value is piece-wise constant and changes only twice, i.e. is such that $u\mid_{\partial D}$ takes only two values $c$ and $d$, such that $u^{-1}(\{c\})$ is a connected. Then, by inspecting the construction above it is not hard to see that if $c,d\in [-a,b]$, one can construct a thin local set $\Aa_{-a,b}^u$ such that $h_{\Aa^u_{-a,b}}+u \in \{-a,b\}$, i.e. such that $h_{\Aa^u_{-a,b}}+u$ satisfies \hyperlink{tvs}{(\twonotes)}. Moreover, inspecting the proof of uniqueness of TVS in \cite{ASW}, one can also deduce their uniqueness. A generalisation of this statement, when $u$ takes finitely many values in the boundary is proved in \cite{ALS1}. In that article TVS are also studied in finitely-connected domains.
\end{rem}

\subsection{Some basic properties of $\Aa_{-2\lambda,2\lambda}$} One of our handles for answering questions on the geometry of two-valued set comes from our existing knowledge about the properties of $\Aa_{-2\lambda,2\lambda}$. Indeed, as discovered by Miller \& Sheffield \cite{MS} and explained in \cite{ASW}, the set $\Aa_{-2\lambda,2\lambda}$ has the law of a CLE$_4$ carpet. We will state some of these properties in a proposition below.

To do this, recall that we say that $\ell$ is a loop of $A$ if $\ell$ is the boundary of a connected component of $D\backslash A$. In our case, these boundary components are indeed Jordan curves. Define $\Lo(A)$ as the set of loops of $A$, and note that $A=\overline{ \bigcup_{\ell\in \Lo(A)} \ell}$. In this context, we say that a loop $\ell \in \Lo(\Aa_{-a,b})$ is labelled $-a$ or $b$, if $h_{\Aa_{-a,b}}$ restricted to the interior of $\ell$ is equal to $-a$ or, respectively, to $b$.

	\begin{prop}\label{BPCLE}
		Let $\Gamma$ be a GFF in $\D$ and $\Aa_{-2\lambda,2\lambda}$ be its TVS of level $(-2\lambda,2\lambda)$. Then this coupling satisfies the following properties:
		\begin{enumerate}
			\item The loops of $\Aa_{-2\lambda,2\lambda}$ are locally finite, i.e. for any $\eps>0$  there are only finitely many loops that have diameter bigger than $\eps$.
			\item Almost surely no two loops of $\Aa_{-2\lambda,2\lambda}$ intersect, nor does any loop intersect the boundary.
			\item The conditional law of the labels of the loops of $\Aa_{-2\lambda,2\lambda}$ given $\Aa_{-2\lambda,2\lambda}$ is that of i.i.d random variables taking values $\pm 2\lambda$ with equal probability.
		\end{enumerate}
	\end{prop}

The first wo properties just follow from the fact $\Aa_{-2\lambda,2\lambda}$ has the law of a CLE$_4$ (See Section 4 of \cite{ASW}) and that this property are true for the CLE$_4$ (see \cite{SW}).  For (3) see, for example, the last comment in Section 4.3 of \cite{ASW}.

\subsection{ALE as the closed union of level lines} \label{Cons union}

In Section \ref{Cons basic A} we saw how to construct $\A{a}$ iteratively. We now show the intuitively appealing statement that this set can be obtained by just taking the union of all $(-a,-a + 2\lambda)$ level lines. 

Let $\Gamma$ be a GFF in $\D$. To simplify notations we write $\eta^{a,x,y}$ for the $(-a,-a + 2\lambda)$-level line going from $x$ to $y$, and $\tilde \eta^{a,x,y}$, $(-a+2\lambda,-a)$-level line, going also from $x$ to $y$. The following lemma makes the previous statement precise:

\begin{lemma}
	Let $(x_i)_{i\in I}$ be a countable dense sets of boundary points. Then, for all $a\in (-\lambda,\lambda)$ a.s.
	\begin{align*}
	\A{a}=\overline{\bigcup_{i,j\in \N} \eta^{a,x_i,y_i}\cup \tilde \eta^{a,x_i,y_i}}.
	\end{align*}
\end{lemma}

\begin{rem}
Notice that the countable union is well-defined as all these level lines can be coupled with the same GFF so that they are measurable w.r.t. that GFF. In particular the union can be constructed by exploring these level lines in any convenient order that guarantees exploring all of them. Additionally, thanks to the reversibility of level lines (Theorem 1.1.6 of \cite{WaWu}) \[\overline{\bigcup_{i,j\in \N} \eta^{a,x_i,y_i}\cup \tilde \eta^{a,x_i,y_i}}= \overline{\bigcup_{i,j\in \N} \eta^{a,x_i,y_i}}.\]
\end{rem}

\begin{proof}
By Lemma \ref{donotenter} for all $i,j\in \N$, the level line $\eta^{a,x_i,y_i}$ is a subset of $\A{a}$. Thus, $U:=\overline{\bigcup_{i,j\in \N} \eta^{a,x_i,y_i}} \subseteq \A{a}$.
	
	To show the other direction, we may assume that $\pm i\in \{x_i:i\in \N\}$, as this was an arbitrary choice in the construction of the ALE. Now, let us prove by induction on $n\in \N$ that the sets $A^n$ of the construction in Section \ref{Cons basic A} are a.s. contained in $ \overline{\bigcup_{i,j\in \N} \eta^{a,x_i,y_i}}$. Note that this is true for $A^1$, as it is just the $(-a,-a+2\lambda)$ level line from $-i$ to $i$.

	Let us note that for every connected component $O$ of $D\backslash A^n$ whose label is not yet equal to constant $-a$ or $-a+2\lambda$, the set $A^{n+1} \backslash A^n$ restricted to (the closure of) $O$ is a level line $\eta_0$ belonging to a unique loop $\ell_O\subseteq \bar O$ of $\A{a}$ such that $\ell_O\cap A^n\neq \emptyset$. Thus, to prove the induction step it is just enough to prove that the level line $\eta_O$ is contained in $U$.
	
	Assume WLOG that we are in a connected component $O$ of $D\backslash A^n$ such that the boundary condition of $h_{A^n}$ restricted to $A_n$ is $2\lambda-a$, this implies that $\ell_O$ is labelled $2\lambda-a$. Denote by $(z_1,z_2)$ the counter-clockwise arc on $\partial \D \cap \partial O$. 
	
	Pick $x_{j_m}$ and $y_{j_m}$ on the arc $(z_1, z_2)$ with $y_{j_m} \to z_1$ and $x_{j_m} \to z_2$. Draw $\eta^{a,x_{j_m}, y_{j_m}}$. Now we finish the construction of $\A{a}$ inside each connected component of $O \backslash \eta^{a,x_{j_m}, y_{j_m}}$ that does not have a part of $A^n$ on its boundary, exactly as in Section \ref{Cons basic A}. This way we obtain a local set $\hat A_m$ inside $O$ such that $h_{\hat A_m}$ is equal to either $-a$ or $-a+2\lambda$ inside all the connected components of $O \backslash \hat A$ which do not have a part of $A^n$ on its boundary. Moreover, in the one remaining component the boundary values are $-a+2\lambda$ everywhere, but on the two small intervals between $z_2$ and the right-most boundary intersection point of $\eta^{a,x_{j_m}, y_{j_m}}$ and between $z_1$ and the left-most boundary intersection point of $\eta^{a,x_{j_m}, y_{j_m}}$, where the boundary value are zero. Notice that the boundary of this one component did not change during the completion of $\A{a}$ inside the other components. 
	
	Now from Lemma \ref{BPLS} (3) and Lemma 5.8 in \cite{ALS1} (or Lemma 4.5 in \cite{SchSh2})it follows that $\hat A_m$ converges in probability to the two-valued set $\A{a}$ inside $O$ w.r.t the Hausdorff distance. In particular, this means that $\eta_O$ is contained in the closure of $\bigcup_m \eta^{a,x_{j_m}, y_{j_m}}$ and the induction step follows.

\end{proof}

\subsection{Boundary touching of single loops}

We will now make our first step towards understanding the intersection between loops: we study when the loops of $\Aa_{-a,b}$ touch the boundary. Moreover, we also determine the Hausdorff dimension of the intersection of any of these loops with the boundary. The key ingredient here and later on in the paper is the uniqueness of TVSs: it allows us to choose a particular construction of the TVS, where the property in question becomes evident.

\begin{lemma}\label{keepingloops}
	Let $a,b, \delta> 0$ with $a+b\geq 2\lambda$. Then almost surely, 
	\begin{enumerate}
		\item Each loop of of $\Aa_{-a,b}$ with label $-a$ is also a loop of $\Aa_{-a,b+\delta}$ with label $-a$.
		\item A loop $\ell$ of $\Aa_{-a,b}$ labelled $-a$ touches the boundary iff $a<2\lambda$ and $\ell$ is a loop of $\A{a}$ labelled $-a$.
	\end{enumerate}
\end{lemma}

\begin{proof}
	For the first part, note that we can construct $\Aa_{-a,b+\delta}$ in the following way:
	\begin{itemize}
		\item explore $\Aa_{-a,b}$; 
		\item explore $\Aa_{-a-b, \delta}(\Gamma^{\Aa_{-a,b}},O)$ inside each connected component $O$ of $D\backslash \Aa_{-a,b}$ labelled $b$. 
	\end{itemize}
	Defining $A'$ as the closed union of $\Aa_{-a,b}$ with the newly explored sets, we see that it is a BTLS that satisfies \hyperlink{tvs}{(\twonotes)} with levels $-a$ and $b+\delta$. The fact that that it is thin follows as in Remark \ref{rem:thin}. Thus, $A'=\Aa_{-a,b+\delta}$ by uniqueness (Proposition \ref{cledesc2}). The claim now follows as in the second step there were no explored set inside the loops labelled $-a$.  
	
	For the second part, let us first assume $a \geq 2\lambda$ and show that no loop with label $-a$ touches the boundary. If $b\geq 2\lambda$, then $\Aa_{-2\lambda,2\lambda}\subseteq \Aa_{-a,b}$ by monotonicity of TVS and Proposition \ref{BPCLE} (ii) we see directly that there are no loops of $\Aa_{-a,b}$ touching the boundary. If $b <2\lambda$, then from part (1) it follows that all loops of $\Aa_{-a,b}$ with label $-a$ are also loops of $\Aa_{-a, 2\lambda}$ with the same label and thus do not touch the boundary.
	
	Now, let us study the case $a <2 \lambda$. By using the part (i) and the fact that all loops of $\A{a}$ touch the boundary we get that all loops of $\Aa_{-a,b}$ that are loops of $\A{a}$ touch the boundary. We still need to show that these are the only ones. Remember that $b \geq -a+2\lambda$. Thus, one can construct $\Aa_{-a,b}$ by first exploring $\A{a}$ and then exploring $ \Aa_{-2\lambda,b+a-2\lambda}(\Gamma^{\A{a}},O)$ inside all connected components $O$ of $D\backslash \A{a}$ with the label $-a+2\lambda$. Thus, the loops with label $-a$ in $\Aa_{-a,b}$ are of two types: either those with labelled $-a$ in $\A{a}$, or those with the label $-2\lambda$ in $ \Aa_{-2\lambda,b+a-2\lambda}(\Gamma^{\A{a}},O)$. We conclude by noting that, thanks to the previous paragraph, the latter loops do not touch the boundary of the domain.
	
\end{proof}

To understand how ``often'' the loops of $\Aa_{-a,b}$ hit the boundary, we need to know the boundary touching behaviour of $SLE_4(\rho_1, \rho_2)$. The case of $\SLE_4(\rho)$ is covered in \cite{Lukas}, and for all other $\kappa\neq 4$ it was covered in Theorem 1.6 of \cite{MW}. The full result then follows by, for example, absolute continuity for the GFF and Proposition 14 of \cite{ASW}. 
\begin{prop}[\cite{Lukas}] \label{dimension bd}
	Take $\rho_1,\rho_2 >-2$ and let $\eta$ be an $\SLE_4(\rho_1, \rho_2)$ on $\H$. Then
	\[
		\text{Dim}_{Haus} (\eta \cap \R^+) = \max\left \{1-\frac{(\rho_2+2)^2}{4},0\right \}.\]
\end{prop}

Now in the construction of  $\A{a}$ (see Section \ref{Cons basic A}) the level lines used to build the loops labelled $-a$ have the law of SLE$_4(-a/\lambda,a\lambda-2)$ and $\SLE_4(-a/\lambda, 0)$ processes (see Theorem 1.1.1 \cite{WaWu}). Moreover, from the construction of $\A{a}$ it follows each loop is finished after a finite number of iterations. Thus, we can calculate the a.s. dimensions of the boundary intersection of the loops labelled $-a$ in $\A{a}$. Together with Lemma \ref{keepingloops} this implies the following corollary.
\begin{cor}\label{cor bd}
	Let $0<a<2\lambda$, and $\Gamma$ a GFF in $\H$. Then a.s. any loop of $\Aa_{-a,b}$ labelled $-a$ either does not touch the boundary or it touches the boundary infinitely often. Moreover, in the latter case, the set of intersection points has Hausdorff dimension $1-(2-a/\lambda)^2/4$.
\end{cor}

Finally, recall that one can construct $\Aa_{-a,b}$ by first exploring $\A{a}$ and then exploring sets $\Aa_{-2\lambda,b+a-2\lambda}$ inside loops with the label $-a+2\lambda$. Thus from (2) of Lemma \ref{keepingloops} and from the fact that $\A{a}=M$ where $M$ is the union of all loops of $\A{a}$ with the label $-a$ (see Remark \ref{ALE}), it follows that we can inversely reconstruct $\A{a}$ from $\Aa_{-a,b}$ by just observing its intersection with the boundary.
\begin{cor}\label{recons}
	Suppose $b\neq a<2\lambda$. Let $M$ be the union of all loops of $\Aa_{-a,b}$ touching the boundary with dimension $1-(2-a/\lambda)^2/4 $, then almost surely $ M=\A{a}$.
\end{cor}

\subsection{Level set percolation of the 2D continuum GFF} 

It comes out that the properties of the two-valued sets we have described are already enough to describe the critical point of ``the level set percolation'' of the two-dimensional GFF. More precisely, we ask when can any two boundary points be connected inside the set $\Aa_{-a,b}$, i.e. heuristically via a path where the GFF only takes values in $[-a,b]$. We expect that such results in the continuum would help to find the level set percolation threshold for the 2D metric graph GFF and the 2D discrete GFF. 

\begin{prop} \label{Perco}Take $a,b>0$ with $a+b \geq 2\lambda$, $\Gamma$ a GFF on $\D$.
Then
\begin{itemize}
\item The boundary is not connected via $\Aa_{-a,b}$ if $a < 2\lambda$ or $b < 2\lambda$ : for any fixed $x, y \in \partial \D$, almost surely there is no continuous path joining $x$ and $y$ and whose interior is contained in $\Aa_{-a,b} \cap \D$. 
\item The boundary is connected via $\Aa_{-a,b}$ if $a,b \geq 2\lambda$: almost surely for any $x, y \in \partial \D$ there exists a continuous path joining $x$ and $y$ and whose interior is contained in $\Aa_{-a,b} \cap \D$.
\end{itemize}
\end{prop}

\begin{rem}
Notice that even in the case of the ALE $\A{a}$ there are exceptional points that are connected via a continuous path in $\Aa_{-a,b} \cap \D$, corresponding to the endpoints of SLE excursions away from the boundary.
\end{rem}

\begin{rem}
We use quotation marks for ``level set percolation'' as the GFF is not defined pointwise. However, let us mention that via convergence results from the metric graph it is shown in \cite{ALS2} that the first passage set $\Aa_{-a} := \overline{\bigcup_{b \in \N} \Aa_{-a,b}}$ does correspond to a certain level set connected to the boundary: it is the limit as the mesh size goes to zero of the set of points on the metric graph that can be connected to the boundary via a path on which the metric graph GFF is bigger or equal $-a$. So when two boundary points can be joined inside $\Aa_{-a}$, then this indeed says that there is level set percolation in terms of boundary-to-boundary percolation. As the proof of the first part of the proposition directly works for $\Aa_{-a}$ as well, we can intrpret the results as follows:
\begin{itemize} 
\item the critical height for the boundary-to-boundary level set percolation of the 2D continuum GFF is $-2\lambda$, and moreover, there is percolation at the critical level.
\end{itemize}
\end{rem}

\begin{proof}
Let us first prove the first statement. It suffices to consider the case $a < 2\lambda.$ We say that $x,y\in \partial D$ are separated by a loop $\ell$ if $x,y\notin\ell$ and they are not in the same connected component of $\partial\D\backslash \ell$. Thanks to the fact that $\D$ is simply connected and that the loops are continuous curves, if $x$ and $y$ are separated, then they belong to different components of $\overline \D \backslash \ell$. Thus, to finish the proof it is enough to find a loop of $\Aa_{-a,b}$ separating $x$ from $y$. Furthermore, by Lemma \ref{keepingloops}, it suffices to find a loop of $\A{a}$ labelled $-a$ that separates $x$ and $y$.

Let us prove now that for any $x,y\in \partial D$ there is a loop  of $\A{a}$ labelled $-a$ that separates $x$ and $y$: consider some points $z_1,z_2\in \partial D$ such that $z_1,x,z_2,y$ are in counter-clockwise order, and conformally map the domain to a disk such that $\phi(z_1)=i$, $\phi(z_2)=-i$, $\phi(x)=1$. Now, consider the construction of $\A{a}$ of Section \ref{Cons basic A}. Almost surely the $(-a,-a+2\lambda)$ level line from $-i$ to $i$ separates $1$ from $\phi(y)$, and does not touch $1$ nor $\phi(y)$. In particular, there will be an excursion off the boundary of this level line (i.e. a subcurve of the level line that touches the boundary only at its two endpoints) that separates $1$ from $\phi(y)$. By the construction of $\A{a}$ such an excursion belongs to the boundary of a connected component with label $-a$ in $\D \backslash \A{a}$.

We will now prove the second statement. By monotonicity of $\Aa_{-a,b}$ it suffices to prove it for the case of $\Aa_{-2\lambda,2\lambda}$. Thanks to Proposition \ref{BPCLE} (iv), for any $\epsilon > 0$ there are only finitely many loops of $\Aa_{-2\lambda,2\lambda}$ of diameter larger than $\epsilon$. Moreover, as mentioned before the boundaries of these connected components are given by Jordan curves. Denote by CLE$_4^\epsilon$ the complement of all  the components of size larger than $\epsilon$ and call the boundaries of these components loops. By the interior of a loop we mean the open set separted by the loop from the boundary.

Now, take $x,y\in \partial \D$ and draw a straight line segment $L$ joining $x$ and $y$ parametrized by $[0,1]$. To construct a continuous curve between $x, y$ inside CLE$_4^\epsilon$ we follow the straight line $L$, unless we meet a CLE$_4^\epsilon$ loop, in which case we follow it in the clock-wise sense until we meet $L$ again. More precisely, we define $C^\epsilon$ as follows: for all $t \in [0,1]$
\begin{itemize}
	\item if $L(t)$ is not in the interior of a loop of CLE$_4^{\epsilon}$, we let $C^\epsilon(t)=L(t)$. 
	\item If $L(t)$ is in the interior of a loop, define $t^-< t$ as the last time $L(t^-)\in CLE_4^{\epsilon}$ and $t < t^+$ as the first time after $t$ such that $L(t^+)\in CLE_4^\epsilon$. Then $C^\epsilon$ restricted to $[t^-,t^+]$, follows in a counter-clockwise manner the loop that contains $L(t^-)$, from $L(t^-)$ to $L(t^+)$ (with some arbitrary continuous speed).
\end{itemize}

As the loops added between $\epsilon$ and $\epsilon'<\epsilon$ have a diameter smaller than $\epsilon$, we have that $\|C^\epsilon-C^{\epsilon'}\|_\infty \leq \epsilon$. Thus, $C^\epsilon$ converges uniformly to $C^0$ as $\epsilon \to 0$. Moreover, thanks to the fact that  $L^\epsilon$ is contained in $CLE_4^\epsilon$ and as $\bigcap_{\epsilon > 0} CLE_4^\epsilon = \Aa_{-2\lambda,2\lambda}$, $C^0$ is contained in $\Aa_{-2\lambda,2\lambda}$. Moreover, $C^0$ cannot touch the boundary in other points than $x$ or $y$: indeed, suppose for contradiction it hits some point $z\in \partial \D$ at distance $\delta$ from $L$. Then notice that all loops of $\Aa_{-2\lambda,2\lambda}$ intersecting in $L$ and not contained in $CLE_4^{\delta/2}$ stay at a distance $\delta/2$ from $z$. However, we know that any loop of $CLE_4^{\delta/2}$ stays at a positive distance from $\partial \D$. Thus the claim follows.

\end{proof}

\subsection{$\Aa_{-a,b}$ are locally finite when $a+b\leq 4\lambda$} We now prove a very basic, but important property of the TVS: we show that when $a+b\leq 4\lambda$, then almost surely $\Aa_{-a,b}$ is locally finite, i.e. there are only finitely many loops of size diameter bigger than $\epsilon$. The proof is based on the monotonicity of TVS and the fact that $\Aa_{-2\lambda,2\lambda}$ is locally finite. 
	
\begin{prop}
	Take $a+b\leq 4\lambda$ and let $\Aa_{-a,b}$ be the TVS in $\D$. Then almost surely $\Aa_{-a,b}$ is locally finite.
\end{prop}
\begin{rem}
	In fact, TVS are always locally finite. The only way we know how to prove this in the case when $a+b$ is larger than $4\lambda$ is to use loop soup techniques. Thus the proof is presented in \cite{ALS2}.
\end{rem}
\begin{proof}
 It suffices to show local finiteness separately for loops labelled $-a$ and for loops labelled $b$. Let us concentrate on the case $-a$, the other case following similarly. By Lemma \ref{keepingloops} it suffices to prove the claim for $\Aa_{-a,-a+4\lambda}$. As for the case $a=2\lambda$, this is (1) of Proposition \ref{BPCLE}, we may assume that $a \neq 2\lambda$.
 
 Let us construct $\Aa_{-2\lambda,2\lambda}$ in two steps: first we explore $\Aa_{a-2\lambda,2\lambda}$ if $a<2\lambda$, or $\Aa_{-2\lambda,a-2\lambda}$ if $a>2\lambda$, and call this local set $A$. Then inside all connected components $O$ of $\D\backslash A$ labelled $a-2\lambda$, we further explore $\Aa_{-a,4\lambda-a}(\Gamma^{A},O)$. 
 Now consider $\hat O$, the connected component of  $\D\backslash A$ containing $0$, and define $f_{\hat O}$ to be the conformal map from $\hat O$ to $\D$ such that $f_{\hat O}(0)=0$. With positive probability $d(0,A)\geq 1/8$ and on this event, by Proposition 3.85 of \cite{Lawler}, the image $f_{\hat O}(\ell)$ of any loop $\ell\subseteq \hat O$ of diameter smaller than $\epsilon$, has diameter smaller than $c\epsilon^{1/2}$ (for a deterministic constant $c>0$). Thus, on this event $f_{\hat O}(\Aa_{-2\lambda,2\lambda}\cap \hat O)$ is locally finite, as $\Aa_{-2\lambda, 2\lambda}$ is locally finite. 
 
 But now, conditionally on $A$ and conditioned on the label of $\hat O$ being $a-2\lambda$, the image of $f_{\hat O}(\Aa_{-2\lambda,2\lambda}\cap \hat O)$ has the law of $\Aa_{-a,-a+4\lambda}$ in $\D$ independently of $\hat O$. Thus we conclude that $\Aa_{-a,-a+4\lambda}$ is locally finite.
\end{proof}

\section{Connectivity properties for two-valued local sets}

This section aims to understand the connectivity properties of the loops of $\Aa_{-a,b}$. To do this, we formulate the connectivity properties using two different graphs. We say that two loops of $\Aa_{-a,b}$ are `side-connected' if the intersection of these two loops contains a set that is homeomorphic to a segment; we say that they are `point-connected' if their intersection is non-empty. Consider the graphs $G_s, G_p$ whose vertex set are the loops of $\Aa_{-a,b}$ and edge sets $E_s$, $E_p$ consisting of pairs of loops that are either `side-connected' or `point-connected' respectively. Notice that by definition $E_s \subset E_p$. 

We now state how the connectivity properties of the loops of $\Aa_{-a,b}$ depend on $a+b$ (see Figure \ref{fig:mainthm})
\begin{thm}\label{mainthm}
	Let $\Aa_{-a,b}$ with $a, b > 0$ and $a+b \geq 2\lambda$ be a two valued set of level $-a$ and $b$. Then 
	\begin{enumerate}
		\item If $a + b = 2\lambda$, the graph $G_s$ is equal to $G_p$. Additionally, it is connected, i.e. one can pass from each loop to any other one in finite number of steps using `side-connections'.
		\item If $2\lambda < a+b < 4\lambda$, the edge set $E_s$ is empty but the graph $G_p$ is connected, i.e. one can pass from each loop to any other one in finite number of steps using `point-connections'.
		\item If $a+b \geq 4\lambda$, then $E_p$ is empty, or in other words all loops are pairwise disjoint.
	\end{enumerate}
	Moreover, in all phases any two loops with the same label are neither side-nor point-connected, in particular $G_p$ and $G_s$ are bipartite. 
\end{thm}

\begin{rem}
Thus in the regime $a+b < 4\lambda$ one can define conformally invariant distance between the loops of $\Aa_{-a,b}$ using the graph distance $G_p$. We will see in Section \ref{aprox CLE_4} how to encode the distances to the boundary as labels of another family of thin local sets and how to rescale these distances in order to extend to define distances in the case $a=b = 2\lambda$.	

 \end{rem}

\begin{figure}[ht!]    
	\centering
	\includegraphics[scale=0.35]{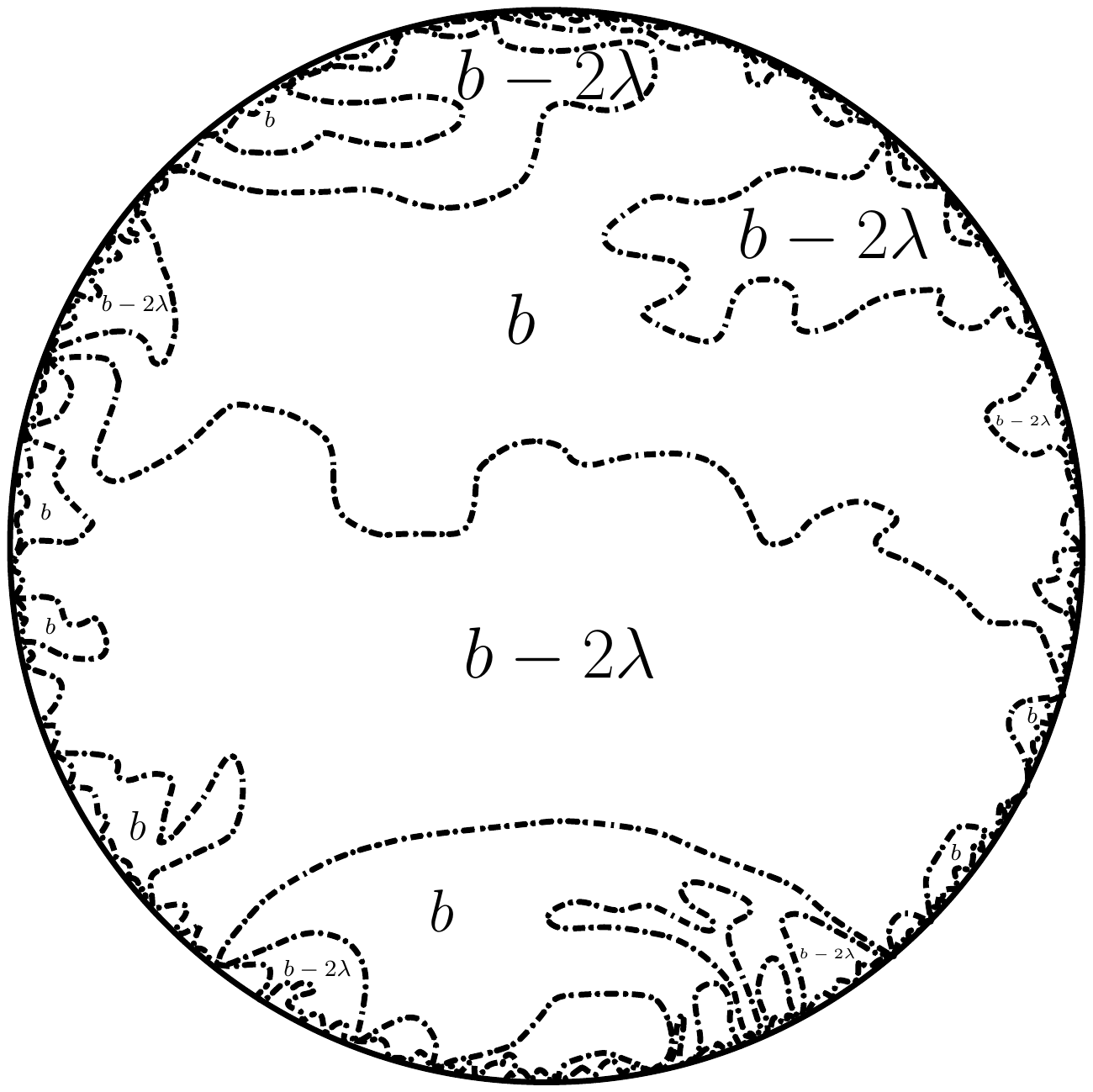}
	\includegraphics[scale=0.35]{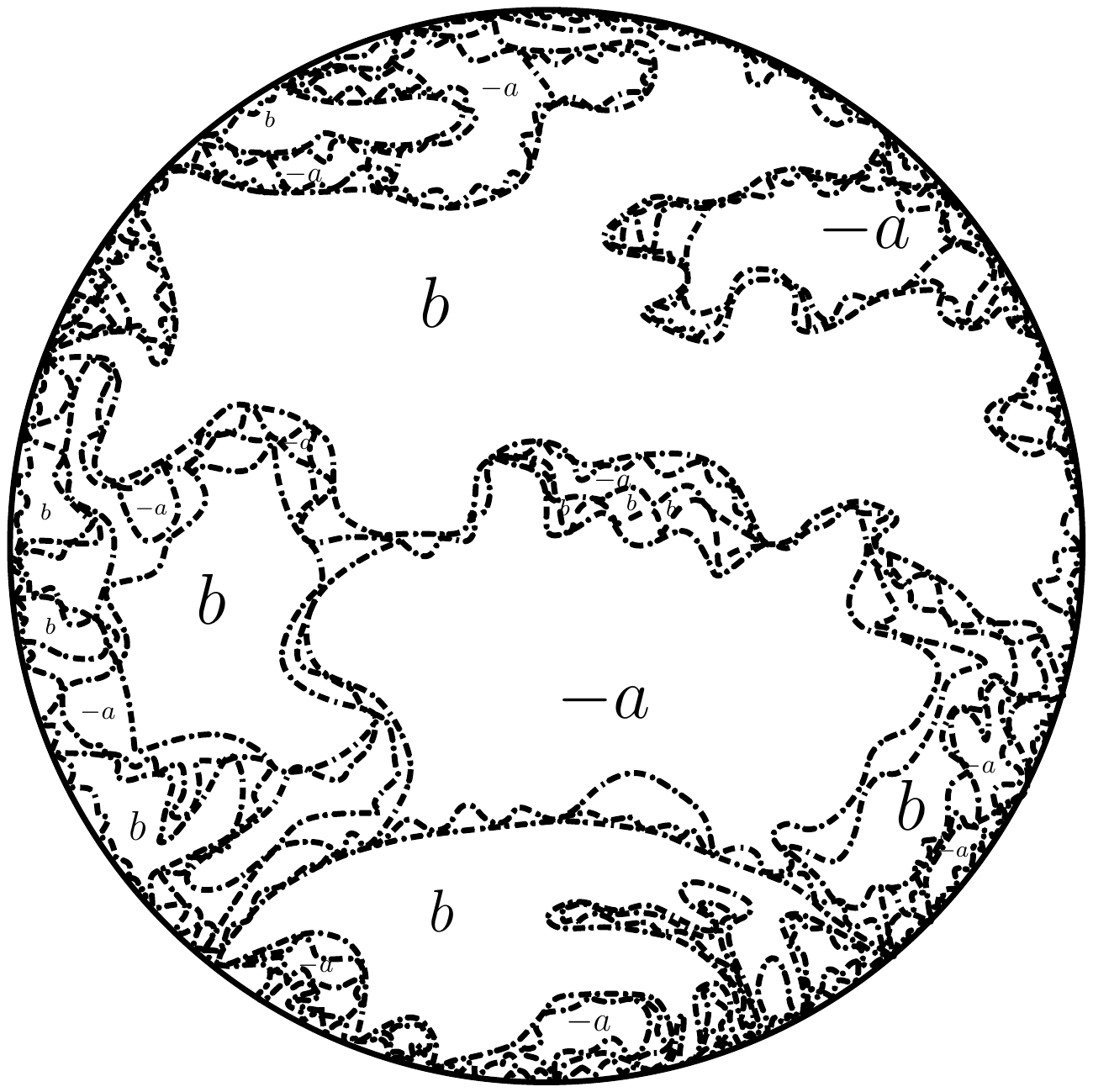}
	\includegraphics[scale=0.22]{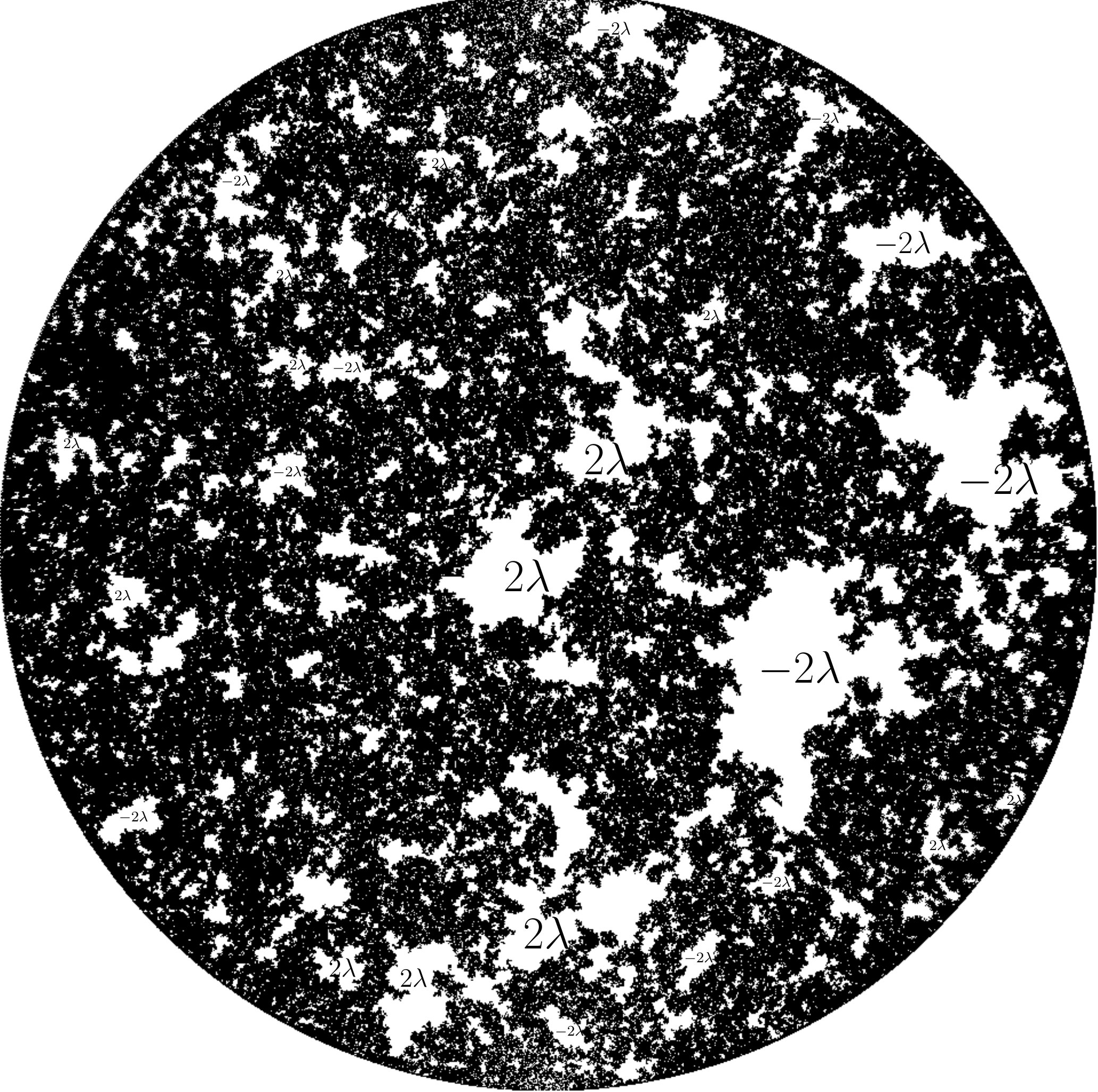}
	\caption {The three phases described in Theorem \ref{mainthm}. The left picture represents $\Aa_{b-2\lambda,b}$: two loops that intersect share a whole side. The middle pictures represents $\Aa_{-a,b}$ with $a\in(2\lambda-b,4\lambda-b)$: each loops intersects with infinitely many other loops, but no two loops share a side; also loops with the same label do not touch. The right picture is a simulation by D. Wilson of $\Aa_{-2\lambda,2\lambda}$, in which case all loops are pairwise disjoint.}
	\label {fig:mainthm}
\end{figure}

\subsection{Proof of Theorem \ref{mainthm}}

We will in sequence prove the parts (1), (2) and (3):
\subsubsection*{\textbf{Part (1): the ALE ($a+b=2\lambda$)}} We use the construction of the basic TVS given in Section \ref{Cons basic A}; in particular recall the notation $A^n$ from this Section - here $n$ refers to the $n-$th layer of level lines in the construction. By conformal invariance, we may assume that we are working in $\H$ and that the first level line is started from $0$ and targeted to $\infty$.
	
	Let us start by showing that all loops of $\A{a}$ that also belong to $A^2$ are connected via a finite path in $G_s$. We differentiate two types of loops: those which contain a segment joining $\R^{-}$ to $\R^{+}$, and those which touch either only $\R^+$ or $\R^-$. Notice that the loops of the second type are at a distance 1 from some loop of the first type. Thus, it suffices to prove that the loops of the first type are connected via a finite path only using loops of the first type, i.e., loops that will also belong to $G_s$. However, this fact follows from the fact that the level line is a.s. non-self-crossing, continuous up to its target point, and attains its target point almost surely (\cite{WaWu}).
	
	Now, notice that the rest follows inductively: indeed, any loop of $\A{a}$ that was not present at $A^n$, but is present at $A^{n+1}$ is side-connected to a loop that appears at level $A^n$. Thus, as any loop of $\A{a}$ appears at $A^N$ for some finite random $N$, we conclude that $G_s$ is connected.
	
	It also follows from the construction that $G_p(A^n) = G_s(A^n)$ and that any loops that share a segment have different labels.

\subsubsection*{\textbf{Part (2): the connected phase ($2\lambda<a+b<4\lambda$)} \label{subcritical}}
First notice that throughout this subsection it is sufficient to work in the case when $a<2\lambda$ (otherwise we can consider $\Aa_{-b,a}$).

From the construction of $A^n$ (Section \ref{Construction A}), it follows that in this phase two loops do not share sides. Indeed, we first construct $\A{a}$ and then iterate TVS in loops with value $-a+2\lambda$; as no loop of any TVS shares a segment with the boundary, the claim follows.

To show that $\Aa_{-a,b}$ is point-connected in this regime, it suffices to prove two things:
\begin{itemize}
	\item All loops are point-connected to a loop with label $-a$ that touches the boundary.
	\item All loops labelled $-a$ that touch the boundary are point-connected between each other.
\end{itemize} 

\begin{claim} \label{con to boun}
	Let $2\lambda \leq a+b<4\lambda$. Then almost surely for every loop of $\Aa_{-a,b}$ there exists a path in $G_p$ connecting it to a loop that touches the boundary and has label $-a$. 
\end{claim}
\begin{proof}
	When $a+b=2\lambda$, we are in the case of an ALE, and thus all loops are point-connected to the boundary. So suppose $2\lambda < a+b<4\lambda$. To deal with this case, recall the very last construction of Section \ref{Construction A} in the concrete case we have $n_2 = 2$. Thus $\Aa_{-a,b}$ can be constructed by starting from $\A{a}$ and then iterating ALEs inside the loops that do not yet have value $-a$ or $b$. As all loops of any ALE touch the boundary, and we iterate at every step only in the loops that don't have value $-a$ or $b$, we see that any loop constructed at some finite step $n$ intersects a loop constructed at step $n-1$. Thus, any loop constructed at step $n$ is point-connected to a boundary-touching loop via a path of (side-length) $n$. As any loop is constructed at some finite random step $N$, the claim follows.
\end{proof}

We now show that any two boundary touching loops with label $-a$ are point-connected. Note that, thanks to Lemma \ref{keepingloops} all boundary-touching loops of $\Aa_{-a,b}$ with label $-a$ are also loops of $\A{a}$ which is point-connected. Thus, it suffices to prove that any two loops with label $-a$ intersecting the boundary and at a distance $2$ in $\A{a}$ are at a finite distance in $G_p(\Aa_{-a,b})$. 

But (as in the previous claim) $\Aa_{-a,b}$ can be constructed by first exploring $\A{a}$ and then further exploring $\Aa_{-2\lambda, b+a -2\lambda}$ inside the loops with label $-a + 2\lambda$. Moreover, any two loops of $\A{a}$ labelled $-a$, and that are at distance $2$ in $\A{a}$ are side-connected to a loop of label $-a + 2\lambda$ of $\A{a}$. In particular, to show that these two loops are at finite distance in $G_p(\Aa_{-a,b})$, it suffices to prove the following claim: 

\begin{claim}\label{btb}
	Let $2\lambda\leq a'+b' <4\lambda$. Then, for any two fixed intervals of the boundary, there is almost surely a point-connected path in $G_p(\Aa_{-a',b'})$ going from a loop touching one interval to another loop touching the other interval.
\end{claim}

For simplicity, assume that we work now in $\D$, we take again $a' = a$, $b' = b$ and fix two disjoint boundary arcs $I = [x,y]$ and $\tilde I = [\tilde x, \tilde y]$, where the arcs are taken in a counter-clockwise sense. 

\begin{proof}[Proof of the claim]
 The proof is again based on choosing a particular way of constructing $\Aa_{-a,b}$: we first construct a loop that touches $I$ and then, iteratively, at each step build a loop in the connected component containing $\tilde I$, that touches the loop of the previous step. We continue until a loop touches $\tilde I$ (see Figures \ref{Boundary connection 1} and \ref{Boundary connection 2}).
	
In this respect, consider $A^1:=\eta^1([0,\infty])$, where $\eta^1$ is a $(-a, -a +2\lambda)$ level line of from $y$ to $x$ and denote $I_0 = I$. Either
\begin{itemize}
\item $\eta^1$ hits $\tilde I$: then from the construction of $\Aa_{-a,b}$ in Section \ref{BPTVE} we see that a part of $\eta^1$ is contained in the boundary of a loop of $\Aa_{-a,b}$ labelled $-a$ intersecting both $I$ and $\tilde I$. More precisely, this part is given by $\eta^1([\tau^-,\tau^+])$, were $\tau^+$ is the first time where $\eta$ hits $\tilde I$ and $\tau^-$ is the largest time before $\tau^+$ where $\eta$ hits $I$.
\item $\eta^1$ disconnects $I$ from $\tilde I$ before hitting $\tilde I$: in this case, $\tilde I$ belongs to $\partial O$, where $O$ is a connected component of $D\backslash \eta^1$. By the fact that $\eta^1$ is non-self-crossing and ends at $y$, we know that $h_{\eta^1}$ restricted to $O$ has boundary condition $-a+2\lambda$ on $\eta^1\cap \partial O$. On the other hand, the segment $\eta^1\cap \partial O$ is also part of the boundary of another connected component $\tilde O$ of $D\backslash \eta^1$, that contains a part of $I$. Moreover, $h_{\eta^1}$ restricted to $\tilde O$ has boundary condition $-a$ on $\eta^1\cap \partial \tilde O$. By exploring additional $-a, -a+2\lambda$ level line in $\tilde O$, as in the construction of Section \ref{Cons basic A}, we can finish the construction of all loops with label $-a$ that intersect $I$ and $\partial O$ (see the dashed lines in the middle of Figure \ref{Boundary connection 1}).
\end{itemize}

In the latter case we continue the construction of a sequence of loops joining $I$ to $\tilde I$ recursively: if $A^n$ does not touch $\tilde I$, we continue the construction of $\Aa_{-a,b}$ in $O^n$, the connected component of $D\backslash A^n$ that contains $\tilde I$. We now take $I^n$ to be the subset $\partial O^n\cap A^n$. Notice that the extremal distance of $I^n$ to $\tilde I$ in $D \backslash A^n$ is larger than the extremal distance of $I^{n-1}$ to $\tilde I$ in $D \backslash A^{n-1}$. Denote $I^n = [x^n, y^n]$ in a counter-clockwise sense and depending on the parity of $n$, continue as follows (see Figures \ref{Boundary connection 1} and \ref{Boundary connection 2}):

	\begin{itemize}
		\item If $n$ is odd, $h_{A^n}$ is equal to $-a+2\lambda$ on $I^n$ and zero on $\partial O^n\backslash A^n$. Hence, as $b-2\lambda<-a+2\lambda<b$, from Lemma \ref{lemext} it follows that we can explore a $(b-2\lambda, b)$-level line $\eta^{n+1}$  of $\Gamma^{A^n}+h_{A^n}$ restricted to $O^n$ from $x^n$ to $y^n$. We now have two scenarios as above: either the level line intersects $\tilde I$, then we stop as above and observe that there is now a loop with label $b$ that joins $\tilde I$ and $I^n$. Otherwise we set	$A^{n+1}:=A^n\cup \eta^{n+1}([0,\infty])$ and note that the boundary values of $h_{A^{n+1}}$, restricted to closed connected component of $D\backslash A^{n+1}$ containing $\tilde I$, are equal to $b-2\lambda$ on the level line $\eta^{n+1} ([0,\infty])$. Moreover, on the other side of this level line segment we can, using additional level lines, again finish all loops labelled $b$ that intersect $I^n$ and $I^{n+1}$.
		\item If $n$ is even, $h_{A^n}$ restricted to $O$ has boundary values constant equal to $b-2\lambda$ on $I_n$ and zero elsewhere on $\partial O^n\backslash A^n$. Thus, it is possible to construct a $(-a, -a + 2\lambda)$ level line $\eta^{n+1}$ of $\Gamma^{A^n}$ from $y^n$ to $x^n$. Again, if $\eta^{n+1}$ intersects $\tilde I$ we stop. If not, we observe that the boundary values of $h_{A^{n+1}}$, restricted to closed connected component of $D\backslash A^{n+1}$ containing $\tilde I$, are equal to $-a+2\lambda$ in $\eta^n$. Moreover, on the other side of this level line segment we can, using additional level lines, again finish all loops labelled $-a$ that intersect $I^n$ and $I^{n+1}$.
	\end{itemize}
\begin{figure}[h!]
	\includegraphics[width=0.3\textwidth]{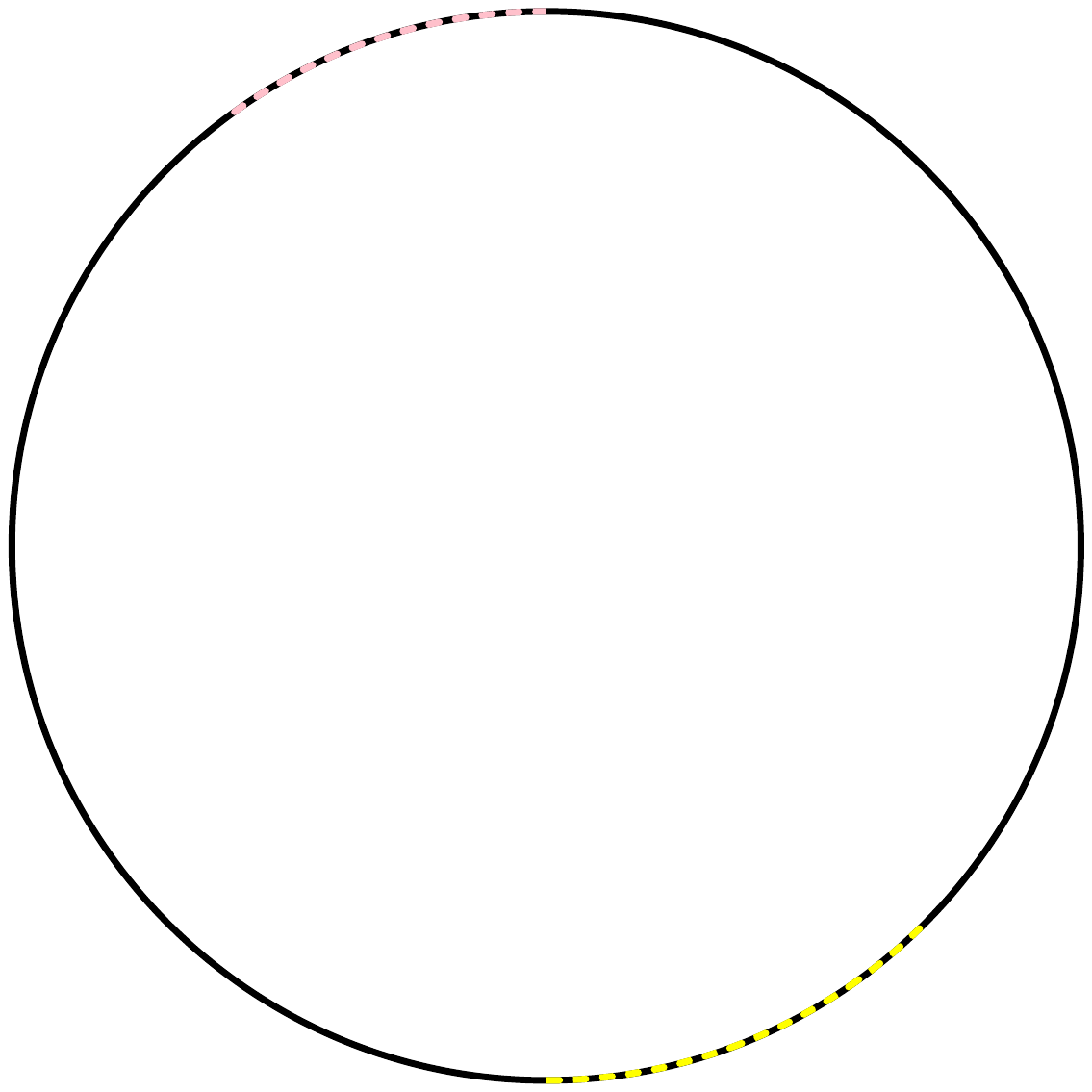}
	\includegraphics[width=0.3\textwidth]{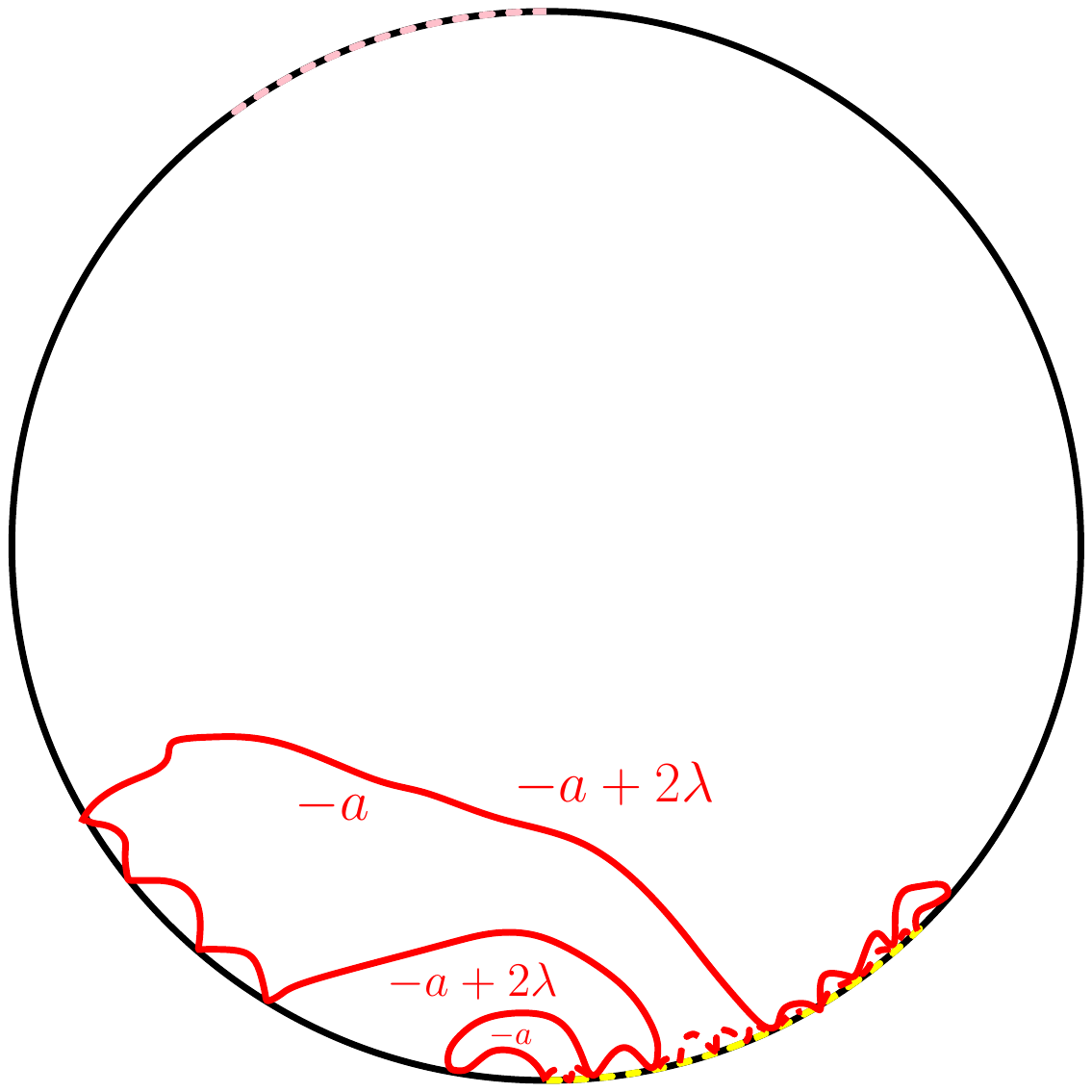}
	\includegraphics[width=0.3\textwidth]{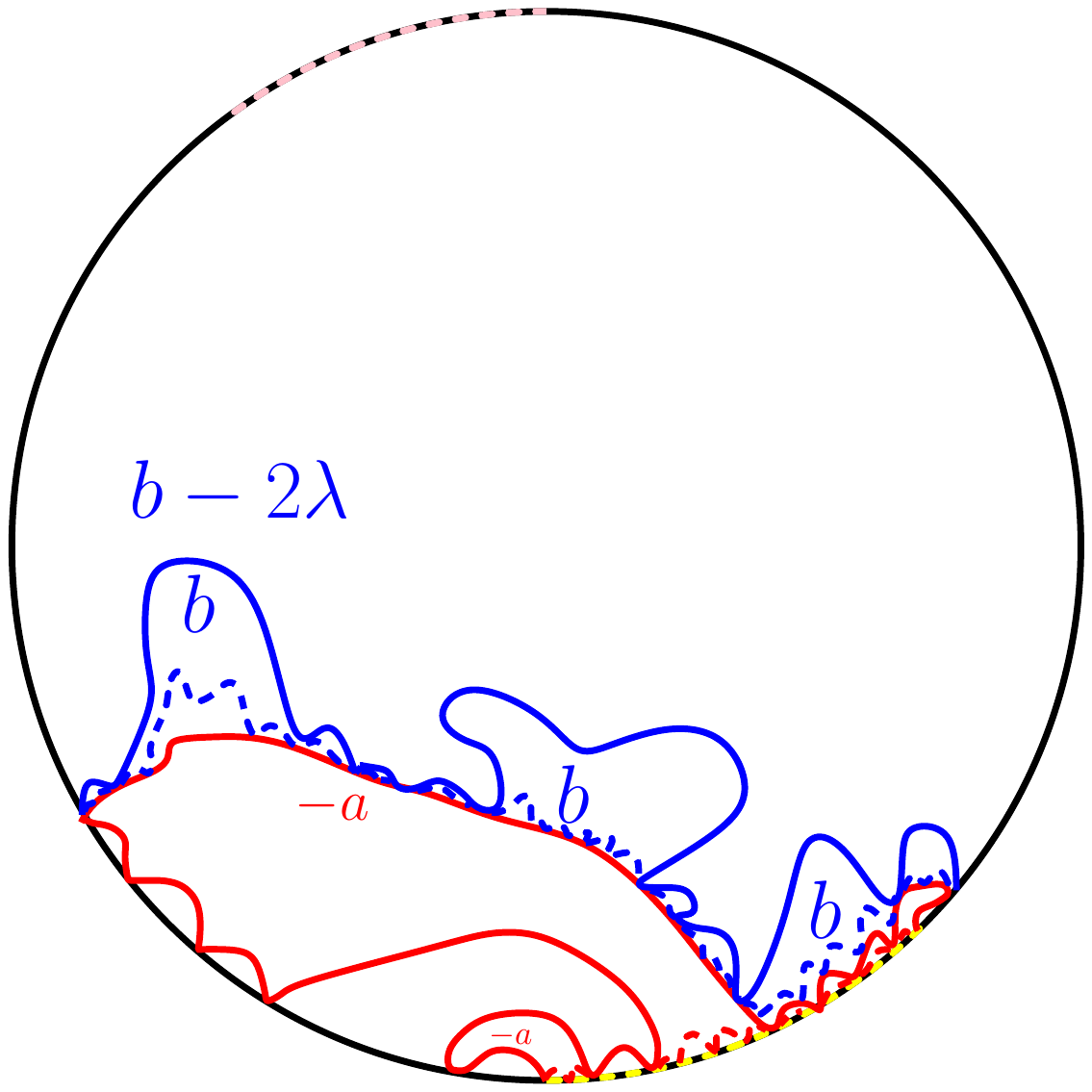}
	\caption{Scheme of the first two steps of the proof. We want to connect the bottom dashed segment with the upper dashed segment. In the Figure $a<2\lambda$ and $b\geq 2\lambda$.}
	\label{Boundary connection 1}
\end{figure}

We claim that this procedure stops at a finite (random) time $N$ almost surely. Indeed, by using conformal invariance of generalised level lines, we can map $O^n$ to the unit disk via $\phi^n$ such that $I^n$ maps to a fixed interval. As the extremal distance between $I^n$ and $\tilde I$ is decreasing, the $\phi^n(\tilde I)$ is increasing in length. As on the other hand the boundary conditions are equal on $I^n$ on even and odd steps separately, and equal to zero elsewhere on $\partial O^n$, we conclude that the probability of hitting $\tilde I$ before finishing at $y$ is increasing separately in even and odd steps (for $n \geq 3$). As this probability is non-zero to begin with (as SLE$_4(\rho_1, \rho_2)$ process hits any interval that it potentially could hit with positive probability), we see that $N$ is stochastically dominated by a geometric random variable of positive parameter $p$. 
		
As each $I^n$ is joined to $I$ via a path of $n$ point-connected loops of labels $-a$ or $b$, it just remains to see that we can finish the construction of $\Aa_{-a,b}$ (i.e. that these loops we constructed indeed are loops of $\Aa_{-a,b}$). Notice that after having finished the local set in the construction above, it remains to construct $\Aa_{-a,b}$ in simply connected components with boundary conditions whose values are piece-wise constant in $[-a,b]$ and change at most twice. Thus, the claim follows from Remark \ref{rem: 2 boundary values}.

\begin{figure}[h!]
	\includegraphics[width=0.3\textwidth]{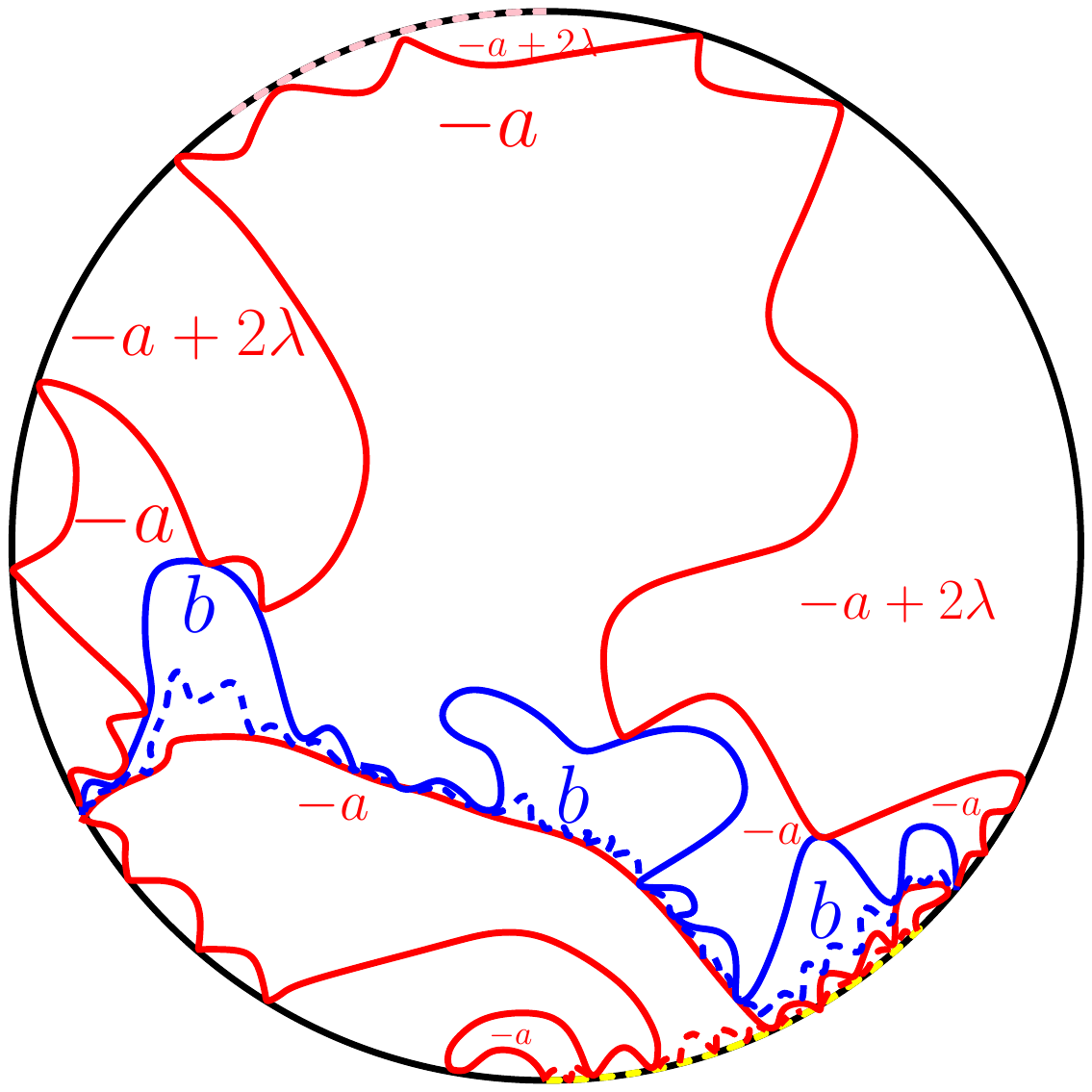}
	\includegraphics[width=0.3\textwidth]{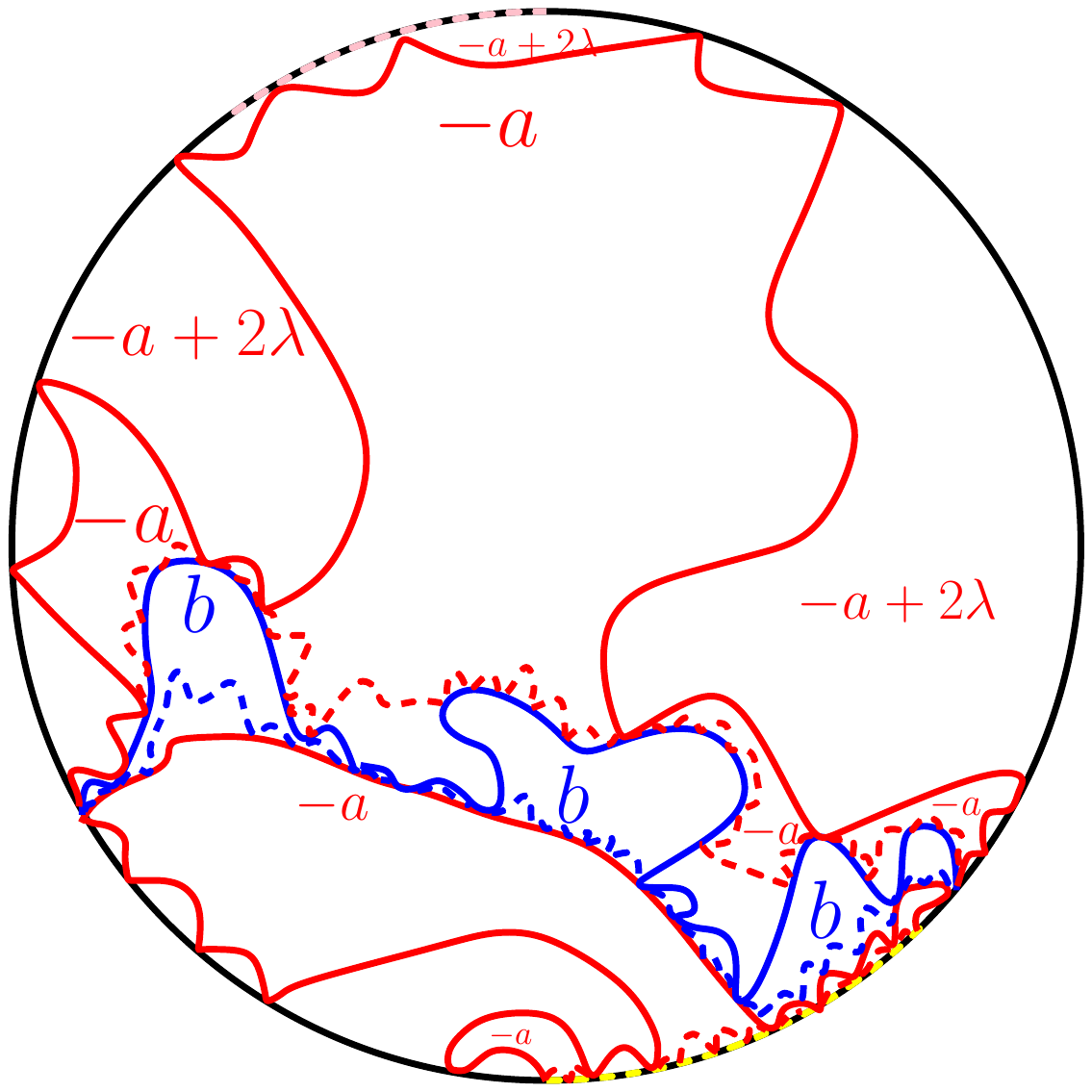}
	\includegraphics[width=0.3\textwidth]{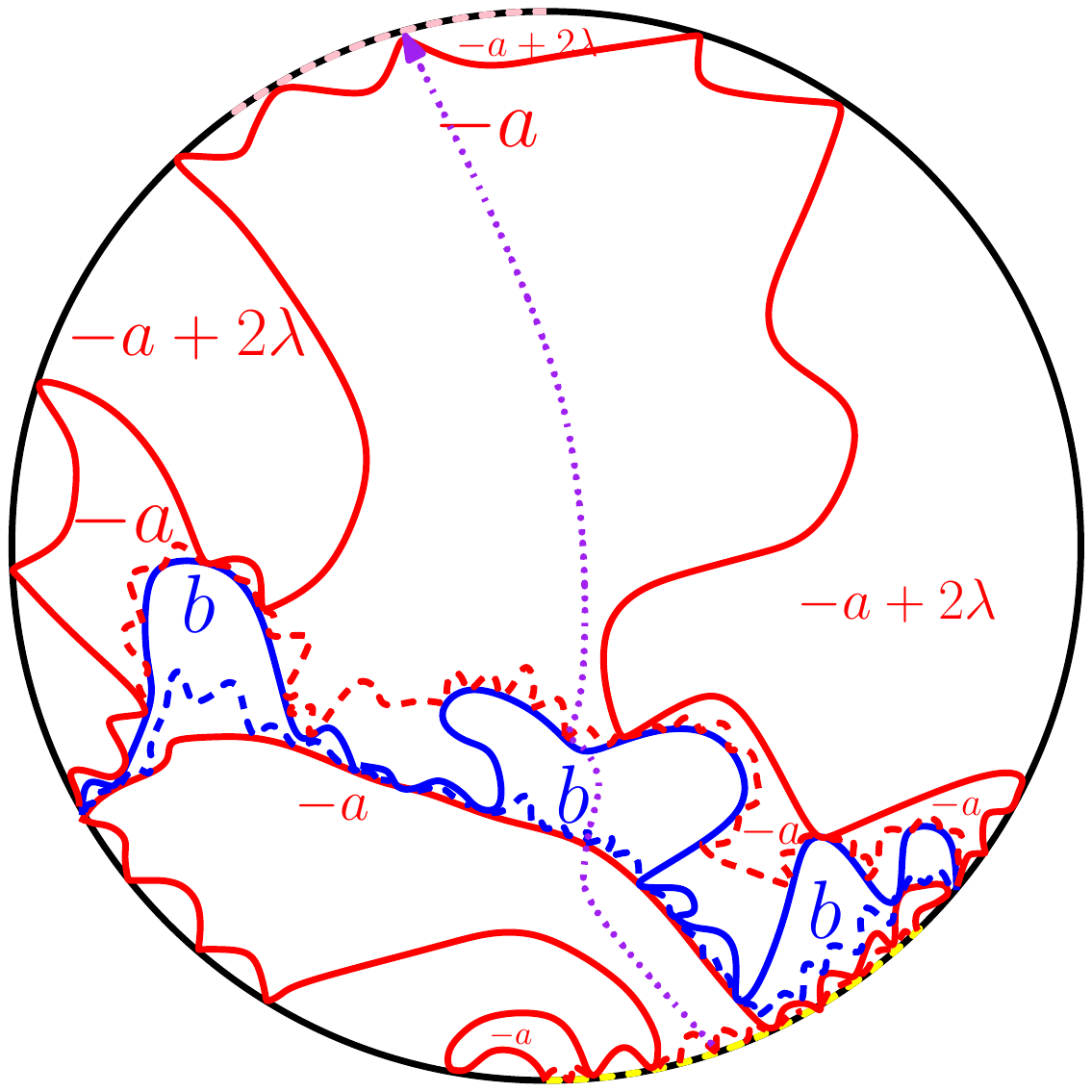}
	\caption{Scheme of the last steps of the proof. In the first image we show the first time a level line hit the targeted interval. In the second one we show how to complete the loops whose boundaries are the explored level lines. In the third image we show the path of loops from one segment to the other.}
	\label{Boundary connection 2}
\end{figure}
		\end{proof}

\subsubsection*{\textbf{Part (3): the disconnected case ($a+b\geq 4\lambda$)}} In the case of $a=b=2\lambda$, the claim follows from the fact that $\Aa_{-2\lambda,2\lambda}$ has the law of a CLE$_4$. Now consider $\Aa_{-a,-a+4\lambda}$ with $a<2\lambda$. We can construct it by first exploring $\A{a}$ and then inside connected components of $D\backslash \A{a}$ with the label $-a+2\lambda$ further exploring $\Aa_{-2\lambda,2\lambda} (\Gamma^{\A{a}},O)$ - the closed union of the explored sets gives precisely $A_{-a,-a+4\lambda}$. 
	
	We know that loops with the label $-a$ that also belong to $\A{a}$ do not touch each other. But all other loops come from exploring $\Aa_{-2\lambda,2\lambda}(\Gamma^{\A{a}},O)$ in the second step. These loops do not touch each other, nor the boundary of $O$, i.e. the loops with label $-a$.
	
	For general $a+b\geq 4\lambda$, we conclude from the previous case and the montonicity: any loop of $\Aa_{-a,b}$ is contained in the interior of some loop of $\Aa_{-a',-a'+4\lambda}$ for some $0<a'<4\lambda$. 
	
	Similarly, we can draw the following Corollary:
	
	\begin{cor}\label{inloops}
	Suppose $a \geq 2\lambda$. Then the loops of $\Aa_{-a,b}$ with the label $-a$ are pairwise disjoint. Similarly if $b \geq 2\lambda$, then all loops with the label $b$ are pairwise disjoint.
	\end{cor}
	
	\begin{proof}
	From Lemma \ref{keepingloops}, we know that loops with label $-a$ in $\Aa_{-a,b}$ remain also loops with label $-a$ of $\Aa_{-a,b+2\lambda}$. But we know that any two loops of the latter are disjoint. 	
	\end{proof}

\subsubsection*{\textbf{Part (4): loops with the same label do not touch.}} As this is trivially true in the case $a+b \geq 4\lambda$, we suppose $a +b < 4\lambda$. 

First, note that for the ALE ($a+b = 2\lambda$) this follows from the construction using level lines. Indeed, the labels on the two sides of a level line are different, so two loops could only touch at the endpoints of SLE$_4(\rho_1, \rho_2)$ excursions constructing them. However, any two SLE$_4(\rho_1, \rho_2)$ excursions away from the boundary are disjoint as they correspond to excursions of Bessel-type of processes with dimension strictly less than $2$. Moreover, for the same reason, none of the excursions also touches the starting point nor the endpoint of the process.

Suppose now that $2\lambda < a+b < 4\lambda$ and assume WLoG that $a < 2\lambda$. Let us first show that no two loops with label $-a$ touch each other. As in the proof of Lemma \ref{keepingloops}, we can construct $\Aa_{-a,b}$ by:
\begin{enumerate}[(1)]
	\item Exploring $\Aa_{-a,-a+2\lambda}$.
	\item Exploring $\Aa_{-2\lambda, b+a-2\lambda}$ inside the connected components of $D\backslash \Aa_{-a,-a+2\lambda}$ with the label $-a+2\lambda$.
\end{enumerate}
 Now, we know that no two loops labelled $-a$ from step (1) can touch each other by the previous paragraph. Also, by Lemma \ref{keepingloops} no loop labelled$-a$ constructed in step (2) touches the loops with the label $-a$ of step (1). Finally, by Corollary \ref{inloops} loops labelled $-a$ constructed in step (2) are also pairwise disjoint.

Consider now the loops with label $b$. If $b < 2\lambda$, then we can argue as just above. If $b \geq 2\lambda$, two loops with label $b$ do not touch each other by Corollary \ref{inloops}.

\begin{rem}\label{rem: gen 4.1} In Remark \ref{rem: 2 boundary values} we discussed TVS with piece-wise constant boundary condition that changes in two points. It is not hard to convince oneself that Theorem \ref{mainthm} also holds in this setting. In fact, with minor modifications in the proof, one can also prove it in the more general setting of a GFF with a piecewise constant boundary condition changing finitely many times, a setting used in \cite{ALS1}.
\end{rem}

\subsection{A corollary: the SLE$_4$ fan}
Before presenting the main consequences of Theorem \ref{mainthm} let us discuss a simple consequence about the SLE$_4$ fan. The SLE$_4$ fan, introduced for other values of $\kappa$ in \cite{MS1}, is roughly the union of all possible $(-a,-a+2\lambda)$ level lines with $a \in (0,2\lambda)$ going from $x$ to $y$ coupled with the same underlying GFF. To make this precise one takes the closed union over any dense countable subset of $(0,2\lambda)$:

\begin{defn}[SLE$_4$ fan]
Let $\Gamma$ be a GFF and fix $x,y \in \partial \D$. Then the SLE$_4$ fan is defined as
	\[	
	\mathbf F(x,y):=\overline{\bigcup_{a\in (0, 2\lambda) \cap \Q} \eta^{a,x,y}}, \]
	where $\eta^{a,x,y}$ is the $(-a,-a+2\lambda)$-level line of $\Gamma $ going from $x$ to $y$. 
\end{defn}

It can be seen that the resulting set does not depend on the underlying choice of the countable dense set of points in the following sense: for any two such choices, the corresponding SLE$_4$ fan is the same. Indeed, this just follows from the monotonicity of the level lines, Lemma 7.2 in \cite{PW}: for any fixed $a$ and $a_n \searrow a$, we have that a.s. the $(-a_n,-a_n + 2\lambda)$ level lines converge to the $(-a,-a+2\lambda)$ level line in the Hausdorff topology. This monotonicity allows to define $a\mapsto \eta^{a,x,y}$ over the whole parameter interval $(0,2\lambda)$. However, this mapping cannot be continuous, as can for example be seen from the fact that the whole union of $(-a,-a+2\lambda)$ level lines over $a \in (0,2\lambda)$ is contained in $\Aa_{-2\lambda, 2\lambda}$. Thus the SLE$_4$ fan, denoted $\mathbf F(x,y)$, is a fractal set whose complement consists of simply-connected open sets. Again, it is natural to ask whether the connected components of the complement are point-connected in the same sense as above. 

\begin{cor} \label{Cor fan}
Let $\Gamma$ be a GFF and fix $x,y \in \partial \D$.	The graph $G_p(\mathbf F(x,y))$ defined as before is connected.
\end{cor}

Let us first see that no connected component of the complement of the SLE$_4$ fan shares a boundary segment with the boundary of the domain:

\begin{claim}\label{claim:: no boundary}
	Let $\Gamma$ be a GFF and fix $x,y \in \partial \D$. Then a.s. no connected component of the complement of the SLE$_4$ fan $\mathbf F(x,y)$ contains a boundary arc of $\partial \D$.
\end{claim} 

\begin{proof}
WLOG let $x = -i$ and $y = i$. It suffices to prove the that for any $z \in \D$, there is some random $\epsilon(z)$, such that the point $z$ is to the right of the $(-2\lambda+\epsilon,\epsilon)$ level line and to the left of the $(-\epsilon, 2\lambda - \epsilon)$ level line. To show this, we argue that $(-2\lambda+\eps,\eps)$ level line from $x$ to $y$ converges to the clock-wise arc from $-i$ to $i$, and that the $(-\epsilon, 2\lambda -\epsilon)$ level line converges to the counter-clock-wise arc from $-i$ to $i$.

We know that the $(-2\lambda+\epsilon,\epsilon)$ level line is contained in $\Aa_{-2\lambda+\epsilon,\epsilon}$. But $\Aa_{-2\lambda+\epsilon,\epsilon}$ converges to $\partial \D$ as $\epsilon \to 0$: indeed, by Lemma \ref{BPLS} it converges to a local set $A$, that is moreover a BTLS (i.e. a thin local set with bounded $h_A$). As $h_A$ is however non-negative, this set has to be empty by Lemma 9 in \cite{ASW} (using $A=\partial \D$, $B=A$ and $k=0$). Thus, we know that the $(-2\lambda+\epsilon,\epsilon)$ level line converges to a part of the boundary in Hausdorff distance. As this part of the boundary is always to the left of the $(-\lambda, \lambda)$ level line, it has to converge to the clock-wise arc from $-i$ to $i$. Similarly the $(-\epsilon, 2\lambda -\epsilon)$ level line converges to the counter-clock-wise arc from $-i$ to $i$ and the claim follows.
\end{proof}

We now prove the corollary:

\begin{proof}
	
	WLOG let $x = -i$ and $y = i$. Define the $\epsilon$-SLE$_4$ fan:
	\[
	\mathbf F^\epsilon(x,y):=\overline{\bigcup_{a\in (\eps, 2\lambda) \cap \Q} \eta^{a,x,y}}.\]
	Let us note that for all $\epsilon >0$, $\mathbf F^\epsilon(x,y)\subseteq \Aa_{-2\lambda,2\lambda-\eps}$. This is because, for any $\epsilon \leq c < 2\lambda$, the $(-c,2\lambda -c)$ level line is contained in $\Aa_{-c, 2\lambda -c} \subseteq \Aa_{-2\lambda, 2\lambda-\eps}$ by monotonicity. Moreover, for all $\epsilon >0$, a.s. $\mathbf F^\epsilon(x,y)$ remains to the left of $(- \epsilon,2\lambda - \epsilon)$-level line $\eta^{2\lambda - \epsilon,x,y}$. 
	
	Now, consider a connected component $O$ of $\D \backslash \eta^{\epsilon,x,y}$. For the ease of notation set $u_\epsilon=h_{\Aa_{\eta^{\epsilon,x,y}}}$. We claim that $\mathbf F^\epsilon(x,y) \cap O$ is contained in  $\Aa_{-2\lambda,2\lambda-\epsilon}^{u_\epsilon}(\Gamma^{\Aa_{\eta^{\epsilon,x,y}}},O)$, the two-valued set for a GFF with piece-wise constant boundary condition changing in two points, as considered in Remark \ref{rem: 2 boundary values}. Indeed, notice that $\Aa_{-2\lambda,2\lambda-\eps}$ (on the whole domain) can be constructed in two steps: we first explore $\eta^{\epsilon,x,y}$, and then in each connected component $O$ of $\D \backslash \eta^{\epsilon,x,y}$ we construct the $\Aa_{-2\lambda,2\lambda-\epsilon}^{u_\epsilon}(\Gamma^{\Aa_{\eta^{\epsilon,x,y}}},O)$. But $\Aa_{-2\lambda+\epsilon,2\lambda} \cap O$ equals $\Aa^{u_\epsilon}_{-2\lambda,2\lambda-\epsilon}(\Gamma^{\Aa_{\eta^{\epsilon,x,y}}},O)$, and we conclude this claim.
	
	Further, any two loops $\ell_1,\ell_2$ of $\mathbf F^\epsilon(x,y)$ that are inside the same component $O$ of $\D \backslash \eta^{\epsilon,x,y}$ always surround some loops of $\Aa_{-2\lambda,2\lambda-\epsilon}^{u_\epsilon}(\Gamma^{\Aa_{\eta^{\epsilon,x,y}}},O)$. From Theorem \ref{mainthm} (see also Remark \ref{rem: gen 4.1} as we are in the setting where the boundary values are piece-wise constant and change twice) it follows that the loops of $\Aa_{-2\lambda,2\lambda-\epsilon}^{u_\epsilon}(\Gamma^{\Aa_{\eta^{\epsilon,x,y}}},O)$ are point-connected. Thus, 
	we conclude that $\ell_1$ and $\ell_2$ are point-connected via loops of $\mathbf F^\epsilon(x,y)$ that remain inside $O$.
	
	To finish the proof it suffices to observe that for any two given loops of $\mathbf F(x,y)$ (i.e. loops around some fixed points $z$ and $w$), there exists (a random) $\epsilon$ such that both of them are loops of $\mathbf F^{\epsilon}(x,y)$ and, moreover, are contained in the same connected component $O$ of $\D \backslash \eta^{\epsilon,x,y}$ that lies to the left of $\eta^{\epsilon,x,y}$. This just follows from the proof of Claim \ref{claim:: no boundary}.
\end{proof}

\section{Measurability of labels for two-valued local sets}
We now use the properties of the graph $G_p$ to study the following question: \textit{can the labels of $\Aa_{-a,b}$ be recovered just from the geometry of $\Aa_{-a,b}$?}. The answer is given by the following proposition:

	\begin{prop}\label{mes label}
		Let $a,b>0$ and consider the local set coupling $(\Gamma, \Aa_{-a,b}, h_{\Aa_{-a,b}})$:
		\begin{itemize}
			\item If $2\lambda \leq a+b <4 \lambda$ and $a\neq b$, then the labels of $\Aa_{-a,b}$ are a measurable function of the set $\Aa_{-a,b}$ (or in other words $h_{\Aa_{-a,b}}$ is measurable w.r.t. $\Aa_{-a,b}$).
			\item If $\lambda \leq a <2\lambda$, the labels of $\Aa_{-a,a}$ are a measurable function of the set $\Aa_{-a,a}$ and the label of the loop surrounding $0$.
			\item If $a+b\geq 4\lambda$, the labels of $\Aa_{-a,b}$ cannot be recovered only knowing $\Aa_{-a,b}$ and any finite number of labels.
		\end{itemize}
	\end{prop}

We now prove this proposition and then describe the explicit law of the labels in the case $a+b = 4\lambda$.

\subsubsection*{Connected case with $a\neq b$ $(2\lambda \leq a+b<4\lambda)$.} We may again assume that $a<2\lambda$. From Corollary \ref{cor bd} we know the label of any loop touching the boundary with Hausdorff dimension $1-(2-a/\lambda)^2/4$ has the label $-a$. Moreover by Lemma \ref{keepingloops} there is some loop $\ell_a$ with label $-a$ that touches the boundary. But now from Theorem \ref{mainthm} it follows that any loop $\ell$ is point-connected to $\ell_a$. The label of $\ell$ is $-a$ if the graph distance in $G_p$ between $\ell$ and $\ell_a$ is even, and $b$ if it is odd.

\subsubsection*{Connected case with case with $a=b$ $(\lambda\leq a< 2\lambda)$.} The previous proof fails as loops of label $\pm$ touch the boundary with the same Hausdorff dimension. However, as soon as we know the label of the loop surrounding $0$ we can again similarly use the connectedness of $G_p$ to deduce the claim. 

\subsubsection*{Disconnected case $(a+b\geq 4\lambda)$.}
This case needs a bit more care. The idea is to use the fact that conditionally on the set $\Aa_{-2\lambda,2\lambda}$, the labels are given by i.i.d. fair coin tosses (Proposition \ref{BPCLE} (3)).

	In this respect, we show that there are two GFF $\Gamma$ and $\tilde \Gamma$ such that a.s.
	\begin{enumerate}[(a)]
		\item $\Aa_{-a,b}(\Gamma)=\Aa_{-a,b}(\tilde \Gamma)$.
		\item There are infinitely many loops of $\Aa_{-a,b}(\Gamma)$, such that their label under $\Gamma$ is different than their label under $\tilde \Gamma$.
		\item Any finite subset of labels has the same value for $\Gamma$ and $\tilde \Gamma$ with positive probability.
	\end{enumerate} Note that this implies the statement, as conditionally on $\Aa_{-a,b}$ and any finite subset of labels, the (conditional) law of the rest of the labels is non-trivial.
	
	First let us construct this coupling when $a=b\geq 2\lambda$. We first sample a GFF $\Gamma$ and then explore $\Aa_{-a,a}(\Gamma)$. We do this in two steps:
		\begin{enumerate}
		\item We explore $\Aa_{-2\lambda,2\lambda}(\Gamma)$. 
		\item  We explore $\Aa_{-a+2\lambda,a+2\lambda}(\Gamma^{\Aa_{-2\lambda,2\lambda}},O)$ in all connected components $O$ of $D\backslash \Aa_{-2\lambda,2\lambda}$ labelled $-2\lambda$, and we explore $\Aa_{-a-2\lambda,a-2\lambda}(\Gamma^{\Aa_{-2\lambda,2\lambda}},O)$ in all connected components $O$ with label $2\lambda$. 
	\end{enumerate}
Note that in this construction the law of the TVS being explored in each component $O$ is the same as the law of $-\Gamma$ and $\Gamma$ agree. 

	Let us now construct $\tilde \Gamma$. We start by constructing $\Aa_{-a,a}(\tilde \Gamma)$ and its labels. First, and as an equivalent of (1), set $\Aa_{-2\lambda,2\lambda}(\tilde \Gamma)=\Aa_{-2\lambda,2\lambda}$ but resample the labels of the loops of $\Aa_{-2\lambda,2\lambda}(\tilde \Gamma)$ independently, by tossing an independent fair coin for each loop. 
	
There are now two types of connected components of  $D\backslash\Aa_{-2\lambda,2\lambda}(\tilde \Gamma)$: those where the new labels agree with the ones sampled before for $\Aa_{-2\lambda,2\lambda}(\Gamma)$ and those where the new labels differ. In the components of the first type, we do the equivalent of (2) using exactly the same set and labels as the ones used for $\Gamma$. In other words, in loops labelled $-2\lambda$ we set $\Aa_{-a+2\lambda,a+2\lambda}(\tilde \Gamma^{\Aa_{-2\lambda,2\lambda}},O)= \Aa_{-a+2\lambda,a+2\lambda}(\Gamma^{\Aa_{-2\lambda,2\lambda}},O)$ and we keep the same labels; analogously for the components of the first type, but with label $2\lambda$. In those connected components $O$ where the sign changed we use again the same set, but change the sign of all the labels inside. We have thus constructed $\tilde h_{\Aa_{-a,a}}$.
	
	Finally, define $\tilde \Gamma^{\Aa_{-a,a}}$ in some way, say by setting it equal to $\Gamma^{\Aa_{-a,a}}$. Due to the equality in law noted at the beginning, $\tilde \Gamma$ has the law of a GFF. Additionally, it is clear to see that $\Gamma$ and $\tilde \Gamma$ satisfy the desired properties. 
	
	For the general case $a \neq b$ assume WLoG that $a\leq b$ and define $m=(b-a)/2>0$.  To construct $\Aa_{-a,b}$ we first explore $\Aa_{-a,m}$.
	To finish $\Aa_{-a,b}$ it then remains to explore  $\Aa_{-a-m,b-m}$, i.e. $\Aa_{-(b+a)/2,(b+a)/2}$, inside the loops labelled $m$. But observe that this is again a two-valued set of the form $\Aa_{-a',a'}$. Thus by doing the same coupling as above for $\Aa_{-(b+a)/2,(b+a)/2}$, we deduce the claim.

\subsubsection*{Law of the labels conditioned on $\Aa_{-a,b}$ in the critical case $a+b = 4\lambda$} In the case $a+b=4\lambda$, one can moreover precisely describe the law of the labels:

\begin{prop}
	Let $0<a\leq 2\lambda$. Then, the law of the labels of $\Aa_{-a,-a+4\lambda}$ given $\Aa_{-a,-a+4\lambda}$ is the following:
	\begin{itemize}
		\item The loops touching the boundary are labelled $-a$.
		\item For each loop that does not touch the boundary we toss independent fair coins to decide whether the label is equal to $-a$ or $-a + 4\lambda$. 
	\end{itemize}
\end{prop} 
\begin{proof}
		Note that the result holds for $\Aa_{-2\lambda,2\lambda}$ due to the fact that no loop touch the boundary (Proposition \ref{BPCLE} (ii)). When $a < 2\lambda$, by Lemma \ref{keepingloops}, as $-a+4\lambda\geq 2\lambda$, the loops that touch the boundary are labelled $-a$. The union of the loops touching the boundary is $\A{a}$ by Remark \ref{ALE}. All the other loops are constructed by exploring $\Aa_{-2\lambda,2\lambda}(\Gamma^{\A{a}},O)$ inside any connected component $O$ of $D\backslash \A{a}$ labelled $-a+2\lambda$. But we know that the labels of $\Aa_{-2\lambda, 2\lambda}$ are given by independent fair coin tosses.
\end{proof}

\subsubsection*{Non-independence of labels $\Aa_{-a,b}$ when $b\neq 2\lambda$}

In this section we prove that $\Aa_{-2\lambda,2\lambda}$ is the only TVS for which the labels of the loops are i.i.d. conditioned on the set itself:

\begin{prop}\label{non-independence}
Let $\Gamma$ be a GFF on $\D$. Let moreover $a,b>0$ be such that $a+b> 4\lambda$. Then, the labels of $\Aa_{-a,b}$ conditioned on the underlying set are not independent.
\end{prop}

\begin{rem}
In fact Proposition \ref{mes label} shows directly that in the case $a+b\leq 4\lambda$, a weaker version of the above proposition is true, where we replace ``independent'' by ``i.i.d.''. Indeed, recall in that case $a+b<4\lambda$ the labels are determined by the set, thus they are trivially conditionally independent. And in the case $a+b=4\lambda$ with $a\neq b$, conditionally on $\Aa_{-a,b}$ the labels are independent, but the loops touching the boundary are determined by the set. In both cases conditioned on the underlying set the labels are however not i.i.d.
\end{rem}

The idea of the proof is to show that - in a certain weak sense - the outer boundary conditions of the loop of $\Aa_{-a,b}$ surrounding $0$ are $-a+2\lambda$, if its label is $-a$, and $b-2\lambda$ if its label is $b$. More precisely, we will show that the averages of the field over tinier and tinier regions around the smallest intersection point of $\R^+$ and the loop converge to either $-a+2\lambda$ or $b-2\lambda$. This choice of point is somewhat arbitrary, but makes the proof technically simpler - the same proof would for example also work for a random point according the harmonic measure on the loop seen from zero.

To define this average, recall that the loops of TVS are Jordan curves, yet they can be relatively rough. Thus it is easier to define averages under a conformal image. In this respect, for a Jordan curve $\ell$ surrounding $0$ and at positive distance from $0$, consider the conformal map $\varphi$ mapping $\ell$ to the unit circle and its outside to the inside of the unit disk, such that minimal intersections points in norm of the coordinate axes $-iR^+, R^+, iR^+$ with $\ell$ are mapped to $-i$, $1$ and $i$ respectively. Let further $L(\theta,\epsilon)$ denote the set of points $z \in \D$ such that $|z|\in [1-\epsilon,1]$ and $\arg(z)\in(-\theta,\theta)$. For some integrable function $f$, we now define the averages 
\[\alpha_n(f,\ell) = \frac{1}{Leb(L(n^{-1},2^{-n}))}\int_{L(n^{-1},2^{-n})}f(\varphi^{-1}(z))dz.\]
Here the exact choices of $\theta$ and $\epsilon$ are not so important, $\theta$ just has to go to zero sufficiently slowly w.r.t. $\epsilon$. The key lemma is as follows: 
\begin{lemma} \label{Recover O}
	Let $\Gamma$ be a GFF in $\D$, $a+b\geq 2\lambda$ and $\hat \ell$ be the loop of $\Aa_{-a,b}$ surrounding $0$. Then, almost surely, as $n\to \infty$ we have that
	\[
		\alpha_n(h_{\Aa_{-a,b}},\hat \ell) \to \left\{\begin{array}{l l}
		-a+2\lambda & \text{if the label of $\hat \ell$ is} -a,\\
		b-2\lambda & \text{if the label of $\hat \ell$ is } b.
		\end{array} \right.\]
\end{lemma}
Proposition \ref{non-independence} follows directly from this lemma, as given the set $\Aa_{-a,b}$ and labels of all loops not surrounding $0$, we can calculate $h_{\Aa_{-a,b}}$ outside of $\hat \ell$ and thus $\alpha_n$. We thus conclude that the label of the loop surrounding $0$ is measurable w.r.t. $\Aa_{-a,b}$ and the labels of the loops not surrounding $0$. 

\begin{proof}[Proof of Lemma \ref{Recover O}]
Assume WLOG that the label of $\hat \ell$ is $-a$.

First let us see that, up to a technical claim, it is enough to construct a BTLS $A$ such that the loop of $A$ surrounding $0$ is exactly $\hat \ell$, and under the assumption that the label of $\hat \ell$ is $-a$, almost surely $\alpha_n(h_A) \to 2\lambda-a$. 
Indeed, given such a BTLS $A$, we write $\Gamma = \Gamma^A + h_A$, where $\Gamma^A$ has the law of the GFF in $\D \backslash A$ and at the same time $\Gamma = \Gamma^{\Aa_{-a,b}} + h_{\Aa_{-a,b}}$, where $\Gamma^{\Aa_{-a,b}}$ has the law of the GFF in $\D \backslash \Aa_{-a,b}$. Assume for now the following claim about the convergence of GFF averages:
\begin{claim}\label{claim:tec}
Almost surely $\alpha_n(\Gamma^A,\hat \ell) \to 0$ and $\alpha_n(\Gamma^{\Aa_{-a,b}},\hat\ell) \to 0$ as $n \to \infty$.
\end{claim}
Given this claim, we have that 
\[\lim_{n\to \infty} \alpha_n(h_{\Aa_{-a,b}},\hat \ell) = \lim_{n\to \infty} \alpha_n(\Gamma) = \lim_{n\to \infty} \alpha_n(h_A)\] 
and thus it indeed suffices to show that $\alpha_n(h_A) \to 2\lambda-a$.
	
Let us thus now see how to construct such an $A$ and then finally prove Claim \ref{claim:tec}. Note that when $a+b = 2\lambda$, we can just take $A=\A{a}$, because by construction of $\A{a}$, $\hat \ell$ almost surely does not hit $1$. 

Assume first that $a\geq2\lambda$. To construct $A$  first define $A^1$ as $\Aa_{-a+\lambda,b}$. As by assumption the loop of $\Aa_{-a,b}$ around $0$ has label $-a$, and as moreover $\Aa_{-a+\lambda,b} \subset \Aa_{-a,b}$, we know that the label of the loop of $\Aa_{-a+\lambda,b}$ surrounding $0$ will be $-a+\lambda$. Then, iterate for $n\in \N$ as follows:
	\begin{itemize}
		\item Inside the connected component $O$ of $D\backslash A^{n}$ containing $0$ explore $\Aa_{-\lambda,\lambda}$ of $\Gamma^{A^{n}}$ and call the union of $A^n$ and the explored set $\tilde A$. If the connected component $O$ of $\tilde A$ containing $0$ has label $-a$ (which happens with probability $1/2$) we stop the procedure and call $A=\tilde A$. Otherwise, the label of this loop is $-a+2\lambda$. In this case we define $A^{n+1}$ as the union between $\tilde A$ and $\Aa_{-\lambda,b+a-2\lambda}(\Gamma^{A^n},O)$. Again, as in the case of $A^1$ above, we know that the label of the loop of $A^{n+1}$ containing $0$ cannot be $b$, and is thus $-a+\lambda$ (see Figure \ref{fig::dependence}).
	\end{itemize}
\begin{figure}[h!]
	\centering
	\includegraphics[width=0.3\textwidth]{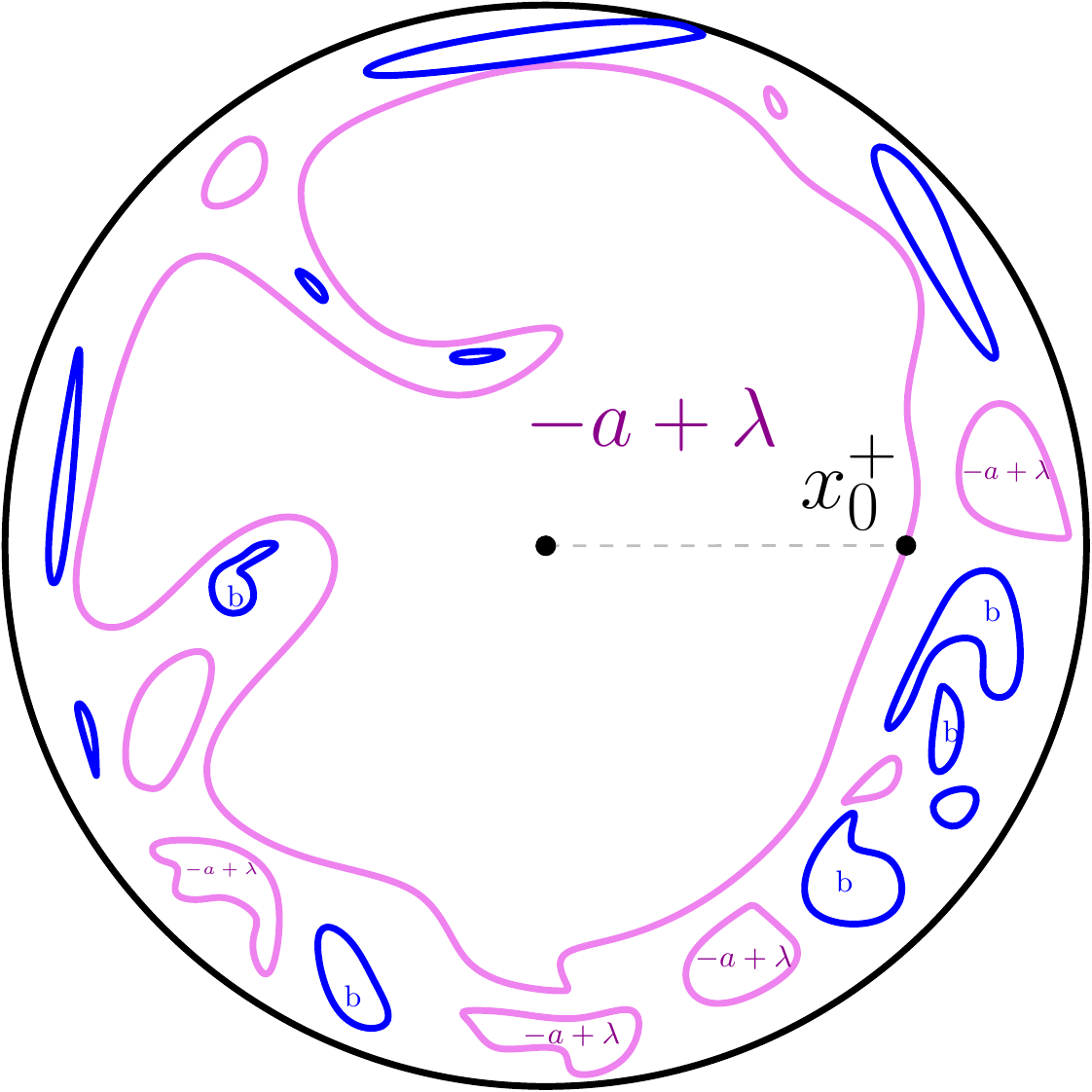}
	\includegraphics[width=0.3\textwidth]{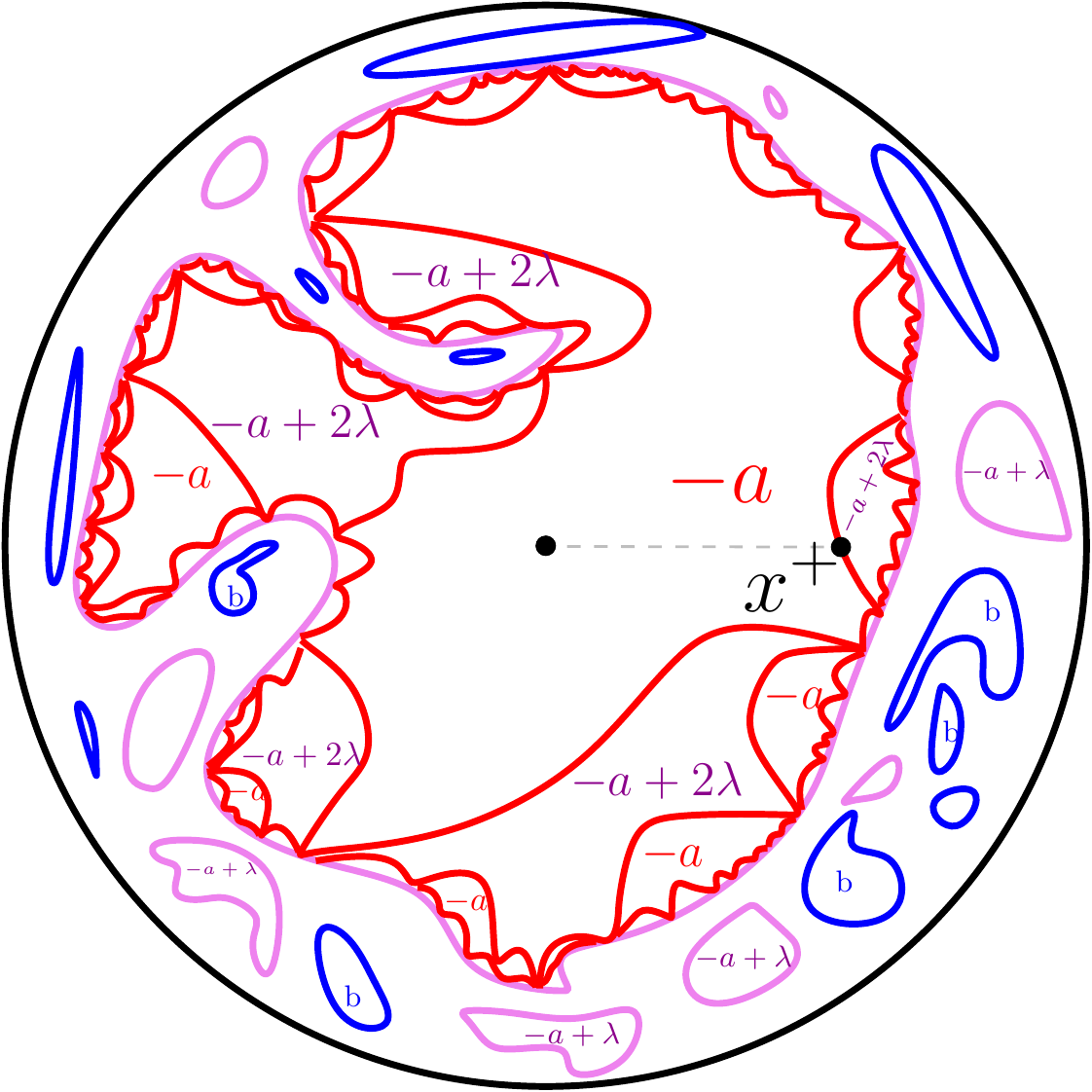}
	\includegraphics[width=0.3\textwidth]{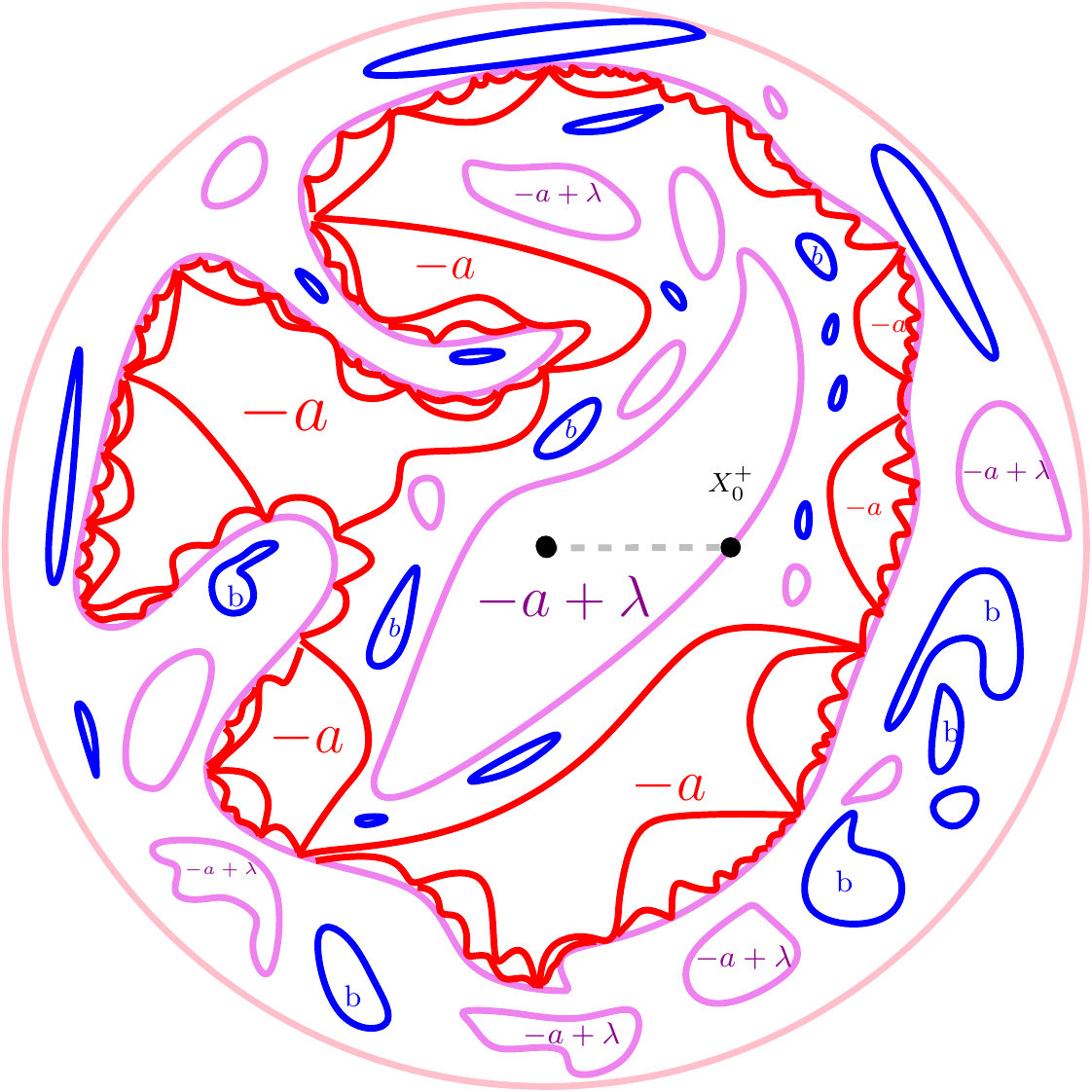}
	\caption{On the left: the first step in the construction of $A$. In the middle: we found the loop with label $-a$ around $0$ in the second step, and we stopped. On the right: we did not finish in the second step, and our iteration brought us back to the same situation (w.r.t the point $0$) as we had in the first step.}
	\label{fig::dependence}
\end{figure}
	
Observe that this iteration finishes a.s. in finite time: on each step of the iteration we finish the construction of the set $A$ with probability $1/2$ independently of the previous steps. Moreover, note that because the label of $\hat \ell$ is $-a$ by assumption, the last step in finishing the construction of $A$ will be exploring $\Aa_{-\lambda, \lambda}$ inside some loop $\ell_0$ labelled $-a+\lambda$ surrounding $0$. 

Now, denote by $x^+$ the smallest intersection point of $\R^+$ and the loop $\hat \ell$. We claim that $x^+ \notin \ell_0$. Indeed, let further $x_0^+$ be the smallest intersection point of $\R^+$ and the loop $\ell_0$. Observe that $x^+ \in \ell_0$ iff $x^+ = x^+_0$. But we know from the construction of $\Aa_{-\lambda,\lambda}$ (as every loop is just a union of finitely many level line segments), that the loop surrounding $0$ of $\Aa_{-\lambda,\lambda}$ does not hit any fixed boundary point. In particular it cannot hit $x_0^+$ almost surely, and thus $x^+ \notin \ell_0$ almost surely. Hence we know that $x^+$ remains in the interior of the loop $\ell_0$. Thus the value of $h_A$ on the other side of $\hat \ell$ is equal to $-a+2\lambda$ in a neighbourhood around $x^+$ (see Figure \ref{fig::dependence} middle part). This implies that $\alpha_n(h_A,\hat \ell) \to \-a+2\lambda$.

If on the other hand $a<2\lambda$, we can first sample $\tilde A:=\A{a}$. If the loop surrounding $0$ is $\hat \ell$ we take $A = \tilde A$; otherwise, the law of $\Aa_{-a,b}$ inside $\tilde A$ is that of $\Aa_{-2\lambda,b-2\lambda+a}$ inside the connected component of $\tilde A$ surrounding $0$. Thus we are back in the case $a\geq 2\lambda$.

To prove the lemma, it finally remains to verify the Claim \ref{claim:tec}.
	\begin{proof}[Proof of Claim \ref{claim:tec}]
		First, note that the variance of $\alpha_n(\Gamma^A,\hat \ell)$ and the variance of $\alpha_n(\Gamma^{\Aa_{-a,b}},\hat \ell)$ are both smaller than the variance of $\alpha_n(\Gamma^{\hat \ell},\hat \ell)$ where $\Gamma^{\hat\ell}$ is a zero boundary GFF in the component of $\C \backslash \hat\ell$ containing $\infty$. Indeed, this just follows from the fact that for a local set $A$, $\Gamma^A$ is equal to the sum of independent zero boundary GFFs in each connected component of the complement of $A$, and the Green's function is monotonously increasing w.r.t. to the underlying domain.
But by definition of $\alpha_n(\Gamma^{\hat\ell},\hat \ell)$, the variance of $\alpha_n(\Gamma^{\hat \ell},\hat \ell)$ is equal to \[Leb(L(n^{-1},2^{-n}))^{-2}\int_{L(n^{-1},2^{-n})^2}G_{\D}(z,w)dzdw,\]
which by a direct calculation can be bounded by $2^{-cn}$ for some deterministic constant $c >0$. By Markov inequality and Borel-Cantelli this proves the claim.\end{proof}
\end{proof}

\section{Labelled CLE$_4$ and the approximate Lévy transforms}  \label{aprox CLE_4}

Let $B_t$ be a standard Brownian motion, $I_t$ its running infimum process and $L_t$ its local time. The Lévy theorem for the Brownian motion states that the pair $(B_t - I_t, I_t)$ has the same law as $(|B_t|, L_t)$. This result provides several identities in law for the exit time. For example, the symmetric exit time $\sigma_{-a,a}$ from the interval $[-a,a]$ has the same law as the first time $\tau_a$ such that the Brownian motion makes a positive excursion of size $a$ above its current infimum.

In this section, we discuss an analogue of such an identity for the GFF by discussing another coupling between the CLE$_4$ and the GFF, introduced in \cite{WaWu}. In this coupling, the associated harmonic function is constant inside each loop $\ell$, and it is given by $2\lambda - t_\ell$, where $t_\ell$ denotes a time-parameter of the construction.In a subsequent article \cite{ALS3}, we will connect this to the Lévy transform of the metric graph GFF introduced in \cite{LuW}.

This section is organised as follows: first, we introduce the labelled CLE$_4$ and state its coupling with the GFF. Then, we revisit this coupling using an approximation via a decreasing sequence of thin local sets $(\B_r)_{0<r<\lambda}$ that are constructed using TVS and have the property that for $0<r<\lambda$, $\B_r$ has the same law as $\Aa_{-2\lambda,2\lambda-r}$. After that, we relate the labels of loops of $\B_r$ with their distance to the boundary in the graph $G_p(\B_r)$. Finally, we show that a.s. as $r\to 0$, $\B_r\searrow \B_0$, where $(\Gamma,\B_0)$ is the coupling of the labelled CLE$_4$ with the GFF, first proved in \cite{WaWu}.

\subsection{Labelled CLE$_4$: Definition and coupling with the GFF}\label{intro label CLE}

Labelled CLE$_4$ in $\D$, introduced in \cite{WernerWu}, corresponds to a Markovian exploration of CLE$_4$ loops, that keeps track of the time when each loop is discovered. We will here give a brief idea of the construction on $\D$ and refer to \cite{WernerWu} or \cite{WaWu} for more details.

The exploration discovers the bubbles of SLE$_4$ pinned at the boundary of the domain discovered so far. More precisely, we consider the SLE$_4$ bubble measure $\mu$ pinned at the boundary point $1$ introduced in \cite{SW}. This is the unique infinite measure on simple loops $\ell$ in $\bar \D$ which satisfies the following properties:
\begin{itemize}
\item $\mu$ almost everywhere, all loops touch the boundary only at 1.
\item If we denote by $\mu_\eps$ the measure on loops with a radius bigger than $\epsilon>0$, then $\mu_\eps$ has finite total mass.
\item When we normalize $\mu_\eps$ to be a probability measure, it satisfies the following Markov property: when we explore the loop $\eta$ from point $1$ until the time $\tau$ when it first exits the disk centred at $1$ of radius $\eps$, then the remaining part is given a chordal SLE$_4$ from $\eta_\tau$ to $1$ in $\D \backslash \eta$.
\item The scaling is fixed by saying that the mass on loops surrounding $0$ is equal to $1$.
\end{itemize}
See Section 6 of \cite{SW} for the existence and uniqueness of this measure. It can be explicitly constructed by properly normalizing the measure on chordal SLE$_4$ curves from $e^{i\eps}$ to $1$ as $\eps \to 0$; the uniqueness follows as the Markov property for all $\mu_\eps$ characterizes the loop on measures surrounding any point $z$ up to arbitrarily small initial segments.
  	
Now define the measure $M = \mu \otimes \omega$, where $\omega$ is the harmonic measure on $\partial \D$ seen from $0$. A key property of this measure is its invariance under Mobius transformations, see Lemma 6 of \cite{WernerWu}. In particular, this means that $0$ is not a distinguished point w.r.t. $M$. 

The exploration process is then given by a Poisson point process (PPP) $E_t$ with the intensity $M$ times the Lebesgue measure in $\R_+$, i.e. a PPP of loops pinned uniformly over the boundary. Here, loops are naturally ordered and each loop $\ell$ comes with a time label $t_\ell$. To obtain CLE$_4$ we embed this exploration inside $\D$ iteratively. 

For example, to define the exploration of the loop surrounding $0$, we consider $\tau_0$, the first time that $E_t$ contains a loop containing $0$. We then look at all the loops $(\ell_t: t \leq \tau_0)$ larger in $\epsilon$ in diameter. As there are finitely many of them, $\ell_{t^\epsilon_1},..., \ell_{t^\epsilon_{n^\epsilon}}$ we can add them one by one: more precisely, we define $D_0 = \D$ and inductively set $D_i$ to be the component of $D_{i-1} \backslash \phi_i(\ell_{t_i^\epsilon})$ containing the origin, where $\phi_i$ is the conformal map $\D \to D_{i-1}$ fixing the origin. We can then define $L_0^\epsilon$ as the boundary of $D_{n^\epsilon-1} \backslash D_{n^\epsilon}$. It can be shown that as $\epsilon \to 0$, $L_0^\epsilon$ converges and that the resulting loop $L_0$ has the law of the CLE$_4$ loop around zero. 

The result of the last paragraph can be strengthened: one can show that, in fact, the whole exploration process converges. One can define the exploration of the loop around any other point $z \in \D$ by conformal invariance. By coupling the exploration processes for any two points $z, w$ together such that they agree until the first time they are separated, and evolve independently thereafter, we obtain the symmetric exploration of CLE$_4$. This process is target-independent and each loop comes with a label indicating its discovery time in the exploration. 

Theorem 1.2.1 of \cite{WaWu} shows a way to couple labelled CLE$_4$ with the GFF. Let us rephrase it in our framework.
\begin{thm}[Theorem 1.2.2 of \cite{WaWu}]\label{thmwawu}
There exist a coupling $(\Gamma, (\ell, t_{\ell})_{\ell \in I})$ between a GFF and a labelled CLE$_4$, such that $\tilde B=\overline{\bigcup_\ell \ell}$ is a thin local set of $\Gamma$ whose harmonic function is constant inside every loop $\ell$ with label given by $2\lambda-t_\ell$.
\end{thm}

From the description above, it follows that if we have explored the loops $\ell_t$ up to any fixed time $T$, then the remaining loops can be sampled by considering an independent labelled CLE$_4$ inside each connected component of the complement of $\overline{\cup_{t \leq T} \ell_t}$. This should remind the reader of the local sets and thus in order to reinterpet the coupling of the GFF and the labelled CLE$_4$, we introduce the following local set: 

\begin{defn}[Lévy transform of $\Aa_{-2\lambda, 2\lambda}$]\label{defB0} Denote by $\B_0$ the thin local set of $\Gamma$ that has the following properties
	\begin{enumerate}
		\item $h_{\B_0}$ is constant inside each loop of $\B_0$ with labels in $\{2\lambda - t: t \geq 0\}$.
		\item  For all dyadic $s\geq 0$, the closed union of the loops of $\B_0$ with label greater or equal $2\lambda-s$, $\B_0^s$, is a BTLS and such that $h_{\B_0^s}\in \{-s\}\cup \{2\lambda-t: 0\leq t\leq s \}$.
	\end{enumerate} 
\end{defn}

This definition requires three clarifications.
\begin{rem}
The usage of the definite article in the definition above is justified by Proposition \ref{prop:labelCLE} below that proves the uniqueness of $\B_0$.
\end{rem}

\begin{rem}\label{remens} 
In fact one can see that the second condition holds for any fixed $s \in \R_+$: indeed, for a $s \in \R_+$ take dyadics $s_n$ such that $s_n \downarrow s$. Then $C^{s_n} \supset C^{s_{n+1}}$ are nested decreasing local sets and thus a.s. converge to a local set $\tilde C^s$. One can see that if we define the set $C^s$ as above, it is also given by the same limit, thus $\tilde C^s = C^s$. As the inclusion property is also clear (w.r.t dyadics), it remains to just show that the local set has the given values for the harmonic function. But this is clear for rational $z$ inside the components of the complement that have $h_{\B_r}(z)\geq 2\lambda - s$, as their value never changes. And for any other rational $z$ that is in a component with $h_{\B_r}(z) < 2\lambda - s$, there is some $s_{n_0}$ such that $h_{\B_r}(z) < 2\lambda - s_{n_0}$ and thus $h_{C^{s_n}} = -s_n$ for all $n \geq n_0$ and thus converges a.s. to $-s$.
\end{rem}

\begin{rem}
It will be shown in \cite{ALS3} that $\B_0$ corresponds to the image under the Lévy transform of $\Aa_{-2\lambda,2\lambda}$. This implies that the labels $t_\ell$ may be interpreted as sort of ``local times''. However, a key step is still missing in completing the Lévy transform picture. Namely, although we prove that the labels are a measurable function of the underlying GFF, we do not prove that these labels are measurable w.r.t. to the underlying set. This measurability problem was first presented in Section 5.1 of \cite{WaWu}, and we hope that the results presented here make a step forward in its proof.
\end{rem}

Theorem \ref{thmwawu} implies the existence of such a local set. One of the key results of this section is to prove the uniqueness of such a coupling: 

\begin{prop}\label{prop:labelCLE}
Almost surely for each $\Gamma$ there exists a unique local set $\B_0$ satisfying Definition \ref{defB0}. Moreover, if the loops of $\B_0$ are indexed by $I$ and the labels of each loop of $\B_0$ are written as $h_\ell = 2\lambda - t_\ell$, then $(\ell,t_\ell)_{\ell \in I}$ has the law of labelled CLE$_4$. 
\end{prop}

In order to prove this proposition, we will approximate $\B_0$ by certain local sets that already appeared in \cite{WaWu}, but whose handling is considerably simplified in our setting.

\subsection{Approximate L\'evy transform}\label{Br}

A natural way to approximate the first time $\tau_a$ when a standard Brownian motion makes a positive excursion of size $a$ off its minimum is as follows:
\begin{itemize}
\item Let $\tau^{r,1} = \sigma_{a-r,-r}$ be the first time that the BM exits the interval $[-r,a-r]$.
\item If $B_{\tau^{r,1}} = a-r$, set $\tau^r_a = \tau^{r,1}$ and otherwise define iteratively $\tau^{r,2} = \sigma_{a-2r,-2r}$ for $B_t, t\geq \tau^{r,1}$. 
\item Iteratively, if the BM exits from $a-2r$, set $\tau^r_a = \tau^{r,2}$ and else continue. 
\end{itemize}

Note that the stopping time $\tau^{r}_a$ comes with an interesting decoration: the end-value $B_{\tau^{r}_a} = a - kr$. One can think of $kr$ as an approximation of the running infimum.

It is also interesting to observe that $\tau^r_a$ has the law of $\sigma_{-a,a-r}$, the fist exit time by a BM from the interval $[-a,a-r]$. This just follows by reflecting the BM at times $\tau^{r,n}$. Moreover, the $k$ appearing before corresponds exactly to the number of crossings of this new BM of the interval $[0,-r]$ before the stopping time $\sigma_{-a,a-r}$, i.e. to a (scaled) approximation of its local time at zero.
In other words, we have a way of approximating the stopping time $\sigma_{-a,a}$ together with the local time at zero. Our aim in this section is to do the same for the 2D GFF:

\begin{defn}[Approximate Lévy transform of $\Aa_{-2\lambda,2\lambda}$]\label{def:br} For any $r\in (0,2\lambda)$, we denote by $\B_r$ the thin local set of $\Gamma$ that has the following properties:
	\begin{enumerate}
		\item $h_{\B_r} \in \{2\lambda - kr: k \in \N\}$.
		\item  For all $j\geq 0$, the closed union of the loops of $\B_r$ with label greater or equal $2\lambda-jr$, $\B_r^j$, is a BTLS  such that $h_{\B_r^j}\in \{-jr\}\cup \{2\lambda-kr: k\in \N,k\leq j \}$.
			\end{enumerate} 
\end{defn}

The reason for using the definite article in the definition comes again from a uniqueness claim, incorporated in the main result of the next few subsections.
\begin{prop}\label{prop:: existence and uniqueness Br}
The thin local sets from Definition \ref{def:br} exist and have the same law as the sets $\Aa_{-2\lambda, 2\lambda - r}$. They are unique, measurable w.r.t the GFF and monotone in the sense that $\B_r \subset \B_{r/2}$.	
\end{prop}

Let us already point out that, as the notations suggest, we will eventually (in Section \ref{B0}) take $r \to 0$, show that $\B_r \to \B_0$ and moreover use the uniqueness of $\B_r$ to prove that of $\B_0$.

\begin{rem}
The sets $\B_r$ appear as $\Upsilon(r)$ in Section 3.5 of \cite{WaWu}. They prove the measurability of these sets, but do not characterise these sets. Moreover, their construction is somewhat lengthy compared to ours.
\end{rem}	

\subsubsection{Construction}\label{construction}
For a GFF $\Gamma$ the construction goes iteratively: Define $\B^1_r:=\A{r}$. Inside the connected components $O$ of $D\backslash  \B_r^1$ labelled $-r$, explore $\A{r}(\Gamma^{\B^1_r},O)$. Define $\B^2_r$ as the closed union of the sets explored. Note that the loops of $\B^2_r$ have labels in $\{2\lambda-r, 2\lambda - 2r,-2r\}$. 

We proceed recursively: suppose we have constructed $\B^{j}_r$. Now, in the connected components $O$ of $\B^j_r$ labelled $-jr$ explore $\A{r}(\Gamma^{\B^{j}_r},O)$. Define $\B^{j}_r$ as the closed union between $\B^j_r$ and the sets explored. Note that it is easy to see from the construction that $\B_r^j \subset \B_r^{j+1}$. The limit of these iterations $\overline{ \bigcup \B_r^j}$ gives the desired set: it satisfies both conditions of the definition by construction. It is also nice to note that the law of the label of the loop surrounding any fixed point $z$ is that of $2\lambda-\ell_zr$, where $\ell_z\sim Geom(r/2\lambda)$.

\begin{figure}[ht!]    
	\centering
	\includegraphics[scale=0.5]{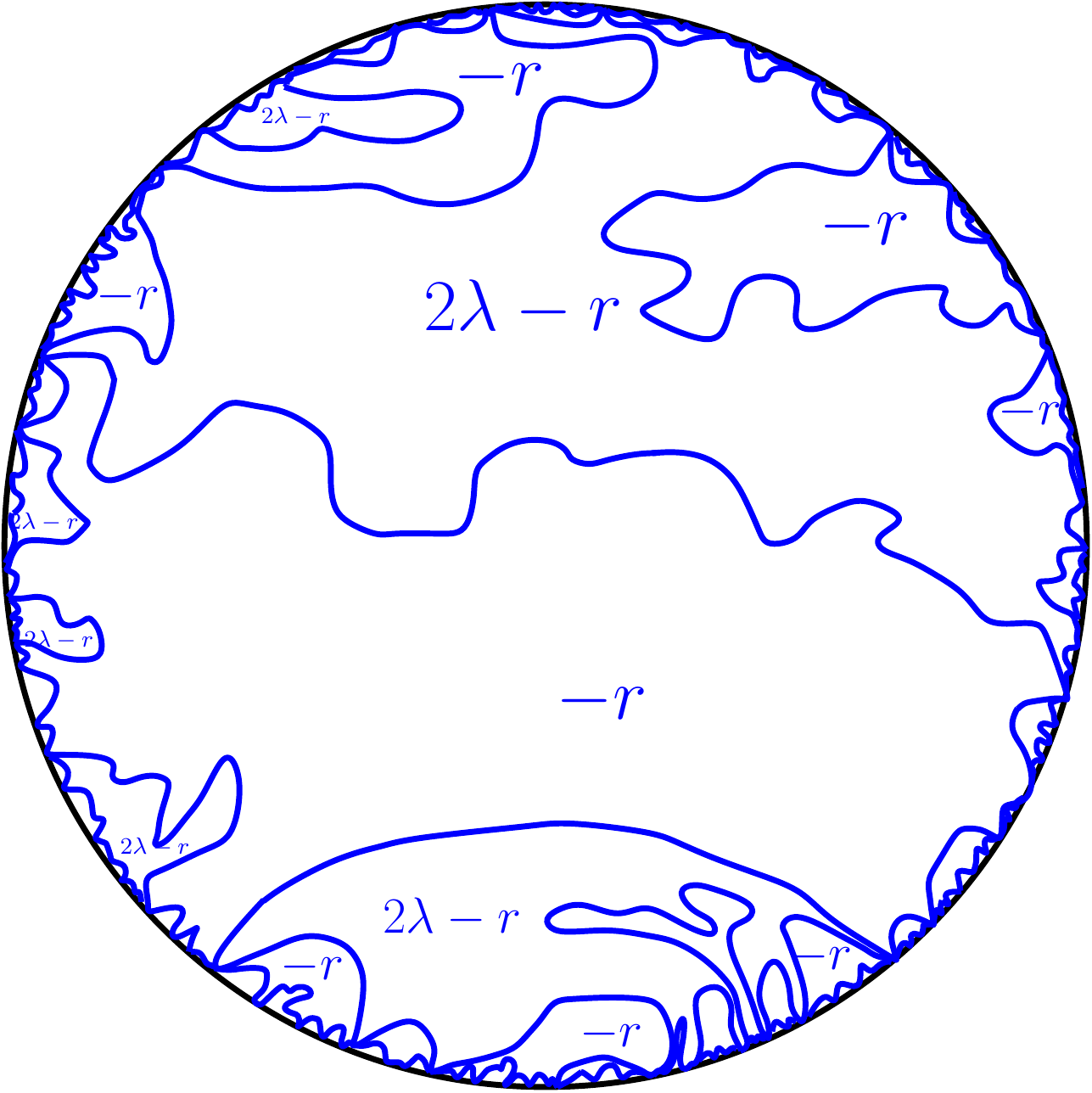}
	\includegraphics[scale=0.5]{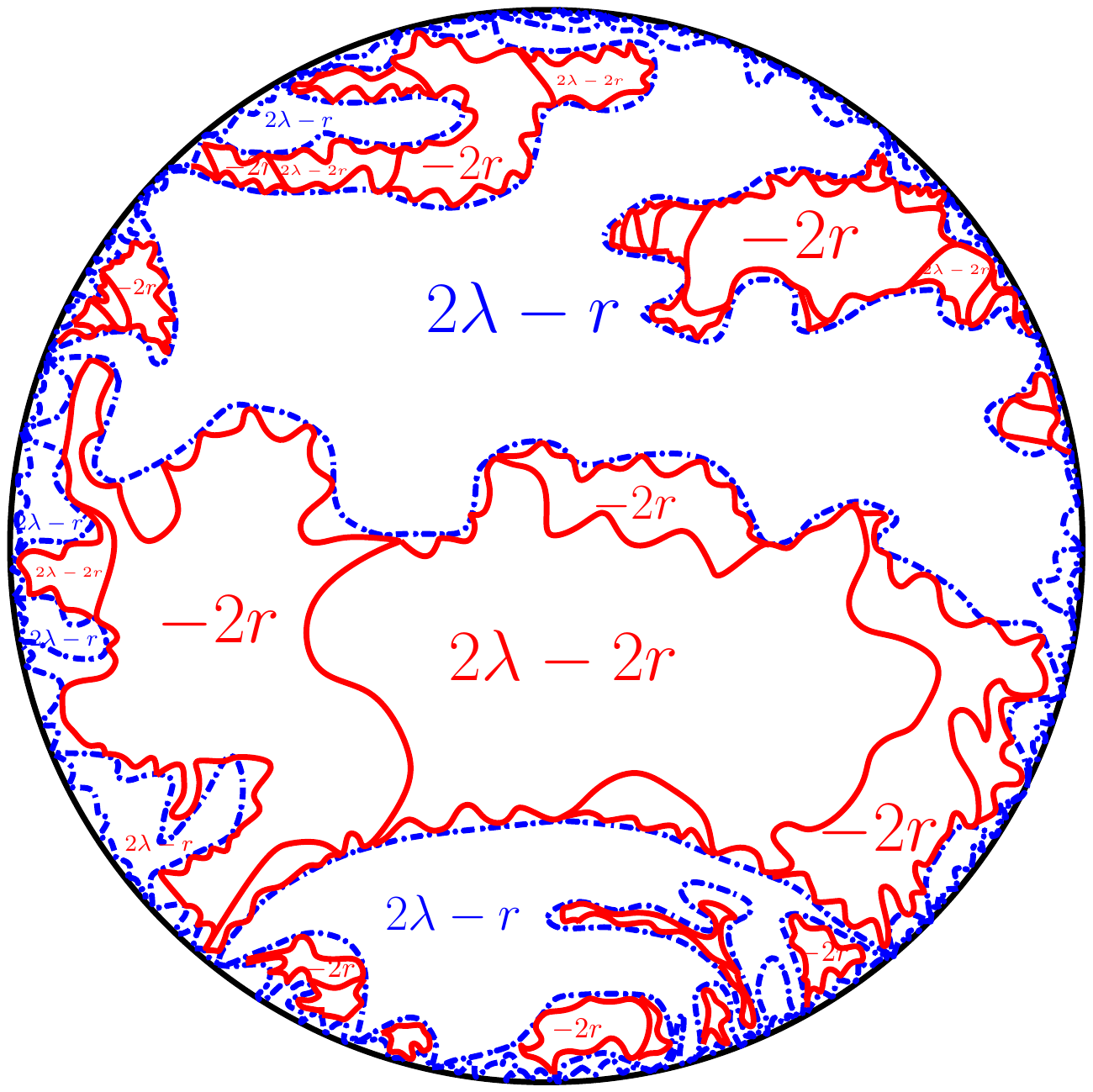}
	\caption {On the left figure we have $\B^1_r$ with its labels. On the right hand one we have iterated $\A{r}$ inside each loop labelled $-r$ and we have obtained $\B^2_r$.}
	\label{ALE const}
\end{figure}

\subsubsection{The set $\B_r$ has the same law as $\Aa_{-2\lambda, 2\lambda-r}$.} \label{same dis}Note that this property implies that the Minkowski dimension of $\B_r\cap K$ is smaller than 2, thus it is thin.

The proof is based on the following observation: $\A{r}$ has the same distribution as $\Aa_{-2\lambda+r,r}$. Indeed, $\B^2_r$ is constructed by exploring $\A{r}(\Gamma^{\A{r}},O)$ inside all connected components $O$ of $D\backslash \A{r}$ labelled $-r$. But, if instead of that we explore $\Aa_{-2\lambda+r,r}(\Gamma^{\A{r}},O)=\A{r}(-\Gamma^{\A{r}},O)$, we obtain a BTLS with the same law but whose labels takes values in $\{2\lambda - r, 0, -2\lambda\}$. We call the explored set $A^2$. Note that the loops of $A^2$ labelled $0$ correspond to the loops of $\B_r^2$ labelled $-2r$, and loops of $A^2$ labelled $-2\lambda$ to those of $\B_r^2$ labelled $2\lambda -2r$. 

We can now iterate the construction: to define $A^k$ we explore $\A{r}((-1)^{k-1}\Gamma^{A^{k-1}},O)$ in all connected components $O$ of $D\backslash A^k$ labelled either $0$ or $-r$. Inductively, the BTLS $A^{k}$ defined this way has the same distribution as $\B_r^k$. Additionally, all the loops of $A^{k}$ have labels in  $\{2\lambda -r, 0, -2\lambda\}$, and those labelled $0$  correspond to loops of $\B_r^{k}$ labelled $-kr$. We know that the sets $\B_r^k$ converge to $\B_r$ as $k \to \infty$. As any $z \in D \cap \Q$ will almost surely be contained in a connected component of $A^k$ with the label $0$ for only finitely many steps, Lemma \ref{BPLS} implies that $A^k$ converges to $\Aa_{-2\lambda, 2\lambda-r}$ and the claim follows.

\subsubsection{Uniqueness}\label{uniqueness B_r}
Suppose that $\tilde \B_r$ also satisfies the conditions of Definition \ref{def:br}  and is sampled conditionally independently of $\B_r$ given the GFF $\Gamma$. We aim to show that a.s. $\tilde \B_r = \B_r$. It is enough to prove that $\tilde \B_r^j$ from the condition (2) of the Definition \ref{def:br} is a.s. $\B_r^j$ as defined in Section \ref{construction}. 

By the uniqueness of $\A{r}$ (Lemma \ref{cledesc2}), $\B_r^1 = \tilde \B_r^1$. If we consider $\tilde \B_r^2$, then by definition the $2\lambda - r$ components are the same as in $\tilde \B_r^1$ and thus as in $\B_r^1$. Now, consider a component $O$ of $\tilde \B_r^1$ with value $-r$. We have that $\tilde \B_r^2\cap O$ is a thin local set with values in $\{-r, 2\lambda -r\}$. But these are again unique by the uniqueness of $\A{r}$. As this is also the case in the construction of $\B_r^2$, we see that the connected components with values $-2r, 2\lambda - 2r$ of $\B_r^2$ and $\tilde \B_r^2$ also agree, and thus $\B_r^2 = \tilde \B_r^2$ almost surely. Iterating the same argument gives that $\tilde \B_r^j$ and $\B_r^j$ agree.  Finally, as both $\B_r$ and $\tilde \B_r$ are thin, Lemma \ref{thin empty} implies that  $\B_r=\overline{\bigcup_j \B_r^j}=\overline{\bigcup_j \tilde \B_r^j} = \tilde \B_r$. 

\subsubsection{Monotonicity}\label{monBr}

We claim that $\B_r \subset \B_{r/2}$. 

First, note that all loops of $\B_{r/2}^2$ have labels in $\{-r, 2\lambda - r, 2\lambda - r/2\}$. Thus we can build $\Aa_{-r,2\lambda-r/2}$ by exploring $\Aa_{-2\lambda,r/2}(\Gamma^{\B_{r/2}^2},O)$ in all the connected components $O$ of $D\backslash \B_{r/2}^2$ labelled $2\lambda-r$. By Lemma \ref{keepingloops} (2), no new loop labelled $-r$ touches the boundary of the original domain. But $\B_r^1$ is equal to $\Aa_{-r,2\lambda-r}$, and we know by Lemma \ref{keepingloops} that all loops labelled $-r$ of $\Aa_{-r,2\lambda-r}$ are also loops of $\Aa_{-r,2\lambda}$ and moreover touch the boundary. Corollary \ref{recons} now implies that $\B_r^1\subseteq \B_{r/2}^2$. Now, observe that all loops labelled $-r$ of $\B_r^1$ are loops with $-r$ in $\B_{r/2}^2$. Thus using the iterative nature of the construction of $\B_r^j$ and induction on $j$, it similarly follows that $\B_r^j \subset \B_{r/2}^{2j}$ and thus the claim follows.

\subsection{Connection with two-valued sets and the distances in $G_p(\Aa_{-2\lambda, 2\lambda -r})$}\label{distances} 

We will now see how to interpret the values of $h_{\B_r}$ as distances to the boundary. More precisely we have the following result:

\begin{prop}\label{prop:lint}
Let $\Gamma$ be a GFF in $\D$. Then, then the harmonic extension $h_{\B_r}$ inside any loop $\ell$ of $\B_r$ is equal to $2\lambda-rd_p(\partial \D,\ell)$, where, $d_p(\partial \D,\ell)$ is one plus the minimum distance, in the graph $G_p(\B_r)$, of $\ell$ to a loop that touches the boundary. 
\end{prop}

As $\B_r$ has the law of $\Aa_{-2\lambda, 2\lambda -r}$ and thus we can interpret this result by saying that the labels of $\B_r$ encode the $G_p$ distances of the loops of $\Aa_{-2\lambda, 2\lambda -r}$ to the boundary. Proposition \ref{prop:lint} follows from the following lemma:

\begin{lemma}\label{Distance Br boundary}
	Let $\Gamma$ be a GFF in $\D$, and consider the local set $\B_r$ of Definition \ref{def:br}. Then, almost surely, all loops of $\B_r$ labelled $2\lambda-r$ touch the boundary and all loops labelled $2\lambda - 2r$ do not touch the boundary. 
\end{lemma}

Indeed Proposition \ref{prop:lint} now follows by the iterative nature of the construction of $\B_r$: the loops with the label $2\lambda - nr$ of $\B_r$ have the same law as the loops with label $2\lambda - r$ of a copy of $\B_r$ inside the components of the complement of $\B_r^{n-1}$ labelled $-r(n-1)$. Moreover, they also have the same law as the loops with label $2\lambda -2r$ inside the components of the complement of $\B_r^{n-2}$ labelled $-r(n-2)$.

\begin{proof}[Proof of Lemma \ref{Distance Br boundary}]
	All loops labelled $2\lambda-r$ stem from loops of $\B_r^1=\A{r}$, thus the first claim is true. The second claim follows from the first paragraph in Section \ref{same dis}. Indeed, loops with the label $2\lambda - 2r$ stem from $\B_r^2$. In this paragraph it is shown that
	\begin{itemize}
	\item $\B_r^2$ has the same law as a bounded type local set $A$ with labels in $\{2\lambda-r,0, -2\lambda\}$, 
	\item the loops of $\B_r^2$ with the label $2\lambda - 2r$ correspond to the loops of $A$ with the label $-2\lambda$. 
	\end{itemize}
	But by Lemma \ref{keepingloops} (2) (as $A$ can be always completed to $\Aa_{-2\lambda,2\lambda-r}$ by exploring $\Aa_{-2\lambda, 2\lambda - r}$ in the loops with label $0$) the loops of $A$ with label $-2\lambda$  do not touch the boundary of the domain.
\end{proof}

\begin{rem}
Notice that the labels of the loops are a function of the distances in $G_p(\B_r)$ and thus in particular the labels are a measurable function of $\B_r$. Moreover, notice also that the proposition above should also give the right scaling for distances between any two loops: one would define $\bar d_r:=rd_p$ in $\Lo(\B_r)$ as $r$ times the minimum distance between the two loops. 
\end{rem}

\subsection{The coupling with labelled CLE$_4$: existence, uniqueness and labels}\label{B0}

In this subsection we prove the Proposition \ref{prop:labelCLE} and interpret the labels in the coupling of the GFF with the labelled CLE$_4$ as distances to the boundary:

\subsubsection{Existence}
The existence of this set is shown by taking a limit of $\B_r$ as $r \downarrow 0$. Indeed, recall that we have $\B_r \subset \B_{r/2}$. We define  $\B_0:=\overline{\bigcup \B_{d_n}}$ for $d_n = 2^{-n}$. Now, thanks \ref{same dis}, $\B_0$ has  the same law as a CLE$_4$, thus its Minkowski dimension is smaller than 2 and it is thin. Lemma \ref{BPLS} implies that $h_{\B_0}$ takes the right values. Thus (1), of Definition \ref{defB0} is satisfied.

Let us prove that the condition (2) is also satisfied. For any dyadic $s$ there is some $n_s$ such that for $n > n_s$ we have that $s / d_{n_s} \in \N$. By the properties of $\B_r$ we see that for all $n > n_s$, $\B_{d_n}$ satisfies the second condition of Definition \ref{defB0} for this $s$, and thus does also $\B_{0}$

\subsubsection{Uniqueness} We next consider the uniqueness of the local set coupling with the labelled $CLE_4$. The key lemma is the following:
\begin{lemma}\label{lemunB0}
	Let $\tilde \B_0$ be coupled with $\Gamma$ as in Definition \ref{defB0}. Then $\B_{d_n}$ is contained in $\tilde \B_0$ for all dyadic $d_n$. Additionally, if two loops of $\tilde \B_0$  labelled $\tilde l_1$ and $\tilde l_2$ are surrounded by the same loop of $\B_{d_n}$ labelled $l$, then both $\tilde l_1 \leq l+d_n$ and $\tilde l_2 \leq l + d_n$.
\end{lemma}
\begin{proof}
	It suffices to prove the first claim for a some fixed dyadic $d_n > 0$, that we denote by $r$ for simplicity. To show that $\B_{r}$ is contained in $\tilde \B_0$, we will consider the approximations $\tilde \B_0^{rk}$ of $\B_0$ as defined in Definition \ref{defB0}, and the approximations $\B_r^k$ of $\B_r$ as defined in Definition \ref{def:br}.
	
	First, let us see that $\B_{r}^1 = \A{r}$ is contained in $\tilde \B_0$. By condition (2) in the Definition \ref{defB0}, we see that $\tilde \B_0^r$ is a BTLS with $h_{\tilde \B_0^r}\in\{-r\}\cup\{2\lambda - t:0\leq t\leq r\}$. Now, in every connected component $O$ of the complement of $\tilde \B_0^r$, with label $2\lambda - t$ with $t$ as above, we can explore $\Aa_{-r+t-2\lambda, t}$ of the GFF $\Gamma^{\tilde \B_0^r}$ restricted to $O$ to obtain $\Aa_{-r,2\lambda}$. This means that we can construct $\Aa_{-r,2\lambda}$ starting from $\tilde \B_0^r$. By the uniqueness of $\Aa_{-r,2\lambda}$, it follows that $\tilde \B_0^r \subset \Aa_{-r,2\lambda}$. Notice that the loops of $\Aa_{-r,2\lambda}$ with the label $-r$ constructed inside such loops (i.e. inside the loops of $\tilde \B_0^r$ with value in $2\lambda - t$) cannot touch the boundary of the domain (exactly for the same reason as in Section \ref{monBr}). Hence, all the boundary touching loops of $\Aa_{-r,2\lambda}$ are in fact loops of $\tilde \B_0^r$ with the label $-r$. But by Corollary \ref{recons}, the union of the loops of $\Aa_{-r,2\lambda}$ with label $-r$ is equal to $\A{r}$. Thus, we deduce that $\B_{r}^1=\A{r} \subset \tilde \B_0^r$. 
	
	We saw that any connected component $O$ of $D\backslash \B_r^1$ with label $-r$ is also a connected component of $\tilde \B_0^r$ labelled $-r$. Now, consider any such component $O$. Then $(\B_r^2\backslash \B_r^1)\cap O$ is equal to $\B_r^1(\Gamma^{\B_r^1},O)$. Moreover, we claim that conditionally on $\tilde \B_0^r$, the set $\hat \B_0 :=( \tilde \B_0 \backslash \tilde \B_0^r)\cap O$ satisfies the conditions of Definition \ref{defB0} for the GFF $\Gamma^{\tilde B_0^r}$ restricted to $O$. Indeed, by Lemma \ref{BPLS} (1) it is a local set of $\Gamma^{\B_0^r}$ restricted to $O$ and satisfies conditions (1) and (2) of Definition \ref{defB0}. Furthermore, it is thin by Corollary 4.4 of \cite{Se}. Thus, by the previous paragraph $\B_r^1(\Gamma^{\B_r^1},O)\subseteq \hat \B_0^r$, which in turn implies that $\B_r^2 \subseteq \tilde \B_0^{2r}$. Iterating this way we see that $\B_r^k \subset \tilde \B_0^{kr}$ for all $k$, which by taking the limit in $k$ shows that $\B_r \subset \tilde \B_0$.
	
	Let us now prove the final statement of the lemma, setting again $r = d_n$. Denote by $\ell$ the loop of $\B_r$ labelled $l$ surrounding both $\tilde \ell_1$ and $\tilde \ell_2$. Take $k \in \N$ such that $kr=2\lambda-l$ and define $O$ to be the connected component of the complement of $\B_{r}^{k-1}$ whose closure contains $\ell$. Due to the fact that the label of $O$ is $-(k-1)r$, we deduce form the last paragraph that it is also a connected component of $\tilde \B_0^{(k-1)r}$ with the same label. Using the definition of $\tilde \B_0^{(k-1)r}$ and the fact that $\tilde \ell_1$ and $\tilde \ell_2$ are also contained in the closure of $O$, we obtain that $l_1,l_2\leq 2\lambda-(k-1)r=l+r$.
	 
\end{proof}

Let us now show how uniqueness follows from the above lemma. As $\tilde \B_0$ is thin and hence by Lemma \ref{thin empty} has empty iterior, it can be written as the closure of the union of its loops. Lemma \ref{lemunB0} shows that $\B_0:=\overline{\bigcup \B_{d_n}} \subset\tilde \B_0$. Thus, it is enough to prove that each loop of $\B_0$ surrounds no more than one loop of $\tilde \B_0$. 

Suppose for contradiction that with positive probability there is one loop $\ell$ of $\B_0$ labelled $l$ that surrounds two or more loops of $\tilde B_0$. Then, by the final part of above lemma it can only surround loops that have label smaller than $l$. In particular $|h_{\B_0}-2\lambda|\geq |h_{\tilde \B_0\cup \B_0}-2\lambda|$ and we can use Lemma 9 of \cite{ASW} to conclude that $\tilde \B_0\subseteq \B_0$. \footnote{In fact, Lemma 9 of \cite{ASW} asks for the sets to be BTLS. However, one can check that it is just enough to have thin local set whose harmonic functions are upper bounded.}

\subsubsection{$\B_0$ has the law of labelled CLE$_4$} 
In this section we will explain why $\B_0$ has the law of labelled CLE$_4$. In fact, this just follows from the uniqueness statement of Proposition \ref{prop:labelCLE} and Theorem \ref{thmwawu} as one can verify that any local set coupling with the labelled CLE$_4$ has to satisfy Definition \ref{defB0}. 

For self-containedness, we will sketch how to prove it by hand. Note that we have also already proven that $\B_0$ as a random set has the law of a CLE$_4$ carpet. So it just remains to argue that the label of a loop correspond to the time when it appears in the Poisson point process described in Section \ref{intro label CLE}. The details for each step of this sketch can be found in Section 3 of \cite{WaWu}:
\begin{itemize}
\item The Renewal / Markov property, inherent to a PPP, comes from the local set property of $\B_0^s$ in Definition \ref{defB0}: inside each component of the complement of $\B_0^s$ with the label $-s$, one explores an independent copy of $\B_0$. 
\item From the construction it follows that all the labels $2\lambda - t_\ell$ of the set $\B_0$ are different and thus the loops can be considered as a simple $\sigma-$discrete point process indexed by $t_\ell$.
\item It is known that such a renewal point process is a Poisson point process and is uniquely defined by its characteristic measure (see for example Theorem 3.1 of \cite{Ito}); and that the characteristic measure can be recovered by looking at the law of any instance of the point process.
\item It thus remains to verify that a loop with label $t_\ell$ can be seen as an instance from the symmetric measure $M$ on pinned SLE$_4$ loops described in the beginning of Section \ref{intro label CLE}. This is the part that requires some work, but it can be argued using the approximation by $\B_r$. Roughly, a loop of $\B_r$ that first appears in $\B_r^n$, is given by a SLE$_4(\rho)$ loop in a connected component of the complement of $\B_r^{n-1}$. These loops touch the boundary, and as $r \to 0$, these loops converge to the pinned loop measure $M$. The proof uses the fact that $\rho \to 0$ as $r \to 0$ and the characterization of $M$ via the Markov property.
\end{itemize}

\subsubsection{CLE$_4$ labels as distances to the boundary} 

Finally, notice that Proposition \ref{prop:lint} gives us also a way to interpet the labels of $\B_0$ as distances to the boundary: $\B_r$ converges to $\B_0$ and $h_{\B_r} \to h_{\B_0}$. As the label of a loop $\ell$ of $\B_r$ is given by $2\lambda - rd_\ell$, where $d_\ell$ is the $G_p$ distance to the boundary of the loop $\ell$, one can see the labels of $\B_0$ as rescaled distances to the boundary. In fact, in \cite{ALS3} we will see that they correspond exactly to the local time distances to the boundary, defined in metric graph GFFs in \cite{LuW}.

In a similar spirit, one should be able to define a conformally invariant distance between any two loops of CLE$_4$ by taking a limit as $r \to 0$ of $\bar d_r:=rd_p$, an thereby an interesting metric space. One can obtain tightness for the distance between any pair of loops surrounding a countable dense set $z_n \in \D$, and thus this seems to be the right scaling. Indeed, this follows from considering the slightly easier distance structure, corresponding to the distances in the ``geodesic tree'' from the boundary: one can define a distance $\tilde{d}(\ell_1,\ell_2)=t_1+t_2-2T$ where $T$ is the biggest $t$ such that $\B_0^t$ has a loop that surrounds both $\ell_1$ and $\ell_2$. However, as for now we run short of proving any convergence of actual interest.

Let us mention that a different approach to a conformally invariant distance between loops of CLE$_4$, that should correspond to the one described just above, has been announced in \cite{WaWu} Section 1.3.

\subsection* {Acknowledgements}  We would like to thank W. Werner for several interesting discussions on ALEs and for pointing out a simpler way to derive Proposition \ref{Perco}. We would also like to thank T. Lupu for fun discussions, F. Viklund for useful comments on an earlier draft, J. Miller for pointing out the question on the SLE$_4$ fan and the anonymous referee for the careful reading of our manuscript and all the helpful comments. A. Sepúlveda is supported by the ERC grant LiKo 676999, Juhan Aru is supported by the SNF grant 175505. Both authors are also thankful to the SNF grant 155922 and the NCCR Swissmap initiative.

\bibliographystyle{alpha} 
\bibliography{biblio}

\end{document}